\documentclass{amsart}
\usepackage{subfiles}
\usepackage[colorlinks]{hyperref}

\usepackage{enumerate}
\usepackage{amssymb,euscript,amsmath, mathrsfs, tensor}
\usepackage{stmaryrd}
\usepackage[dvips]{graphicx}
\usepackage[dvips]{color}
\usepackage[all,cmtip]{xy}
\usepackage{xcolor}
\usepackage[margin=1.2in]{geometry}
\usepackage{tikz}
\usetikzlibrary{cd}

\newcounter{ENUM}

\def\Ker{\mathrm{Ker}}
\def\im{\mathrm{im}}
\def\To#1{\buildrel\hbox{\tiny{$#1$}}\over\longrightarrow}

\title{Local langlands in families:\\ The banal case}
\author{Jean-Fran\c cois Dat}
\address{Jean-Fran\c cois Dat, Institut de Math\'ematiques de Juss\'ieu, 4 Place Jussieu, 75252, Paris.}
\email{jean-francois.dat@imj-prg.fr }
\author{David Helm}
\address{David Helm, Department of Mathematics, Imperial College, London, SW7 2AZ, United Kingdom.}
\email{d.helm@imperial.ac.uk }
\author{Robert Kurinczuk}
\address{Robert Kurinczuk, School of Mathematics and Statistics, University of Sheffield, Sheffield, S3 7RH, United Kingdom.}
\email{robkurinczuk@gmail.com}
\author{Gilbert Moss}
\address{Gil Moss, Department of Mathematics, The University of Utah, Salt Lake City, UT 84112, USA.}
\email{gil.moss@gmail.com}

\def\univ{\mathrm{uni}}
\def\hG{\widehat{\G}}
\def\CC{\mathbb{C}}
\def\QQ{\mathbb{Q}}
\def\ZZ{\mathbb{Z}}
\def\cO{\mathscr{O}}

\def\cP{\mathscr{P}}
\def\SO{\mathrm{SO}}

\def\O{\mathrm{O}}
\def\U{\mathrm{U}}
\def\CP{\mathcal{P}}
\def\cK{\mathcal{K}}

\def\Spec{\mathrm{Spec}}
\def\Spm{\mathrm{Spm}}

\def\Gal{\mathrm{Gal}}
\def\Hom{\mathrm{Hom}}
\def\End{\mathrm{End}}
\def\Res{\mathrm{Res}}
\def\Fr{\mathrm{Fr}}
\def\Rep{\mathrm{Rep}}
\def\ind{\mathrm{ind}}
\def\LLIF{\mathrm{LLIF}}

\def\LL{\mathrm{LL}}
\def\rad{\mathrm{rad}}
\def\ab{\mathrm{ab}}
\def\SL{\mathrm{SL}}

\def\Frac{\mathrm{Frac}}

\def\LG{{\tensor*[^L]\G{}}}

\def\LM{{\tensor*[^L]{\mathrm{M}}{}}}

\def\Ql{\overline{\mathbb{Q}_\ell}}
\def\Qbar{\overline{\mathbb{Q}}}
\def\Zbar{\overline{\mathbb{Z}}}
\def\ur{\mathrm{ur}}

\def\St{\mathrm{St}}
\def\ad{\mathrm{ad}}

\def\I{\mathrm{I}}

\def\V{\mathrm{V}}

\def\A{\mathrm{A}}
\def\B{\mathrm{B}}
\def\C{\mathrm{C}}

\def\F{\mathrm{F}}
\def\E{\mathrm{E}}
\def\G{\mathrm{G}}
\def\H{\mathrm{H}}
\def\I{\mathrm{I}}
\def\J{\mathrm{J}}
\def\K{\mathrm{K}}
\def\L{\mathrm{L}}
\def\M{\mathrm{M}}
\def\O{\mathrm{O}}
\def\P{\mathrm{P}}
\def\R{\mathrm{R}}
\def\S{\mathrm{S}}
\def\T{\mathrm{T}}
\def\U{\mathrm{U}}
\def\Q{\mathrm{Q}}

\def\W{\mathrm{W}}
\def\ad{\mathrm{ad}}
\def\Z{\mathrm{Z}}

\def\Ind{\mathrm{Ind}}
\def\LL{\mathrm{LL}}
\def\Qbar{\overline{\mathbb{Q}}}

\newcommand{\margh}[1]{}
\theoremstyle{definition}

\newtheorem{theorem}{Theorem}[section]
\newtheorem{proposition}[theorem]{Proposition}
\newtheorem{lemma}[theorem]{Lemma}
\newtheorem{corollary}[theorem]{Corollary}
\newtheorem{conjecture}[theorem]{Conjecture}
\newtheorem{definition}[theorem]{Definition}

\newtheorem{remark}[theorem]{Remark}
\numberwithin{equation}{section}

\newtheorem*{propositionpara}{Proposition \ref{parabolicinduction:theorem}}
\newtheorem*{propositionauto}{Proposition \ref{automorphisms:theorem}}
\newtheorem*{propositionint}{Proposition \ref{integralclass}}
\newtheorem*{propositionintchar}{Proposition \ref{propintegralparameters}}
\newtheorem*{propositionext}{Proposition \ref{extendedpacketclass}}

\newtheorem*{theoremBD}{Theorem \ref{banaldecomptheorem}}
\newtheorem*{propositionlocalcusplift}{Proposition \ref{localpropositionlifting}}
\newtheorem*{theoremALLIF}{Theorem \ref{axiomaticLLIF}}

\def\Qbar{{\overline{\mathbb{Q}}}}
\def\Fl{\overline{\mathbb{F}_\ell}}
\def\Zl{\overline{\mathbb{Z}_\ell}}
\def\stab{\mathrm{stab}}
\def\GL{\mathrm{GL}}
\def\N{\mathrm{N}}
\def\Kbar{{\overline{\mathrm{K}}}}
\def\univ{\mathrm{univ}}

\def\univ{\mathrm{univ}}

\begin{document}

\maketitle

\begin{abstract}
We state a conjecture, \emph{local Langlands in families}, connecting the centre of the category of smooth representations on~$\mathbb{Z}[\sqrt{q}^{-1}]$-modules of a quasi-split~$p$-adic group~$\G$ (where~$q$ is the cardinality of the residue field of the underlying local field), the ring of global functions on the stack of Langlands parameters for~$\G$ over~$\mathbb{Z}[\sqrt{q}^{-1}]$, and the endomorphisms of a Gelfand--Graev representation for~$\G$.  
For a class of classical~$p$-adic groups (symplectic, unitary, or {odd special orthogonal groups}), we prove this conjecture after inverting an integer depending only on~$\G$. Along the way, we show that the local Langlands correspondence for classical~$p$-adic groups (1) preserves integrality of~$\ell$-adic representations; (2) satisfies an ``extended'' (generic) packet conjecture; (3) is compatible with parabolic induction up to semisimplification (generalizing a result of Moussaoui), hence induces a \emph{semisimple} local Langlands correspondence; and (4) the semisimple correspondence is compatible with automorphisms of~$\mathbb{C}$ fixing~$\sqrt{q}$.
\end{abstract}

\setcounter{tocdepth}{1}
\tableofcontents

\section{Introduction}
\subsection{LLIF for general linear groups}
Let~$\F$ be a~$p$-adic field and~$\ell\neq p$ prime.  In \cite{EH14}, the second author and Emerton conjectured an interpolation of the (generic) local Langlands correspondence for~$\GL_n(\F)$ across families of~$\ell$-adic representations.  The existence of this correspondence was reduced by the second author in \cite{HelmDuke,curtis} to showing that there is a (unique) isomorphism
\[\LLIF:(\mathfrak{R}_{\GL_n,\Zl})^{\GL_n}\rightarrow \mathrm{End}_{\Zl[\GL_n(\F)]}(\ind_{\U_n(\F)}^{\GL_n(\F)}(\psi)),\]
compatible with the local Langlands correspondence, where~$(\mathfrak{R}_{\GL_n,\Zl})^{\GL_n}$ is the ring of global functions on the stack of Langlands parameters for~$\GL_n$ over~$\Zl$, and~$\ind_{\U_n(\F)}^{\GL_n(\F)}(\psi)$ is a Gelfand--Graev representation.  This statement was finally proven using an inductive argument in a pair of papers: by the second author, and the second and fourth authors \cite{curtis, HM}, thus establishing ``local Langlands in families'' for~$\GL_n$.  This paper is the second in a series of articles devoted to generalizing this picture to a general connected reductive~$p$-adic group.  

\subsection{Reductive groups and LLC}
Let~$\G$ be the~$\F$-points of a connected reductive group~$\mathbf{G}$ defined over~$\F$.  Let~$\widehat{\G}$ denote the dual group of~$\mathbf{G}$ considered as a~$\mathbb{Z}[1/p]$-group scheme.  The~$\F$-rational structure on~$\mathbf{G}$ induces a finite action of the Weil group~$\W_\F$ of~$\F$ on~$\widehat{\G}$ and we set~$\LG=\widehat{\G}\rtimes \W_\F$, the \emph{Langlands dual group} of~$\G$.

The local Langlands correspondence for~$\G$ is expected to take the following form: there is a unique (finite-one) map~$\mathcal{LL}_{\G}$ from the set~$\Pi_{\mathbb{C}}(\G)$ of isomorphism classes of irreducible smooth representations of~$\G$ to the set~$\Phi_{\SL_2,\mathbb{C}}(\G)$ of~$\widehat{\G}(\mathbb{C})$-conjugacy classes of ``$\SL_2$-parameters'' $\rho:\W_\F\times \SL_2(\mathbb{C})\rightarrow \LG(\mathbb{C})$ for~$\G$ satisfying a list of natural desiderata (see Section \ref{LLcorresp} for a precise statement).  

It is known for~$\GL_n(\F)$ and classical~$p$-adic groups (see Section \ref{LLforclassgroups} for a list of references).  In this work, we take as an input the local Langlands correspondence for~$\G$ together with some of its expected properties.  We will verify these properties are satisfied for ``classical~$p$-adic groups'' (which, for us, means symplectic, unitary, and odd special orthogonal groups), either by referencing the literature or giving proofs, so those who wish to stand on more stable ground can assume for this introduction that~$\G$ is classical.

\subsection{The work of Bernstein and Haines}
There is a natural semisimplification map on the set of~$\Phi_{\SL_2,\mathbb{C}}(\G)$ (see Section \ref{semisimpsection}), which associates to a Langlands parameter $\rho:\W_\F\times \SL_2(\mathbb{C})\rightarrow \LG(\mathbb{C})$ a \emph{semisimple parameter} (or \emph{infinitesimal character})~$\rho_{ss}:\W_\F\rightarrow \LG(\mathbb{C})$.   Write~$\Phi_{ss,\mathbb{C}}(\G)$ for the set of~$\widehat{\G}(\mathbb{C})$-conjugacy classes of semisimple parameters for~$\G$.

{The analogue of semisimplification on~$\Pi_{\mathbb{C}}(\G)$ should factor through the supercuspidal support map}.  Compatibility of the local Langlands correspondence with parabolic induction implies that the local Langlands correspondence for~$\G$ induces a semisimple correspondence
\[\LL:\{\text{supercuspidal supports}\}/\G\text{-conjugacy}\rightarrow \Phi_{ss,\mathbb{C}}(\G).\]

In \cite{BD}, Bernstein described the centre~$\mathfrak{Z}_{\mathbb{C}}(\G)$ of the category of all smooth~$\mathbb{C}$-representations of~$\G$, identifying the~$\mathbb{C}$-points of the centre with the set~$\{\text{supercuspidal supports}\}/\G(\F)$-conjugacy.  This naturally gives the left hand side of~$\LL$ the structure of an ind-affine variety over~$\mathbb{C}$.

In~\cite{Haines}, Haines introduced the structure of an ind-affine variety on~$\Phi_{ss,\mathbb{C}}(\G)$, so that (under the local Langlands correspondence for~$\G$ and some desiderata for it)~$\LL$ is induced by a morphism of~$\mathbb{C}$-algebras:~$\mathscr{O}_{\mathrm{Haines}}\rightarrow \mathfrak{Z}_{\G,\mathbb{C}}$.

\subsection{Moduli of Langlands parameters}
In our first paper \cite{DHKM}, we introduced a stack of Langlands parameters over~$\mathbb{Z}[1/p]$ generalizing a construction of the second author for~$\GL_n$ \cite{curtis}.   
The construction depends on some auxilary choices -- the choice of a discretization of the tame quotient of the Weil group -- and it is not currently known if the stack over~$\mathbb{Z}[1/p]$ is independent of these choices.  
However, its base change to~$\mathbb{Z}_\ell$, and its ring of global functions~$(\mathfrak{R}_{\LG})^{\widehat{\G}}$ are independent of these choices.   In general, the ring~$(\mathfrak{R}_{\LG})^{\widehat{\G}}$ is quite complicated and admits no simple description (in particular, it is far from being normal).  However, its base change to~$\mathbb{C}$ has a simpler structure, and we show in \cite[6.3]{DHKM} that~$(\mathfrak{R}_{\LG,\mathbb{C}})^{\widehat{\G}}\simeq \mathscr{O}_{\mathrm{Haines}}$.   Thus, under local Langlands for~$\G$ (and some desiderata for it), we obtain a map
\[(\mathfrak{R}_{\LG,\mathbb{C}})^{\widehat{\G}}\rightarrow \mathfrak{Z}_{\G,\mathbb{C}},\]
compatible with local Langlands.  

\subsection{An isomorphism of Bushnell and Henniart}
When~$\G$ is quasi-split, Bushnell and Henniart proved that the action of~$\mathfrak{Z}_{\G,\mathbb{C}}$ on a Gelfand-Graev representation~$\ind_{\U}^{\G}(\psi)$ for~$\G$ induces a canonical isomorphism
\[\mathfrak{Z}_{\G,\mathbb{C}}^{\psi\text{-}\mathrm{gen}}\rightarrow \End_{\mathbb{C}[\G]}(\ind_{\U}^{\G}(\psi)),\]
from the~$\psi$-generic factor~$\mathfrak{Z}_{\G,\mathbb{C}}^{\psi\text{-}\mathrm{gen}}$ of the Bernstein centre to the endomorphism algebra of the Gelfand--Graev representation.  For general linear groups~$\mathfrak{Z}_{\GL_n(\F),\mathbb{C}}^{\psi\text{-}\mathrm{gen}}=\mathfrak{Z}_{\GL_n(\F),\mathbb{C}}$. Under expected properties of the local Langlands correspondence (listed in Section \ref{LLcompletesection}, and verified for classical groups in Section \ref{classical} -- in particular Conjecture \ref{genericpacket} is added to Haines' desiderata), we then obtain that the composition~$(\mathfrak{R}_{\LG,\mathbb{C}})^{\widehat{\G}}\rightarrow  \End_{\mathbb{C}[\G]}(\ind_{\U}^{\G}(\psi))$ is an isomorphism of~$\mathbb{C}$-algebras. 

\subsection{Integral models of Gelfand--Graev representations}
We write~$\mathfrak{Z}_{\G,\mathbb{Z}[\sqrt{q}^{-1}]}$ for the centre of the category of all smooth representations of~$\G$ on~$\mathbb{Z}[\sqrt{q}^{-1}]$-modules.  

If~$\G$ is~$\F$-quasi-split we can associate Whittaker data to~$\G$, and their compactly induced representations --- \emph{Gelfand-Graev representations}--- satisfy many remarkable properties.  Gelfand--Graev representations are naturally defined with coefficients in~$\mathbb{Z}[1/p,\mu_{p^\infty}]$-algebras, and to state our conjecture in broadest generality we need to consider integral models of their endomorphism algebras over~$\mathbb{Z}[\sqrt{q}^{-1}]$.  In Proposition \ref{ringofdefprop}, we show that the full Gelfand-Graev representation~$\ind_{\U}^{\G}(\psi)$ descends to a $\cO_\K[1/p][\G]$-module~$W_{\U,\psi}$ for a small Galois extension~$\K/\mathbb{Q}$ contained in~$\mathbb{Q}(\mu_{p^3})$ (and hence the endomorphism algebra also descends to this ring).  In fact, for their endomorphism algebras we can do better: writing~$W_{\U,\psi,n}$ for the depth~$n$ component of~$W_{\U,\psi}$, we show there exists a natural finite type, flat~${\mathbb Z}[1/p]$-algebra $\mathfrak{E}_{\G,n}$, and natural isomorphisms 
\[\mathfrak{E}_{\G,n} \otimes_{{\mathbb Z}[1/p]} \R \simeq \End_{\R[\G]}(W_{\U',\psi',n} \otimes_{\cO_{\K}[1/p]} \R)\]
for each $\cO_K[1/p]$-algebra $\R$ and each Whittaker datum $(\U',\psi')$ of $\G$. Let~$\mathfrak{Z}_{\G,\mathbb{Z}[1/p]}^{\mathrm{ad}}$ be the subring of~$\mathfrak{Z}_{\G,\mathbb{Z}[1/p]}$ which is invariant under the natural action of~$(\mathbf{G}/\mathbf{Z})(\F)$. Then there is a natural map:
\[\mathfrak{Z}_{\G,\mathbb{Z}[1/p]}^{\ad} \rightarrow \mathfrak{E}_{\G,n}\]
that, after base change to $\mathbb{Z}[1/p,\mu_{p^\infty}]$ and the identifications above, coincides with the canonical map arising from the action of the Bernstein centre.  Set~$\mathfrak{E}_{\G}=\prod_n\mathfrak{E}_{\G,n}$.

\subsection{The main conjecture}

\begin{conjecture}\label{mainconj}
There exists a unique morphism
\[\LLIF:(\mathfrak{R}_{\LG,\mathbb{Z}[\sqrt{q}^{-1}]})^{\widehat{\G}}\rightarrow \mathfrak{Z}_{\G,\mathbb{Z}[\sqrt{q}^{-1}]}\]
 interpolating the semisimple local Langlands correspondence for~$\G$. 
 Moreover, the image of~$\LLIF$ is contained in~$\mathfrak{Z}_{\G,\mathbb{Z}[\sqrt{q}^{-1}]}^{\ad}$ and, if~$\G$ is quasi-split, then composing~$\LLIF$ with the natural map to~$\mathfrak{E}_{\G,\mathbb{Z}[\sqrt{q}^{-1}]}$ defines an isomorphism~$(\mathfrak{R}_{\LG,\mathbb{Z}[\sqrt{q}^{-1}]})^{\widehat{\G}}\xrightarrow{\sim}  \mathfrak{E}_{\G,\mathbb{Z}[\sqrt{q}^{-1}]}$.
\end{conjecture}
In particular, when~$\G$ is quasi-split, we expect that~$\LLIF$ is injective.  

\subsection{The main theorem}
We axiomatize some expected properties of the semisimple local Langlands correspondence in Definition \ref{semisimplecorrespdef} (which we verify for classical groups) and in Theorem \ref{axiomaticLLIF} we prove Conjecture \ref{mainconj} after inverting an integer depending only on~$\G$.

Let~$\N_\G$ be the product of all primes that divide the pro-order of a compact open subgroup of~$\G$ (that is, the product of the primes which are ``non-banal'' for~$\G$).  Let~$\N_{\LG}$ denote the product of all primes which either divide the integer~$\N_{\widehat{\G}}$ or are non-$\LG$-banal primes, both introduced in~\cite[Sections 4, 5.4]{DHKM}.  Finally let~$\M_{\G}$ denote the lowest common multiple of~$\N_{\G}$ and~$\N_{\LG}$.  

Note that, if~$\G$ is unramified with no exceptional factor, then $\M_{\G}=\N_{\G}$ by \cite[Corollary 5.29 and Remark 6.3]{DHKM}.

\begin{theoremALLIF}
Suppose~$\mathscr{C}=(\mathscr{C}_\M)$ is a semisimple correspondence for~$\G$ over~$\Qbar$ as in Definition \ref{semisimplecorrespdef}.  
\begin{enumerate}
\item There exists a unique morphism
\[\mathscr{C}\mathrm{IF}_\G:(\mathfrak{R}_{\LG,\mathbb{Z}[\sqrt{q}^{-1},1/\N_{\G}]})^{\widehat{\G}}\rightarrow \mathfrak{Z}_{\G,\mathbb{Z}[\sqrt{q}^{-1},1/\N_{\G}]}\]
interpolating~$\mathscr{C}$.   
\item   The image of~$\mathscr{C}\mathrm{IF}_\G$ 
is contained in the subrings~$\mathfrak{Z}^{\St}_{\G,\mathbb{Z}[\sqrt{q}^{-1},1/\N_{\G}]}$ of Definition \ref{stablecentredef}, and
~$\mathfrak{Z}_{\G,\mathbb{Z}[\sqrt{q}^{-1},1/\N_{\G}]}^{\ad}$.
\item Suppose further~$\G$ is~$\F$-quasi-split.  
Then, the composition with the natural map~$ \mathfrak{Z}_{\G,\mathbb{Z}[\sqrt{q}^{-1},1/\M_{\G}]}^{\ad}\rightarrow\mathfrak{E}_{\mathbb{Z}[\sqrt{q}^{-1},1/\M_{\G}]}(\G)$, 
\[(\mathfrak{R}_{\LG,\mathbb{Z}[\sqrt{q}^{-1},1/\M_{\G}]})^{\widehat{\G}}\xrightarrow{\mathscr{C}\mathrm{IF}_\G}  \mathfrak{Z}^{\ad}_{\G,\mathbb{Z}[\sqrt{q}^{-1},1/\M_{\G}]}\rightarrow \mathfrak{E}_{\mathbb{Z}[\sqrt{q}^{-1},1/\M_{\G}]}(\G)\]
is an isomorphism, and these maps induce isomorphisms
\[(\mathfrak{R}_{\LG,\mathbb{Z}[\sqrt{q}^{-1},1/\M_{\G}]})^{\widehat{\G}}\simeq \mathfrak{Z}^{\St}_{\G,\mathbb{Z}[\sqrt{q}^{-1},1/\M_{\G}]}\simeq \mathfrak{E}_{\mathbb{Z}[\sqrt{q}^{-1},1/\M_{\G}]}(\G).\]
\end{enumerate}
\end{theoremALLIF}


\subsection{Bernstein's decomposition in the banal setting}
For a Levi subgroup~$\M$ of~$\G$ let~$\M^\circ$ denote the subgroup of~$\M$ generated by all compact open subgroups.  Write~$\mathfrak{B}_{\Qbar}(\G)$ for the set of inertial equivalence classes of supercuspidal supports of $\Qbar$-representations of~$\G$.  For~$[\M,\rho]_{\G}\in\mathfrak{B}_{\Qbar}(\G)$ we write $P_{(\M,\rho)}$ for the lattice we introduce in \eqref{Integralprog}.  It is defined over the ring of integers in a number field, and~$P_{(\M,\rho)}\otimes\Qbar$ is a finitely generated projective generator of Bernstein for the direct factor subcategory of~$\Rep_{\Qbar}(\G)$ associated to~$[\M,\rho]_{\G}$.

The first step in our proof of the conjecture in the banal case is to obtain the following description of~$\mathfrak{Z}_{\G,\overline{\mathbb{Z}}[1/\N_{\G}]}$ analogous to Bernstein's description of~$\mathfrak{Z}_{\G,\Qbar}$:  
\begin{theoremBD}
The category~$\Rep_{\overline{\mathbb{Z}}[1/\N_\G]}(\G)$ decomposes as
\[\Rep_{\overline{\mathbb{Z}}[1/\N_\G]}(\G)=\prod_{[\M,\rho]_\G\in\mathfrak{B}_{\Qbar}(\G)}
  \Rep_{\overline{\mathbb{Z}}[1/\N_\G]}(\G)_{[\M,\rho]_{\G}},\]
where~$\Rep_{\overline{\mathbb{Z}}[1/\N_\G]}(\G)_{[\M,\rho]_{\G}}$ is the direct factor subcategory generated by the finitely generated projective~$P_{(\M,\rho)}\otimes\overline{\mathbb{Z}}[1/\N_{\G}]$.  Moreover, the choice
of~$P_{(\M,\rho)}$ identifies the
centre of~$\Rep_{\overline{\mathbb{Z}}[1/\N_\G]}(\G)_{[\M,\rho]_{\G}}$ with\[(\overline{\mathbb{Z}}[1/\N_\G][\M/\M^\circ]^{\H_{(\M,\rho)}})^{\W_{(\M,\rho)}}.\]
\end{theoremBD}
Here,~$\H_{(\M,\rho)}$ and~$\W_{(\M,\rho)}$ denote finite groups introduced in Bernstein's decomposition (summarised in Theorem \ref{BernsteinDeligne}).

This description has long been expected (see, for example, the discussion in \cite[Section 6]{datnu}).  On our way to establishing it we prove the following pleasing lifting result for cuspidals in banal characteristic (the proof of which uses second-adjointness):
\begin{propositionlocalcusplift}
Let~$\ell$ be a banal prime for~$\G$.  Then the reduction modulo~$\ell$ of any irreducible integral cuspidal~$\Ql$-representation of~$\G$ is irreducible and cuspidal, and conversely all irreducible cuspidal~$\Fl$-representations of~$\G$ lift.
\end{propositionlocalcusplift}

\subsection{An application to finiteness}
While it lies outside the main thrust of this paper, as explained in the introduction of \cite{DHKMfiniteness}, the description of~$ \mathfrak{Z}_{\G,\overline{\mathbb{Z}}[1/\N_{\G}]}$ above is the final ingredient in our proof of finiteness of Hecke algebras over~$\mathbb{Z}[1/p]$, and we obtain:
\begin{theorem}\label{noetherianness}
For any Noetherian~$\mathbb{Z}[1/p]$-algebra~$\R$, and any compact open subgroup~$\H$, the algebra~$\R[\H\backslash\G/\H]$
is a finitely generated module over its centre, which is a finitely generated~$\R$-algebra.
\end{theorem}

\subsection{Integrality of Langlands parameters}\label{Integralitysubsection}
Our approach to Theorem \ref{axiomaticLLIF} requires us to transfer integrality over a semisimple correspondence; we do this using criteria for integrality.  Recall that an irreducible~$\Ql$-representation of~$\G$ is integral if and only if its supercuspidal support is integral by \cite{DHKMfiniteness}, and a supercuspidal {~$\Ql$-}representation is integral if and only if its central character is integral by \cite[II 4.12]{Vig96}.  We provide a complete analogue of this characterization for Langlands parameters:
\begin{propositionintchar}
Let~$\phi_{\ell}:\W_{\F}\rightarrow \LG(\Ql)$ be a continuous $L$-homomorphism.
\begin{enumerate}
\item Suppose $\phi_{\ell}$ is discrete and Frobenius semisimple.  Then~$\phi_{\ell}$ is integral if and only if its central character is integral.
%
%
\item In general, the following are equivalent:
\begin{enumerate}
\item $\phi_{\ell}$ is integral;
  \item the
 Frobenius-semisimplification of $\phi_{\ell}$ is
  integral.
\item  the semisimplification of $\phi_{\ell}$ is
  integral. 
\item \label{iiIntegralProp3} the $\Ql$-point of $\underline{\Z}^1(\W_\F, \widehat{\G})\sslash \widehat{\G}$ corresponding to $\phi_{\ell}$ factors through $\Spec (\Zl)$.
\end{enumerate}
\end{enumerate}
\end{propositionintchar}
These characterizations of integrality show that if the local Langlands correspondence for~$\G$ exists and satisfies some natural desiderata then it preserves integrality, cf.~Remark \ref{remarkpreservesintegrality}.

\subsection{Proof of the main theorem}
For banal primes~$\ell$, our description of the centre~$\mathfrak{Z}_{\G,\Zl}$, together with compatibility of the local Langlands correspondence with twisting by unramified characters, parabolic induction, and with central characters, then allows us to show that there exists a unique morphism~$(\mathfrak{R}_{\LG,\Zl})^{\widehat{\G}}\rightarrow \mathfrak{Z}_{\G,\Zl}$, interpolating the {semisimple correspondence} for~$\G$.  

If, moreover,~$\G$ is quasi-split, then fixing a Whittaker datum over~$\Zl$, we show that an analogue of Bushnell--Henniart's isomorphism holds over~$\Zl$ (for~$\ell$ banal), and together with the description of the connected components of~$(\mathfrak{R}_{\LG,\Zl})^{\widehat{\G}}$ in \cite{DHKM} (here we need to also assume~$\ell$ does not divide~$\N_{\LG}$) we obtain the second statement of the conjecture over~$\Zl$ for~$\ell$ not dividing~$\M_{\G}=\mathrm{lcm}(\N_\G,\N_{\LG})$.  Galois equivariance of the semisimple local Langlands correspondence allows us to descend these maps to~$\mathbb{Z}_{\ell}[\sqrt{q}]$, and considering these maps together, for all~$\ell$ not dividing~$\M_{\G}$, in the simpler case over~$\mathbb{Q}(\sqrt{q})$ gives us our result over~$\mathbb{Z}[\sqrt{q}^{-1},1/\M_{\G}]$. 
 %

\subsection{Classical groups}
For our (conditionless) results for classical $p$-adic groups, we need to verify that the {(semisimple)} local Langlands correspondence and the desiderata for it we use in our proof are known in these cases.  This is the subject of Section \ref{classical}. 

Let~$\G$ be a classical~$p$-adic group.  Firstly we show that the local Langlands correspondence is compatible with parabolic induction, generalizing a result of Moussaoui \cite{Moussaoui} in the split case by a different method, in the following sense:

\begin{propositionpara}
Let~$\P$ a parabolic subgroup of~$\G$ with Levi decomposition~$\P=\M \N$, ~$\rho$ be an irreducible representation of~$\M$, and~$\pi$ be an irreducible subquotient of~$i^{\G}_{\P}(\rho)$.  Then, letting~$\iota_{\M,\G}:\LM(\mathbb{C})\hookrightarrow \LG(\mathbb{C})$ denote an embedding dual to~$\M\hookrightarrow \G$,~the semisimple parameters~$\iota_{\M,\G}\circ(\mathcal{LL}_{\M}(\rho))_{ss}$ and~$(\mathcal{LL}_{\G}(\pi))_{ss}$ are~conjugate in~$\widehat{\G}(\mathbb{C})$. 
\end{propositionpara}

We prove this using properties of the Plancherel measure -- in particular, its compatibility with parabolic induction, and its interpretation in terms of gamma factors of Langlands parameters (also known as ``Langlands' conjecture on the Plancherel measure'') -- together with a semisimple converse theorem (Proposition \ref{semisimpleconverse}).  This approach is inspired by recent approaches of Gan--Savin \cite{gan_savin}, and Gan--Ichino \cite{GanIchino14,GanIchino16} to the local theta correspondence.  

For this we have to extend the proof of Gan--Ichino\cite[B.5]{GanIchino14} of compatibility of the Plancherel measure with subrepresentations of parabolic inductions to compatibility with sub\emph{quotients} of parabolic inductions.  We follow the pattern of their proof, but have added some more details (cf., Remark \ref{GanIchinodetailsremark}) in the form of Appendices \ref{Appendix:intertwining} and \ref{apdx:j_functions} to establish some basic properties of the Plancherel measure following the algebraic treatments of the Plancherel measure of \cite{Waldspurger} and \cite{datnu}.

Similarly, using compatibilities of the Plancherel measure and gamma factors with field automorphisms, we show that~$(~)_{ss}\circ\mathcal{LL}_\G$ is compatible with field automorphisms fixing~$\sqrt{q}$:

\begin{propositionauto}
Let~$\pi$ be an irreducible representation of~$\G$, and~ $\sigma:\mathbb{C}\to \mathbb{C}$ be an automorphism of fields fixing~$\sqrt{q}$.   Then~$\mathcal{LL}_\G(\pi^\sigma)_{ss}$ is conjugate to $(\mathcal{LL}_{\G}(\pi)^{\sigma})_{ss}$ in~$\widehat{\G}(\mathbb{C})$.
\end{propositionauto}


We also need to show~$\mathcal{LL}_\G$ preserves integrality of supercuspidal~$\Ql$-representations (after rewriting the correspondence for~$\Ql$-representations via choosing an isomorphism~$\Ql\simeq\mathbb{C}$).  In fact, we prove this for all irreducible representations:

\begin{propositionint}
An irreducible representation of~$\G$ is integral if and only if its associated~$\ell$-adic Langlands parameter is integral.  
\end{propositionint}
Proposition \ref{integralclass} follows from the characterizations of integrality mentioned in Section \ref{Integralitysubsection}, together with the compatibility of local Langlands for general linear groups with central characters and the compatibility of local Langlands for classical groups with parabolic induction.

The fibres of~$\mathcal{LL}_\G$ are called \emph{$L$-packets}.  We collect $L$-packets with the same semisimple parameter together into \emph{extended~$L$-packets}, and conjecture for a quasi-split group~$\G$, and a Whittaker datum~$(\U,\psi)$ in~$\G$, that each extended $L$-packet contains a unique~$\psi$-generic representation (Conjecture \ref{genericpacket}).  We prove this for quasi-split classical groups:

\begin{propositionext}
Let~$\G$ be a quasi-split classical~$p$-adic group and~$(\U,\psi)$ be a Whittaker datum for~$\G$.  In each extended~$L$-packet of~$\G$ there exists a unique~$\psi$-generic representation.
\end{propositionext}

The proof of this proposition follows from results in the literature on generic representations in $L$-packets, and a geometric interpretation of a conjecture of Gross, Prasad, and Rallis -- see Conjecture \ref{Davidsmaxorbitconj}, and Propositions \ref{Davidsmaxorbitprop} and \ref{extendedpacketunderconjectures}.

\subsection{Connections}
Some of the ideas for this article took shape and were circulating for several years before it was completed (for example, Conjecture \ref{mainconj} was announced in 2018 in talks including \cite{KOberwolfach}).  There have been many developments in the literature since then, we now explain some connections to our work.

In a tour de force, for any connected reductive~$p$-adic group~$\G$ and{~$\ell$ not dividing the order of~$\pi_1(\widehat{\G})_{\text{tor}}$}, Fargues--Scholze \cite{FarguesScholze} construct a map
\[(\mathfrak{R}_{\LG,\mathbb{Z}_{\ell}[\sqrt{q}]})^{\widehat{\G}}\rightarrow\mathfrak{Z}_{\G,\mathbb{Z}_{\ell}[\sqrt{q}]}\]
by developing the geometric Langlands programme on the Fargues--Fontaines curve.  Not much is known about this map beyond the basic properties of \cite[Theorem I 9.6]{FarguesScholze}, which include that for~$\GL_n(\F)$ their construction is compatible with the usual (semisimple) local Langlands correspondence for~$\GL_n(\F)$.  Recently, Hamann \cite{Hamann}  has shown their construction is compatible with the (semisimple) local Langlands correspondence of Gan--Takeda for~$\mathrm{GSp}_4(\F)$, Hansen--Kaletha--Weinstein \cite{HKWKott} have shown compatibility for inner forms of general linear groups, and Bertoloni-Meli--Hamann--Nguyen \cite{BMHN} have shown compatibility for odd unramified unitary groups with the (semisimple) local Langlands correspondence of Mok. For quasi-split unitary groups of odd dimension Bertoloni-Meli--Hamann--Nguyen \cite[Proposition 2.11]{BMHN} also establish compatibility of Mok's correspondence with parabolic induction, a special case of Proposition \ref{parabolicinduction:theorem} of this paper.  

{Late in the preparation of the paper, we were informed of the local Langlands parameterization \cite[Theorem 3.7]{aubert2022}.  It seems plausible that our results can be extended to include even special orthogonal groups using this result, however some arguments, such as Propositions \ref{GGPprop} and \ref{Propplanchereldecomp}, would require modifications that we have not attempted.  In~\cite[Theorem 3.7(a)]{aubert2022}, Aubert--Moussaoui--Solleveld also establish compatibility of the local Langlands correspondence for classical groups with parabolic induction.}

More recently, at a late stage in the preparation of this paper, Cunningham--Dijols--Fiori--Zhang \cite[Prop.~4.1]{CDFZ24} have independently established the equivalence between properties \eqref{maxorbitprop1} and \eqref{maxorbitprop2} Proposition \ref{Davidsmaxorbitprop}. (Note that the definition of ``open parameter" in ibid.~is equivalent to our definition of a ``parameter with maximal monodromy").  This equivalence gives a geometric reformulation of Gross--Prasad--Rallis' conjecture on generic representations in~$L$-packets (cf.,~Conjecture \ref{Davidsmaxorbitconj} or \cite[Conjecture 4.6]{CDFZ24}).  We provide a further geometric characterization here: Proposition \ref{Davidsmaxorbitprop} \eqref{maxorbitprop3} in terms of the moduli space of Langlands parameters of \cite{DHKM} which does not appear in~\cite{CDFZ24}.

Initial versions of this article were written under weak hypotheses on~$\G$, which were known to be satisfied for all classical $p$-adic groups (with $p\neq 2$) and all ``tame'' groups thanks to \cite{datfinitude}.  We needed these hypotheses so that we could apply ``second-adjointness'' of parabolic functors integrally.  Recently, in \cite{DHKMfiniteness}, using Fargues--Scholze's morphism, we proved second-adjointness holds in general, so we have been able to remove this hypothesis -- however, this means some of our results  (notably, Theorems \ref{banaldecomptheorem} and \ref{axiomaticLLIF}) depend on Fargues and Scholze's construction whenever we fall out of the range of \cite{datfinitude} (see Remark \ref{remarkFSdependence} for more details).

There has also been a race towards a categorification of the local Langlands correspondence, beginning with conjectures inspired by categorical statements and conjectures in the geometric Langlands programme.  See \cite{BZCHN} and \cite{Hellmann} for early ideas towards this, and \cite{FarguesScholze} and \cite{Zhu} for the most ambitious conjectures which relate (derived) categories of smooth representations to (derived) categories of ind-coherent sheaves on stacks of Langlands parameters.  Integral versions of these conjectures  also predict, for quasi-split groups, a natural map from the ring of global functions on the stack of Langlands parameters to the endomorphisms of a Gelfand--Graev representation (as in Conjecture \ref{mainconj}), as they fix their equivalence by sending (a choice of) Gelfand-Graev representation of a quasi-split connected reductive group to the structure sheaf on the stack of Langlands parameters.  

\subsection{Acknowledgements}
The first author was partially supported by ANR grant COLOSS ANR-19-CE40-0015. The second author was partially supported by EPSRC New Horizons grant EP/V018744/1. The third author was supported by EPSRC grant EP/V001930/1 and the Heilbronn Institute for Mathematical Research. The fourth author was partially supported by NSF grants DMS-2001272 and DMS-2302591. We thank Jessica Fintzen, David Hansen, Nadir Matringe, Ahmed Moussaoui, Gordan Savin, {Sug Woo Shin}, Shaun Stevens, and Marie-France Vign\'eras for helpful conversations on the subject of the paper, {and Tom Haines and Maarten Solleveld for their comments and corrections.}

\section{Notation}

Let~$\F$ be a non-archimedean local field with finite residue field of cardinality~$q$ a power of $p$. For any non-archimedean local field~$\E$ (for example, for finite extensions of~$\F$), we write~$\cO_\E$ for its ring of integers,~$\cP_\E$ for the unique maximal ideal, and $k_\E$ for its residue field~$\cO_\E/\cP_\E$.   

Let~$\mathbf{G}$ be a connected reductive algebraic group defined over~$\F$ and~$\G=\mathbf{G}(\F)$.  

Unless otherwise stated ``module" means ``left module'', and~$\R$ denotes a commutative~$\mathbb{Z}[1/p]$-algebra.  We suppose that all~$\R[\H]$-modules for a locally profinite group~$\H$, equivalently that all~$\R$-representations of~$\H$, henceforth considered are smooth, and we denote by~$\Rep_\R(\H)$ the abelian category of all (smooth)~$\R[\H]$-modules.

Let~$\ell\neq p$ be prime.  Let~$\Ql$ denote an algebraic closure of~$\mathbb{Q}_\ell$,~$\Zl$ the ring of integers of~$\Ql$, and~$\Fl$ its residue field.  We fix an algebraic closure~$\Qbar$ of~$\mathbb{Q}$, and denote by~$\Zbar$ the subring of algebraic integers.   We fix once and for all embeddings~$\Qbar\hookrightarrow \Ql$ and~$\Qbar\hookrightarrow\mathbb{C}$.

\section{The integral centre}

\subsection{The centre}

The \emph{centre}~$\mathfrak{Z}_{\G,\R}$ of the category of~$\R[\G]$-modules~$\Rep_{\R}(\G)$ is the ring of endomorphisms of the identity functor~$1_{\Rep_{\R}(\G)}:\Rep_{\R}(\G)\rightarrow \Rep_{\R}(\G)$.   We can identify an element~$z\in\mathfrak{Z}_{\G,\R}$ with a collection~$(z_{\M})$, over all~$\R[\G]$-modules~$\M$ of endomorphisms~$z_\M\in\End_{\R[\G]}(\M)$, satisfying for all morphisms of~$\R[\G]$-modules~$f:\M\rightarrow \M'$,
\begin{equation}
\tag{$\dagger$}
z_{\M'}\circ f=f\circ z_{\M};\end{equation}
i.e.~commuting with all morphisms in the category.  The ring structure on the collections~$(z_{\M})$ is given by componentwise addition and composition of endomorphisms, and applying $(\dagger)$ to all endomorphisms~$f\in\End_{\R[\G]}(\M)$, we see that~$z_\M$ belongs to the centre of~$\End_{\R[\G]}(\M)$ and~$\mathfrak{Z}_{\G,\R}$ is a commutative~$\R$-algebra which acts naturally on all~$\R[\G]$-modules.

Let~$\phi:\R\rightarrow \R'$ be a homomorphism of commutative rings (with identity).   Then we have a \emph{restriction of scalars} functor
\[\phi^*:\Rep_{\R'}(\G)\rightarrow\Rep_{\R}(\G),\]
which takes a smooth~$\R'[\G]$-module~$\M$ to the smooth~$\R[\G]$-module~$\M$, whose underlying set and action of~$\G$ is the same and on which~$r\cdot m:=\phi(r)\cdot m$, and which is the identity on morphisms.   Note that, with this prescribed scalar action, it is obvious that~$g\in\G$ acts on~$\M$ by~$\R$-linear scalar automorphisms of~$\M$.

\begin{lemma}
The morphism~$\phi$ induces a homomorphism
\[\phi^*:\mathfrak{Z}_{\G,\R}\rightarrow \mathfrak{Z}_{\G,\R'}\]
where, for~$z\in\mathfrak{Z}_{\G,\R}$ and a smooth~$\R'[\G]$-module~$\M$, we set~$\phi^*(z)_{\M}:=z_{\phi^*(\M)}$.
\end{lemma}
\begin{proof}
As~$\phi^*(z)_{\M}$ is an element of the centre of~$\End_{\R[\G]}(\M)$, it commutes with multiplication by elements of~$\R'$ hence, and defines an element of~$\End_{\R'[\G]}(\M)$.  Moreover, the~$(\phi^*(z)_{\M})$ commute with all morphisms of smooth~$\R[\G]$-modules, so in particular commute with all morphisms of smooth~$\R'[\G]$-modules.
\end{proof}

\begin{lemma}
%
Suppose~$\phi^*:\R\rightarrow \R'$ is a monomorphism, then~$\phi^*:\mathfrak{Z}_{\G,\R}\rightarrow \mathfrak{Z}_{\G,\R'}$ is monomorphism.
\end{lemma}

\begin{proof}
To show injectivity we need to show that if~$z\in \mathfrak{Z}_{\G,\R}$ annihilates all~$\R'[\G]$-modules then $z=0$.  For this, it suffices to prove that~$z_V=0$  for a generating family of objects of~$\Rep_\R(\G)$. So we can consider those~$\V$ of the form~$\R[\G/\K]$ with~$\K$ compact open, and then~$z_V$ is certainly zero since it is zero on~$\R'[\G/\K]$.
\end{proof}

\begin{corollary}\label{flatnessDD}
Let~$\R$ be an integral domain with field of fractions~$\K$ of characteristic $0$.  The integral centre~$\mathfrak{Z}_{\G,\R}$ is a reduced torsion free~$\R$-algebra.  If~$\R$ is Dedekind, then~$\mathfrak{Z}_{\G,\R}$ is flat over~$\R$.
\end{corollary}
\begin{proof}
The centre $\mathfrak{Z}_{\G,\R}$ is torsion free and reduced as it embeds in~$ \mathfrak{Z}_{\overline{\K}}(\G)$ which is torsion free and reduced by \cite{BD}.  If~$\R$ is a Dedekind domain, then as~$\mathfrak{Z}_{\G,\R}$ is torsion free, it is flat.
\end{proof}

\subsection{The decomposition by level}\label{leveldecomp}
Now using our basic assumption that~$p$ is invertible in our coefficient ring, Moy--Prasad theory allows one to give a coarse decomposition of~$\Rep_{\R}(\G)$.  Following \cite[Appendice A]{datfinitude}, there is a family of finitely generated projective~$\mathbb{Z}[1/p][\G]$-modules~$Q_n$~($n\in\mathbb{N}$), defined by the finite sums
\begin{align*}
Q_0&:=\bigoplus_{x\in \mathrm{Vert}(\G)/\G}\ind_{\G_{x,0+}}^\G(1),\\
Q_m&:=\bigoplus_{x\in \mathrm{Opt}(\G)/\G}\ind_{\G_{x,r_m}}^\G\left(\bigoplus_{\chi\in \mathrm{UR}_{m,x}}\chi\right), \qquad(m>0),
\end{align*}
where~$\mathrm{Vert}(\G)/\G$ (respectively~$\mathrm{Opt}(\G)/\G$)
denotes a set of representatives for the~$\G$-conjugacy classes of
vertices (respectively optimal points) of the Bruhat--Tits building
of~$\G$; and~$\mathrm{UR}_{m,x}$ denotes a set
of~$\mathbb{Z}[1/p,\zeta_p]$-valued characters
of~$\G_{x,r_m}/\G_{x,r_m+}$: the \emph{unrefined minimal types of
  level~$m$} {(where~$(r_m)_{m\in\mathbb{N}}$ denotes the increasing sequence of rational numbers indexing the jumps in filtrations of the parahorics associated to optimal points in the building of~$\G$ as in \cite[Appendice A]{datfinitude}). } 

Suppose~$\R$ is a~$\mathbb{Z}[1/p]$-algebra, and set~$Q_{n,\R}=Q_n\otimes \R$ for all~$n\in\mathbb{N}$.  By \cite[Lemma A.3]{datfinitude}, we have a \emph{decomposition by level}
\begin{align*}
\Rep_\R(\G)&=\prod_{n\in\mathbb{N}}\Rep_\R(\G)_n,\\
\mathfrak{Z}_{\G,\R}&=\prod_{n\in\mathbb{N}}\mathfrak{Z}_{\G,\R,n},
\end{align*}
where~$Q_{n,\R}$ is a progenerator for~$\Rep_\R(\G)_n$ and we (hence) can identify~$\mathfrak{Z}_{\G,\R,n}$ with the centre of~$\End_{\R[\G]}(Q_{n,\R})$.  In particular, an element~$(z_\M)$ in~$\mathfrak{Z}_{\G,\R}$ is completely determined by the endomorphisms~$(z_{Q_{n,\R}})$, for~$n\in\mathbb{N}$.

\begin{proposition}\label{Corembed}
Suppose~$\R$ is a Noetherian~$\mathbb{Z}[1/p]$-algebra, and~$\R'$ is a flat
commutative~$\R$-algebra. Then 
the natural map 
\[\mathfrak{Z}_{\G,\R,n}\otimes \R'\rightarrow \mathfrak{Z}_{\G,\R',n},
\]
is an isomorphism. 
\end{proposition}

\begin{proof}
  Pick an open pro-$p$-subgroup $\H$ of $\G$ such that $Q_{n}$ is generated by its
  $\H$-invariants, and denote by $\varepsilon_{n,H}$ the central idempotent in
  $\mathbb{Z}[1/p][\H\backslash \G/\H]$ given by the projection on the level $n$
  factor of $\mathbb{Z}[1/p][\G/\H]$. We still write $\varepsilon_{n,H}$ for its image in
  $\R[\H\backslash \G/\H]$ or $\R'[\H\backslash \G/\H]$.
  Then $\varepsilon_{n,H} \R[\G/\H]$ is a finitely generated
projective generator of $\Rep_\R(\G)_n$, hence
$\mathfrak{Z}_{\G,\R,n}=\varepsilon_{n,H} \Z(\R[\H\backslash \G/\H])$.
Similarly we have   $\mathfrak{Z}_{\G,\R',n}=\varepsilon_{n,H} \Z(\R'[\H\backslash \G/\H])$,
and we see that it suffices to prove that the natural map
$ \Z(\R[\H\backslash \G/\H])\otimes_{\R}\R' \rightarrow \Z(\R'[\H\backslash \G/\H])$ is an
isomorphism, which follows from Lemma \ref{centrelemma} (1).
\end{proof}

\begin{corollary}
If~$\R$ is  flat over~$\mathbb{Z}[1/p]$, then~$\mathfrak{Z}_{\G,\R,n}$ is flat
over $\R$. If $\R$ is also Noetherian, then $\mathfrak{Z}_{\G,\R}$ is flat over $\R$.
\end{corollary}

\begin{proof}
By Proposition \ref{Corembed},~$\mathfrak{Z}_{\G,\R,n}\simeq
\mathfrak{Z}_{\G,\mathbb{Z}[1/p],n}\otimes \R$ and~$\mathfrak{Z}_{\G,\mathbb{Z}[1/p],n}$
is flat over $\mathbb{Z}[1/p]$ by Corollary \ref{flatnessDD},
hence~$\mathfrak{Z}_{\G,\R,n}$ is flat. When $\R$ is Noetherian, a product of flat
$\R$-modules is flat, hence  $\mathfrak{Z}_{\G,\R}$ is flat.
\end{proof}

\begin{corollary}\label{GaloisBC}
In the context of Proposition \ref{Corembed}, assume further that $\R$ is flat over
$\mathbb{Z}[1/p]$, and is the fixed subring
$(\R')^{\Gamma}$ of $\R'$ under a finite group of ring automorphisms $\Gamma$ of $\R'$. Then
the group $\Gamma$ acts naturally on $\mathfrak{Z}_{\G,\R'}$ and the natural map induces
an isomorphism 
$\mathfrak{Z}_{\G,\R}\xrightarrow{\sim}(\mathfrak{Z}_{\G,\R'})^{\Gamma}$.
\end{corollary}
\begin{proof}
The action is induced by functoriality of $\R'\mapsto \mathfrak{Z}_{\G,\R'}$ (for $\R'$ in
the category of $\R$-algebras) and preserves the decomposition
$\mathfrak{Z}_{\G,\R'}=\prod_{n} \mathfrak{Z}_{\G,\R',n}$ since it is induced from
$\R$. By construction, the natural map
$\mathfrak{Z}_{\G,\R,n}\otimes_{\R}\R'\rightarrow \mathfrak{Z}_{\G,\R',n}$
is $\Gamma$-equivariant if we let $\Gamma$ act on the source via its action on $\R'$.
By Proposition \ref{Corembed}, it thus suffices to prove that
the canonical map $\mathfrak{Z}_{\G,\R,n}\rightarrow (\mathfrak{Z}_{\G,\R,n}\otimes_{\R}\R')^{\Gamma}$ is an
isomorphism. Writing $\R=(\R')^{\Gamma}$ as the kernel of the map $\R'\rightarrow
(\R')^{\oplus\Gamma}$,
$a\mapsto (a,\cdots,a)-(\gamma(a))_{\gamma\in\Gamma}$, this follows from
flatness of $\mathfrak{Z}_{\G,\R,n}$ over $\R$, as in the previous corollary.
\end{proof}

Note that the last corollary applies in particular to the case where $\R$ is a field and
$\R'$ is a Galois extension with group $\Gamma$, or when $\R$ is a Dedekind ring flat over
$\mathbb{Z}$ and $\R'$ is its normalization in a Galois extension of its field of fractions
with group $\Gamma$.

\section{The Bernstein decomposition in banal characteristics}

The aim of this section is to prove an integral version of Bernstein's decomposition
theorem (which originally holds over $\mathbb{C}$ or $\overline{\mathbb{Q}}$). This will be done
after inverting the so-called ``non-banal'' primes, whose definition is recalled in
Subsection 4.4.  Note  that no banal hypothesis is
required in the definitions and results of subsections 4.1 and 4.2. 

\subsection{Parabolic induction and cuspidal support}
Let~$\P$ be a parabolic subgroup of~$\G$ and let~$\P=\M\U$ be a Levi decomposition of~$\P$.  We write~$\W_\M$ for the Weyl group of~$\M$ and~$\Z_\M$ for its centre.  We set~$\Z=\Z_\G$ and~$\W=\W_\G$.

Write~$I^\G_{\M,\P}$ for the (non-normalized) \emph{parabolic induction functor}~$I^\G_{\M,\P}:\Rep_\R(\M)\rightarrow\Rep_{\R}(\G)$, and~$R^\G_{\M,\P}$ for its left adjoint the (non-normalized) \emph{parabolic restriction functor}~$R^\G_{\M,\P}:\Rep_\R(\G)\rightarrow\Rep_\R(\M)$. These functors are exact, by \cite[II 2.1]{Vig96}.  A much subtler property is the so-called ``second-adjointness'' of parabolic functors, famously established by Bernstein for~complex representations and which we recently extended to all~$\mathbb{Z}[1/p]$-algebras:

\begin{theorem}[{\cite[Corollary 1.3]{DHKMfiniteness}}]\label{secadj}
For all pairs of opposite parabolic subgroups~$(\P, \P^{\circ})$ in~$\G$ with common Levi component~$\M=\P\cap\P^\circ$, the twisted opposite Jacquet functor~$\delta_{\P} R^{\G}_{\M,\P^{\circ}}:\Rep_{\R}(\G)\rightarrow \Rep_{\R}(\M)$ is right adjoint to the parabolic induction functor~$I_{\M,\P}^{\G}:\Rep_\R(\M)\rightarrow \Rep_{\R}(\G)$, where $\delta_\P$ denotes the modulus character of~$\P$.
\end{theorem}

\begin{remark}\label{remarkFSdependence}
The proof of \cite[Corollary 1.3]{DHKMfiniteness} uses the main result of Fargues--Scholze \cite{FarguesScholze}, however second-adjointness with coefficients in~$\mathbb{Z}[1/p]$-algebras was proved in \cite[Th\'eor\`eme 1.5]{datfinitude} by purely representation theoretic methods, under the hypothesis that some loose form of type theory exists for $\G$. In particular, these methods apply to classical groups (with~$p\neq 2$, by \cite{St08}), or any group for which  Yu's construction of ``generic types'' is exhaustive for $\G$ (e.g.,~when $p$ does not divide the order of the Weyl group, by Fintzen \cite{Fintzen}). 
\end{remark}

We fix a choice of square root~$\sqrt{q}$ of~$q$.  When~$\R$ is a~$\mathbb{Z}[\sqrt{q}^{-1}]$-algebra, we consider the normalized variants of the parabolic functors~$i_{\M,\P}^\G(-):=I_{\M,\P}^\G(\delta_\P^{1/2}\otimes-)$ and~$ r^\G_{\M,\P}(-):=\delta_\P^{-1/2}\otimes R^\G_{\M,\P}(-)$.  As the \emph{modulus character}~$\delta_\P$ of~$\P$ takes values in~$q^{\mathbb{Z}}$, these are well-defined as we have fixed~$\sqrt{q}$.  

Let~$\G^\circ$ denote the subgroup of~$\G$ generated by all compact open subgroups, it is open and normal in~$\G$, cf.~\cite[I 1.3]{Vig96}.  Moreover, as in ibid., $\G/\G^\circ$ is a free abelian group of finite rank equal to the rank of a maximal~$\F$-split torus in the centre of~$\G$; and~$\G/\Z\G^\circ$ is finite.  We define the \emph{set of unramified~$\R$-valued characters of~$\G$} by
\[\Psi_\G(\R):=\{\chi:\G\rightarrow \R^\times:\chi_{|\G^\circ}=1\}.\]
We can identify~$\Psi_\G(\R)$ with the group of $\R$-points of the algebraic torus $\Psi_{G}:=\Spec(\mathbb{Z}[\G/\G^\circ])$.  



We recall some preliminary definitions and results from Vign\'eras' book \cite{Vig96}.  
\begin{definition}
\begin{enumerate}
\item Let~$\H$ be a locally profinite group, and suppose~$\R$ is Noetherian.  An~$\R[\H]$-module~$\pi$ is called \emph{admissible} if~$\pi^\K$ is finitely generated for all compact open subgroups~$\K$ of~$\H$.
\item An~$\R[\G]$-module is called \emph{cuspidal} if it is admissible and is annihilated by all proper parabolic restrictions (that is, by all parabolic restrictions defined by a proper parabolic subgroup of~$\G$).
\end{enumerate}
\end{definition}

 When an~$\R[\H]$-module~$\pi$ has a central character, we denote this character by~$\omega_\pi$.

\begin{theorem}[{\cite[II 2.8]{Vig96}}]
Suppose~$\R$ is an algebraically closed field and~$\pi$ is a simple smooth~$\R[\G]$-module.   Then~$\pi$ is admissible,~$\End_{\R[\G]}(\pi)=\R$, and~$\pi$ has a central character. 
\end{theorem}

Suppose~$\R$ is an algebraically closed field and~$\pi$ is a simple~$\R[\G]$-module.  The \emph{cuspidal support} (resp.~\emph{supercuspidal support}) of~$\pi$ consists of all pairs~$(\M,\rho)$, where~$\M$ is a Levi subgroup of~$\G$ and~$\rho$ is a simple cuspidal (resp.~supercuspidal)~$\R[\M]$-module, such that there is a parabolic subgroup~$\P=\M\U$ with~$\pi$ a submodule (resp.~subquotient) of~$i^\G_{\M,\P}(\rho)$.
\begin{theorem}[{\cite[II 2.20]{Vig96}}]
Suppose~$\R$ is an algebraically closed field.  Then the cuspidal support of a simple~$\R[\G]$-module is unique up to conjugacy, i.e.~any two cuspidal supports of a given simple~$\R[\G]$-module are conjugate in~$\G$.
\end{theorem}
It is interesting that the supercuspidal support of a simple~$\R[\G]$-module is not necessarily unique up to conjugacy, cf.~\cite{DatDudas} for a counterexample when~$\G=\mathrm{Sp}_8(\F)$ and~$\R=\Fl$ for $\ell$ dividing $q^2+1$. 
However, when~$\R$ has  characteristic zero, any cuspidal
irreductible representation is supercuspidal, so the supercuspidal support of a
simple~$\R[\G]$-module is unique up to conjugacy. 

%
%
\subsection{Integral representations}
\begin{definition}
Let~$\H$ be a locally profinite group, and~$\cO$ be an integral domain with field of fractions~$\K$.  An admissible~$\K[\H]$-module~$\pi$ is called~\emph{$\cO$-integral} if there exists an~$\H$-stable~$\cO$-submodule~$L$ of~$\pi$ satisfying:
\begin{enumerate}
\item the natural map~$L\otimes_{\cO}\K\rightarrow \pi$ is an isomorphism;
\item $L$ is admissible as an~$\cO[\H]$-module.
\end{enumerate}
Such an~$L$ is called an~\emph{$\cO$-lattice} in~$\pi$.
\end{definition}
If~$\cO$ is a principal ideal domain, then by \cite[Appendice C]{Vig96} a lattice in an admissible~$\K[\G]$-module of countable dimension is~$\mathscr{O}$-free.   

\begin{remark}\label{latticesincuspsremark}
If~$L$ is a lattice (a~$\Zl$-lattice) in an integral simple~$\Ql[\G]$-module, then it is realisable over a principal ideal domain: the ring of integers~$\mathscr{O}$ of a finite extension of the maximal unramified extension~$\mathbb{Q}_{\ell}^{\ur}$ of~$\mathbb{Q}_\ell$ by \cite[II 4.9, 4.10]{Vig96}, and hence the lattice~$L$ is~$\mathscr{O}$-free, so the definition of integrality is consistent with \cite[II 4.11]{Vig96}.
\end{remark}

\begin{proposition}\label{integralreps}
Let~$\pi$ be an irreducible cuspidal~$\Qbar[\G]$-module, and suppose that~$\omega_\pi$ has
finite order.  Then there exists a number field~$\K$, such that~$\pi$ is realisable
over~$\K$ and~$\cO_\K$-integral.  Moreover,
any stable~$\cO_\K$-lattice in~$\pi$  is projective as an~$\cO_\K$-module. 
\end{proposition}
See \cite[I 9.4]{Vig96} for an analogue with coefficients in a principal complete local ring.

\begin{proof}
As~$\omega_\pi$ is algebraic over~$\mathbb{Q}$,~$\pi$ is realisable over a number field~$\K$ by \cite[II 4.9]{Vig96}.   As~$\omega_{\pi}$ has finite order, it is is trivial on some cocompact subgroup~$\S$ of~$\G$.  Let~$\widetilde{v}\in \Hom_{\K}(\pi,\K)^\infty$ be a (smooth) vector in the~$\K$-contragredient of~$\pi$.   As~$\pi$ is cuspidal it is~$\Z$-compact, and the map~$\varphi:v\mapsto (g\mapsto \widetilde{v}(\pi(g)v))$ defines an embedding~$\varphi:\pi\hookrightarrow \mathcal{C}_c^\infty(\S\backslash \G,\K)$ from~$\pi$ to~$\mathcal{C}_c^\infty(\S\backslash \G,\K)$ the space of compactly supported smooth functions~$\S\backslash \G\rightarrow \K$.

We consider~$L:=\varphi(\pi)\cap \mathcal{C}_c^\infty(\S\backslash \G,\cO_\K)$ the image of~$\pi$ in the~$\cO_\K$-valued functions.  We can make any element of~$\varphi(\pi)$ integrally valued by scaling it, hence~$L$ is non-zero and hence $L\otimes \K=\varphi(\pi)$ as~$\pi$ is irreducible.  For any compact open subgroup~$\U$ of~$\G$,~$\pi^\U$ is a finite dimensional~$\K$-vector space, and hence the image of~$\pi^\U\hookrightarrow \mathcal{C}_c^\infty(\S\backslash \G/\U,\K)$ is 
contained in~$\S\backslash \mathfrak{S}/\U$ for some compact mod~$\S$ subset~$\mathfrak{S}$ of~$\G$.  Hence~$L^\U=\varphi(\pi^\U)\cap \mathcal{C}_c^\infty(\S\backslash \G,\cO_\K)$ is an~$\cO_\K$-submodule of~$\mathcal{C}_c^\infty(\S\backslash \mathfrak{S}/\U,\cO_\K)$ which is finitely generated.  Hence~$L$ is an~$\cO_\K$-lattice.

Now, let $L$ be any stable~$\cO_\K$-lattice in $\pi$, and $\U$ as above.
As~$L^\U$ is a finitely generated torsion-free~$\cO_{\K}$-submodule of the finite
dimensional~$\K$-vector space $\pi^{\U}$,
it is projective.
Hence,~$L=\lim_{\rightarrow}L^{\U}$ is a colimit of projective~$\cO_{\K}$-modules, and we
can take the colimit over a system of compact open pro-$p$ subgroups of~$\G$,
thus~$\Hom_{\cO_{\K}}(L,-)=\lim_{\leftarrow}\Hom_{\cO_{\K}}(L^{\U},-)$ is a limit over
split surjections and~$L$ is projective. 

\end{proof}

Suppose~$\pi$ is a finite length integral~$\Ql[\G]$-module and~$L$ is a~$\Zl$-lattice in~$\pi$.  Then by the \emph{Brauer--Nesbitt principle} of Vign\'eras,~$L\otimes_{\Zl}\Fl$ has finite length as an~$\Fl[\G]$-module and its semisimplification is independent of the choice of~$L$.  We write~$r_{\ell}(\pi)$ for its semisimplification and call it the \emph{reduction modulo~$\ell$ of~$\pi$}.  Given an admissible~$\Zl[\G]$-module~$L$, we say that~$L$ \emph{lifts} the~$\Fl[\G]$-module~$L\otimes_{\Zl} \Fl$.

It is quite easy to see that a supercuspidal simple $\Ql[\G]$-module $\pi$ is integral if and only if its central character~$\omega_{\pi}$ takes values in~$\Zl^\times$, see \cite[II 4.12]{Vig96}.  A consequence of second-adjointness (Theorem \ref{secadj}), is that the Jacquet functor preserves admissibility over~$\R$.  This, and the easier admissibility of parabolic induction, shows that a simple~$\Ql[\G]$-module~$\pi$ is integral if and only if its supercuspidal support is integral \cite[Corollary 1.6]{DHKMfiniteness}.

\begin{proposition}\label{supercuspidalsappear}
Let~$\overline{\pi}$ be a simple supercuspidal~$\Fl[\G]$-module  Then there exists a simple integral cuspidal~$\Ql[\G]$-module~$\pi$ such that~$r_{\ell}(\pi)$ contains~$\overline{\pi}$ as a subquotient.
\end{proposition}

\begin{proof}
Using the level decomposition, we can find a finitely generated projective~$\Zl[\G]$-module~$\Pi$ which surjects onto~$\overline{\pi}$ for some compact open subgroup~$\H$ of~$\G$.  By \cite[Lemma 3.4]{DHKMfiniteness}, we can embed~$\Pi$ into a product~$\prod_{(\P,\rho)} I^{\G}_{\M,\P}(\rho)$ of representations parabolically induced from finitely generated~$\ell$-torsion free cuspidal~$\Zl$-representations. Thus,~$\overline{\pi}$ is a subquotient of (at least) one of the~$I_{\M,\P}^{\G}(\rho)$.  Thus,~$\overline{\pi}$ is a subquotient of~$I_{\M,\P}^{\G}(\rho\otimes \Fl)$, and hence of some parabolic induction from a simple~$\Fl$-subquotient of~$\rho\otimes \Fl$ by the main result of \cite{DatDudas} (note that second adjointness is a hypothesis of this theorem so we are implicitly using deep techniques when we are out of the range of type theoretic constructions, cf.~Remark \ref{remarkFSdependence}).  Hence, by supercuspidality of~$\pi$, we have~$\M=\P=\G$.   Thus,~$\pi$ is a subquotient of the reduction of~$\rho$, and hence is a subquotient of the reduction modulo~$\ell$ of an irreducible integral~$\Ql$-representation.  
\end{proof}

\subsection{The Bernstein decomposition and centre}
We first introduce some notation:
\begin{definition}\label{def_inertial}
\begin{enumerate}
\item Let~$\M,\M'$ be Levi subgroups of~$\G$, and~$\rho,\rho'$ simple cuspidal~$\R$-representations of~$\M,\M'$ respectively.  We say that~$(\M,\rho)$ and~$(\M',\rho')$ are \emph{inertially equivalent (over~$\G$)} if there exists~$g\in\G$ and~$\chi\in\Psi_{\M'}(\R)$ such that
\[\M'=\M^g,\qquad \rho'=\chi\otimes\rho^g.\]
We write~$[\M,\rho]_\G$ for the inertial equivalence class (over~$\G$) of~$(\M,\rho)$.
\item Let~$\mathfrak{B}_\R(\G)$ denote the set of inertial equivalence classes of pairs~$(\M,\rho)$ consisting of a Levi subgroup~$\M$ of~$\G$ and a simple supercuspidal~$\R[\M]$-module~$\rho$.
\item We denote by~$\chi_{\univ,\M,\R}:\M\rightarrow \R[\M/\M^\circ]^\times$ the \emph{universal unramified~$\R$-valued character} of the Levi subgroup~$\M$.  In settings where the ring~$\R$ is clear, we write this more simply as~$\chi_{\univ,\M}$.
\end{enumerate}
\end{definition}

Following \cite{BD}, we recover the \emph{Bernstein decomposition} and an explicit description of the \emph{Bernstein centre}: 
\begin{theorem}\label{BernsteinDeligne}
Suppose~$\Kbar$ is an algebraically closed field of characteristic zero.\begin{enumerate}
\item The category~$\Rep_{\Kbar}(\G)$ decomposes as an infinite product 
\[\Rep_{\Kbar}(\G)=\prod_{[\M,\rho]_\G\in\mathfrak{B}_{\Kbar}(\G)}\Rep_\Kbar(\G)_{[\M,\rho]_\G}\]
of indecomposable full abelian subcategories $\Rep_\Kbar(\G)_{[\M,\rho]_\G}$ where a~$\Kbar[\G]$-module is an object of~$\Rep_\Kbar(\G)_{[\M,\rho]_\G}$ if and only if all of its simple subquotients have supercuspidal support in the inertial class~$[\M,\rho]_\G$.
\item \label{Bernstein2} Choose~$(\M,\rho)\in[\M,\rho]_\G$, then~$i_{\M,\P}^{\G}(\rho\otimes \chi_{\univ,\M})$ is a progenerator of~$\Rep_\Kbar(\G)_{[\M,\rho]_\G}$, and its isomorphism class in~$\Rep_\Kbar(\G)_{[\M,\rho]_\G}$ is independent of the choice of representative for the inertial class.
\item For~$[\M,\rho]_\G\in\mathfrak{B}_{\Kbar}(\G)$ with representative~$(\M,\rho)$, the subgroup~$\H_{(\M,\rho)}=\stab_{\Psi_\M(\Kbar)}(\rho)$ is finite and depends only on the inertial class~$[\M,\rho]_\M$.  Moreover, the subgroup
\[\W_{(\M,\rho)}=\{g\in \G: \M^g=\M,\text{ and  }(\M,\rho^g)\in[\M,\rho]_{\M}\}/\M\]
depends only on the inertial class~$[\M,\rho]_\M$.  
\item Let~$e_{[\M,\rho]}$ denote the idempotent of~$\mathfrak{Z}_{\Kbar}(\G)$ corresponding to the factor~$\Rep_{\Kbar}(\G)_{[\M,\rho]_\G}$ and set~$\mathfrak{Z}_{\G,\Kbar,[\M,\rho]_\G}=e_{[\M,\rho]_\G}\mathfrak{Z}_{\G,\Kbar}$.  The space of representations of~$\M$ inertially equivalent to~$\rho$ is a torsor for the torus~$\Psi_\M(\Kbar)/\H_{(\M,\rho)}=\Spm(\Kbar[\M/\M^\circ]^{\H_{(\M,\rho)}})$ with an action of the finite group~$\W_{(\M,\rho)}$, and the centre~$\mathfrak{Z}_{\G,\Kbar,[\M,\rho]_\G}$ is the ring of~$\W_{(\M,\rho)}$-invariant regular functions on this torus.  In particular, a choice of~$(\M,\rho)\in[\M,\rho]_\G$ yields an isomorphism
\[\mathfrak{Z}_{\G,\Kbar,[\M,\rho]_\G}\simeq (\Kbar[\M/\M^\circ]^{\H_{(\M,\rho)}})^{\W_{(\M,\rho)}}.\]
\end{enumerate}
\end{theorem}

\subsection{Supercuspidal representations in banal characteristics}
\begin{definition}
\begin{enumerate}
\item A prime~$\ell$ is called \emph{banal for~$\G$} if it does not divide the pro-order of any compact open subgroup of~$\G$. 
\item We write~$\N_{\G}$ for the product of all primes which are non-banal for~$\G$.
\end{enumerate}
\end{definition} 

\begin{remark}
If~$\G$ is unramified then it has an integral model~$\underline{\G}$ as a reductive group over~$\mathscr{O}_\F$ and~$\N_\G$ is equal to the product of the primes dividing~$|\underline{\G}(k_\F)|$ and~$p$ by \cite[Lemma 5.22]{DHKM}.
\end{remark}

In banal characteristics, the~$\ell$-modular representation theory of~$\M$ is expected to resemble the complex representation theory of~$\M$.   The goal of this section is to establish a description of~$\mathfrak{Z}_{\overline{\mathbb{Z}}[1/\N_\G]}(\G)$ akin to Bernstein's description of~$\mathfrak{Z}_{\mathbb{C}}(\G)$.  For this we need basic results on cuspidal representations in banal characteristics.  Using sheaves on the building, Vign\'eras shows:
\begin{theorem}[{\cite[Theorem p373]{VigCohomologyofsheavesonthebuilding}}]\label{Vigbanalproj}
Suppose~$\ell$ is banal for~$\M$. A simple cuspidal~$\Fl[\M]$-module is~$\Z_{\M}$-projective. 
\end{theorem}
It follows from this that for~$\ell$ banal:
\begin{enumerate}
\item Any simple cuspidal~$\Fl[\M]$-module is supercuspidal.  In particular, the supercuspidal support of a simple~$\Fl[\M]$-module is unique up to conjugacy.  
\item\label{banalnoextensions}There are no non-trivial extensions between cuspidal representations with the same central characters.\end{enumerate}

We will need some reduction and lifting results we first prove in the local setup:
\begin{proposition}\label{localpropositionlifting}
\begin{enumerate}
\item
 Let~$\pi$ be a simple integral cuspidal~$\Ql[\G]$-module with~$\ell$ banal.  Then~$r_{\ell}(\pi)$ is irreducible and cuspidal.
\item Let~$\overline{\pi}$ be a simple cuspidal~$\Fl[\G]$-module with~$\ell$ banal.  Then there exists a simple integral cuspidal~$\Ql[\G]$-module~$\pi$ with~$r_{\ell}(\pi)=\overline{\pi}$.
\end{enumerate}
\end{proposition}

\begin{proof}
\begin{enumerate}
\item  We may reduce from $\Ql$ to a discretely valued field $\E$ of definition of $\pi$, with uniformizer~$\varpi_\E$.  Pick an invariant lattice $L_0$ in $\pi$, and an irreducible quotient  $L_0/ \varpi_{\E} L_0 \rightarrow \overline{\pi} $.  Put \[L_1:= L_0 + \varpi_{\E}^{-1}\ker(L_0 \rightarrow \overline{\pi})\]
This is a new invariant lattice that contains $L_0$. Moreover, the map 
$L_0/ \varpi_\E L_0  \rightarrow L_1/ \varpi_\E L_1  $
factors  as   
\[L_0/ \varpi_\E L_0  \rightarrow  \overline{\pi}  \hookrightarrow L_1/ \varpi_\E L_1 .\]
Now, we know that $L_1/ \varpi_\E L_1$ is semisimple because it is cuspidal with central character (by (\ref{banalnoextensions}) above as~$\ell$ is banal), so there is a retraction
$L_1 \rightarrow \overline{\pi}$
whose composition with the inclusion $L_0 \hookrightarrow L_1$  is the map $L_0 \rightarrow \overline{\pi} $ we started with.

Now apply this construction inductively to get an increasing sequence of lattices  $L_n$, $n \in \mathbb{N}$ equipped with compatible surjections onto $\overline{\pi}$. The colimit $L_\infty$  of this sequence is an $\cO_E[\G]$-submodule of $\pi$.  Its divisible part  is either $0$ or $\pi$, since $\pi$ is irreducible.  It can't be $\pi$ since it surjects  to $\overline{\pi}$. So it is $0$, which means that the sequence becomes constant after $r>>0$.  But then, $L_{r+1}=L_r$ means that  
$L_r/ \varpi_{\E} L_r  = \overline{\pi}$ which in turn means that $\pi$ has irreducible reduction.

\item By Proposition \ref{supercuspidalsappear},~$\overline{\pi}$ appears in the reduction modulo~$\ell$ of an integral supercuspidal~$\Ql[\G]$-module, and by the first part, such a reduction is irreducible in the banal case.
\end{enumerate}
\end{proof}

With global coefficients, we have:
\begin{proposition}\label{mainliftingreductionprop}
 Let~$\K$ be a number field, and~$L$ be a lattice in an~$\cO_\K$-integral absolutely irreducible cuspidal~$\K[\G]$-module~$\pi$.  
\begin{enumerate} 
\item For all maximal ideals~$\mathfrak{m}$ of~$\cO_\K$ not dividing $\N_\G$, the~$(\cO_\K/\mathfrak{m})[\G]$-module~$L/\mathfrak{m}L$ is absolutely irreducible and cuspidal.
\item The $\cO_{\K}[1/N_{\G}]$-lattice~$L[1/N_{\G}]$ is unique up to multiplication by
  some fractional ideal of $\cO_\K[1/N_{\G}]$. 
\item Let~$\overline{\pi}$ be a simple cuspidal~$\Fl[\G]$-module with~$\ell$ banal.  Then there exists a number field~$\K$, and an absolutely irreducible integral cuspidal~$\K[\G]$-module~$\pi$ with~$\cO_{\K}$-lattice~$L$, and a maximal ideal~$\mathfrak{m}$ of~$\cO_{\K}$ above~$\ell$, such that~$\overline{\pi}=(L/\mathfrak{m}L)\otimes\Fl$.
\end{enumerate}
\end{proposition}
\begin{proof}
\begin{enumerate}
\item As localization is an exact functor,~$L/\mathfrak{m}L\simeq L_{\mathfrak{m}}/\mathfrak{m}L_{\mathfrak{m}}$, it follows from Proposition \ref{localpropositionlifting}.
\item Let~$L'$ be another~$\cO_{\K}$-lattice in~$\pi$, and choose a compact open
  subgroup~$\U$ such that~$L$ and~$L'$ are generated by their~$\U$-invariants.
  Then~$L^\U,(L')^{\U},$ and~$(L+L')^{\U}$ are~$\cO_{\K}$-lattices in a finite
  dimensional~$\K$-vector space, and hence~$(L+L')^{\U}/L^{\U}$
  and~$(L+L')^{\U}/(L')^{\U}$ are finitely generated torsion~$\cO_{\K}$-modules, and are
  thus supported on a finite set of maximal ideals.  Therefore, for all but finitely many
  maximal ideals,~$L_{\mathfrak{m}}=L'_{\mathfrak{m}}$.  Moreover, from the first part,
  for all maximal ideals not dividing $N_{\G}$, the local lattices~$L_{\mathfrak{m}}$ are
  unique up to homothety.  Therefore there exists a collection~$(\alpha_{\mathfrak{m}})$
  with all but finitely many units, such
  that~$\alpha_{\mathfrak{m}}L_{\mathfrak{m}}=L'_{\mathfrak{m}}$. Therefore,~$\alpha
  L[1/N_{\G}]=L'[1/N_{\G}]$ for some fractional ideal~$\alpha$ of~$\cO_{\K}[1/N_{\G}]$.
\item By Proposition \ref{localpropositionlifting}, there exists a simple cuspidal~$\Ql[\G]$-module~$\pi_{\Ql}$ lifting~$\overline{\pi}$.  By twisting by an unramified character of~$\G$ if necessary, we can assume that the central character of~$\pi_{\Ql}$ has finite order, and that~$\pi$ is defined over~$\Qbar$.  Hence, by Proposition \ref{integralreps},~$\pi$ is defined over a number field~$\K$ and~$\cO_{\K}$-integral with~$\cO_\K$-lattice~$L$.  Moreover, as localization is an exact functor,~$(\L/\mathfrak{m} L)  \otimes\Fl=\overline{\pi}$ for any maximal ideal~$\mathfrak{m}$ above~$\ell$.
\end{enumerate}
\end{proof}

Finally, we need to show that the lattices we consider are projective on restriction to~$\G^\circ$ in the banal setting:

\begin{proposition}\label{banalcuspidalZproj}
Let~$\K$ be a number field, and~$L$ be a lattice in an~$\cO_\K$-integral absolutely irreducible cuspidal~$\K[\G]$-module~$\pi$.    The restriction $(\L\otimes \cO_\K[1/\N_{\G}])_{|\G^\circ}$ is~projective in~$\Rep_{\cO_\K[1/\N_\G]}(\G^\circ)$.
\end{proposition}

\begin{proof}
By \cite[3.11]{Meyer},~$L_{|\G^\circ}$ is ``c-projective'', meaning
that for any open compact subgroup $K$ of $\G$, taking $\Hom_{\cO_{\K}[1/p][\G^{\circ}]}(L,-)$ is exact on~``$K$-split short exact sequences'', where a sequence~$0\rightarrow \rho_1\rightarrow \rho\rightarrow \rho_2\rightarrow 0$ of~$\cO_{\K}[1/p][\G^\circ]$-modules,
is called~\emph{$K$-split} if there is a~$K$-equivariant section~$\rho_2\rightarrow \rho$.
By the level decomposition, for example, we can find a
projective~$\cO_{\K}[1/p][\G^\circ]$-module~$P$ covering~$L$ with
projection~$\pi:P\rightarrow L$.  As~$L$ is projective as an~$\cO_{\K}[1/p]$-module by
Proposition \ref{integralreps}, we can split~$\pi$ as a~$\cO_{\K}[1/p]$-module
morphism, and for any compact open subgroup~$K$ of ~$\G$, we can average this splitting to
split~$p$ as an~$\cO_{\K}[1/\N_\G][K]$-module morphism.  By
c-projectivity~$\Hom_{\cO_{\K}[1/\N_\G][\G^\circ]}(L,P)\rightarrow
\Hom_{\cO_{\K}[1/\N_\G][\G^\circ]}(L,L)$ is surjective, hence we can split~$\pi$
over~$\cO_{\K}[1/\N_\G][\G^\circ]$ and~$L$ is projective as
a~$\cO_{\K}[1/\N_\G][\G^\circ]$-module. 
\end{proof}

\subsection{The integral centre over~$\overline{\mathbb{Z}}[1/\N_\G]$}\label{intcentre}
Let~$\K/\mathbb{Q}$ be a number field, and $\pi$ be an absolutely irreducible cuspidal~$\K[\G]$-module with finite order central character~$\omega_{\pi}$.    

As~$\G^{\circ}\Z_{\G}$ is a normal subgroup of~$\G$ with finite index and~$\pi$ has a central character we can write
\[(\pi\otimes \overline{\K})_{|\G^\circ}\simeq (\pi_1\oplus \cdots \oplus \pi_r)^{\oplus m},\]
where the~$\pi_i$ are pairwise non-isomorphic irreducible~$\mathbb{C}[\G^{\circ}]$-modules. Moreover, the isomorphism class of~$\ind_{\G^{\circ}}^{\G}(\pi_i)$ is independent of~$i$.  Enlarging~$\K$ if necessary, so that the above decomposition is valid over~$\K$, we can write
\[\pi_{|\G^\circ}\simeq (\pi_1\oplus \cdots \oplus \pi_r)^{\oplus m},\]
as a decomposition of~$\pi$ into pairwise non-isomorphic (absolutely) irreducible~$\K[\G^{\circ}]$-modules.   

Choose an irreducible~$\K[\G^{\circ}]$-subspace~$\mathcal{W}$ of~$\pi$ such that \[\mathcal{H}=\{g\in \G:\pi(g)\mathcal{W}=\mathcal{W}\}\] is maximal, and write~$\pi_{\mathcal{H}}$ for the natural representation of~$\mathcal{H}$ on~$\mathcal{W}$.  Set~$\pi^{\circ}=\Res^{\mathcal{H}}_{\G^{\circ}}(\pi_{\mathcal{H}})$.   By \cite[Lemma 8.3]{BHGen},~$\ind_{\mathcal{H}}^{\G}(\pi_{\mathcal{H}})\simeq \pi$.

Let~$\H_{\pi}$ denote the \emph{ramification group} of~$\pi$: that is, the (finite) group of unramified characters~$\chi$ of~$\G$ which satisfy~$\pi\simeq \pi\otimes\chi$, and introduce the following finite index normal subgroups of~$\G$
\begin{align*}
\mathcal{S}&=\{g\in \G:(\pi^{\circ})^g\simeq \pi^{\circ}\};\quad\mathcal{T}=\bigcap_{\chi\in \H_{\pi}}\Ker(\chi),
\end{align*}
 as in \cite{BHGen}.  By \cite[Lemma 8.3]{BHGen} we have~\[\mathcal{S}\supseteq \mathcal{H}\supseteq \mathcal{T}\supseteq\Z_\G\G^{\circ}\text{ with~}[\mathcal{H}:\mathcal{T}]=[\mathcal{S}:\mathcal{H}]=m.\]

By the same argument as in Proposition \ref{integralreps} (which only used irreducibility,~$\Z$-compactness, and finite order central character, so applies equally well to the~$\K$-representation~$\pi_{\mathcal{H}}$ of~$\mathcal{H}$), $\pi_{\mathcal{H}}$ is~$\cO_{\K}$-integral and we choose an~$\cO_{\K}$-lattice~$L_{\mathcal{H}}$ in~$\pi_{\mathcal{H}}$.  

Write~$L^{\circ}=\Res^{\mathcal{H}}_{\G^{\circ}}(L_{\mathcal{H}}),\text{ and }L_{\mathcal{T}}=\Res^{\mathcal{H}}_{\mathcal{T}}(L_{\mathcal{H}})$ so that~$L^{\circ}$ (respectively~$L_{\mathcal{T}}$) is an~$\cO_{\K}$-lattice in~$\pi^{\circ}$ (respectively~$\pi_{\mathcal{T}}=\Res_{\mathcal{T}}^{\mathcal{H}}(\pi_{\mathcal{H}})$).
 
 \begin{lemma}\label{lemmacuspend}
 Let~$\R=\cO_{\K}[1/\N_{\G}]$.  
 \begin{enumerate}
 \item The~$\R[\G^{\circ}]$-module~$L^{\circ}\otimes\R$ is finitely generated projective.
 \item The~$\R[\G]$-module~$\ind_{\G^{\circ}}^{\G}( L^{\circ}\otimes \R)$ is finitely generated projective.
 \item The natural inclusion~$\End_{\R[\mathcal{S}]}(\ind_{\G^{\circ}}^{\mathcal{S}}( L^{\circ}\otimes \R))\subseteq \End_{\R[\G]}(\ind_{\G^{\circ}}^{\G}( L^{\circ}\otimes \R))$ is an isomorphism.
\item The centre of~$\End_{\R[\G]}(\ind_{\G^{\circ}}^{\G}( L^{\circ}\otimes \R))$ is given by
\[\Z(\End_{\R[\G]}(\ind_{\G^{\circ}}^{\G}( L^{\circ}\otimes \R)))\simeq \End_{\R[\mathcal{T}]}(\ind_{\G^\circ}^{\mathcal{T}}(L^{\circ}\otimes\R))\simeq  \R[\mathcal{T}/\G^{\circ}]\simeq \R[\G/\G^\circ]^{\H_{\pi}}.\]
\end{enumerate}
\end{lemma}
\begin{proof}
The representation~$\ind_{\mathcal{H}}^{\G}(L_{\mathcal{H}})$ is a~$\G$-stable lattice in~$\ind_{\mathcal{H}}^{\G}(\pi_{\mathcal{H}})\simeq \pi$, hence
\[\Res^{\G}_{\G^{\circ}}(\ind_{\mathcal{H}}^{\G}(L_{\mathcal{H}}))\otimes\R\simeq \bigoplus_{\G/\mathcal{H}} (\Res^{\mathcal{H}}_{\G^{\circ}}(L_{\mathcal{H}})^g \otimes\R )\]
 is projective by Proposition \ref{banalcuspidalZproj}.  Hence, as it is a summand,~$L^{\circ}\otimes\R $ is projective.  As~$\G^\circ$ is open in~$\G$,~$\ind_{\G^\circ}^{\G}$ is left adjoint to an exact functor (restriction), hence~$\ind_{\G^{\circ}}^{\G}( L^{\circ}\otimes \R)$ is a finitely generated projective~$\R[\G]$-module.
 
 By Frobenius reciprocity and Mackey Theory,
\begin{equation}
\label{eqcusp1}\End_{\R[\G]}(\ind_{\G^\circ}^{\G}(L^{\circ})\otimes\R)\simeq \bigoplus_{\G/\G^{\circ}}\Hom_{\R[\G^{\circ}]}(L^{\circ},(L^{\circ})^g).\end{equation}
Moreover, the natural map
\[\Hom_{\R[\G^{\circ}]}(L^{\circ},(L^{\circ})^g)\otimes \K\simeq \Hom_{\K[\G^{\circ}]}(\pi^{\circ},(\pi^{\circ})^g),\]
is an isomorphism by Lemma \ref{scalarextandpros}.  Hence the sum in Equation \ref{eqcusp1} is supported on the cosets~$\mathcal{S}/\G^{\circ}$.  Hence we have an isomorphism of algebras, given by functoriality of compact induction, 
\[\End_{\R[\mathcal{S}]}(\ind_{\G^\circ}^{\mathcal{S}}(L^{\circ}\otimes\R))\simeq \End_{\R[\G]}(\ind_{\G^\circ}^{\G}(L^{\circ}\otimes\R)).\]


We next compute the endomorphism algebra~$\End_{\R[\mathcal{T}]}(\ind_{\G^\circ}^{\mathcal{T}}(L^{\circ}\otimes\R))$.  As~$\Res_{\G^{\circ}}^{\mathcal{T}}(L_{\mathcal{T}})=L^{\circ}$, we have
\[\ind_{\G^\circ}^{\mathcal{T}}(L^{\circ})\simeq L_{\mathcal{T}}\otimes \ind_{\G^\circ}^{\mathcal{T}}(1)\] where~$\mathcal{T}$ acts diagonally on the tensor product.  This decomposition induces a morphism of algebras
\begin{equation}
\label{eqcusp2} \End_{\R[\mathcal{T}]}(L_{\mathcal{T}}\otimes\R)\otimes \R[\mathcal{T}/\G^{\circ}]\rightarrow \End_{\R[\mathcal{T}]}(\ind_{\G^\circ}^{\mathcal{T}}(L^{\circ}\otimes\R)).\end{equation}
By Frobenius reciprocity, we have an isomorphism of~$\R$-modules
\begin{align*}
\End_{\R[\mathcal{T}]}(\ind_{\G^\circ}^{\mathcal{T}}(L^{\circ}\otimes\R))&\simeq \Hom_{\R[\G^{\circ}]}(L^{\circ}\otimes\R,(L^{\circ}\otimes\R)\otimes \Res_{\G^{\circ}}^{\mathcal{T}}( \ind_{\G^\circ}^{\mathcal{T}}(1)\otimes\R) )\\&\simeq \End_{\R[\G^{\circ}]}(L^{\circ}\otimes \R)\otimes\R[\mathcal{T}/\G^{\circ}]\end{align*}
and the morphism of Equation \ref{eqcusp2}, is induced by the inclusion~$\End_{\R[\mathcal{T}]}(L_{\mathcal{T}}\otimes\R)\subseteq \End_{\R[\G^{\circ}]}(L^{\circ}\otimes \R)$.  As~$L^{\circ},~L_{\mathcal{T}}$ are~$\cO_{\K}$-lattices, they are locally free (cf.~Remark \ref{latticesincuspsremark}), hence~$\End_{\R[\mathcal{T}]}(L_{\mathcal{T}}\otimes\R),~\End_{\R[\G^{\circ}]}(L^{\circ}\otimes\R)$ are locally free.  Moreover,
\[\End_{\R[\mathcal{T}]}(L_{\mathcal{T}}\otimes\R)\otimes \K\simeq \End_{\K[\mathcal{T}]}(\pi_{\mathcal{T}})\simeq\K,\quad \End_{\R[\G^{\circ}]}(L^{\circ}\otimes\R)\otimes \K\simeq \End_{\K[\G^{\circ}]}(\pi^{\circ})\simeq\K,\] 
and hence~$\End_{\R[\mathcal{T}]}(L_{\mathcal{T}}\otimes\R)\simeq\End_{\R[\G^{\circ}]}(L^{\circ}\otimes\R)\simeq \R$ (as~$\R$ is a normal domain), 
and hence we have an isomorphism of algebras
\[  \End_{\R[\mathcal{T}]}(\ind_{\G^\circ}^{\mathcal{T}}(L^{\circ}\otimes\R))\simeq \R[\mathcal{T}/\G^{\circ}].\]
%
%
%

It remains to show this is~$\Z(\End_{\R[\mathcal{S}]}(\ind_{\G^{\circ}}^{\mathcal{S}}( L^{\circ}\otimes \R)))$.  
Now, by \cite[Lemma 8.4]{BHGen},
\[\Z(\End_{\K[\mathcal{S}]}(\ind_{\G^{\circ}}^{\mathcal{S}}( \pi^{\circ})))\simeq\End_{\K[\mathcal{T}]}(\ind_{\G^\circ}^{\mathcal{T}}(\pi^{\circ})) \simeq \K[\mathcal{T}/\G^{\circ}]\]  and hence~
\begin{equation}
\label{inclusion}\End_{\R[\mathcal{T}]}(\ind_{\G^\circ}^{\mathcal{T}}(L^{\circ}\otimes\R))\subseteq\Z(\End_{\R[\mathcal{S}]}(\ind_{\G^{\circ}}^{\mathcal{S}}( L^{\circ}\otimes \R))).\end{equation}  Moreover,
\[\Z(\End_{\R[\mathcal{S}]}(\ind_{\G^{\circ}}^{\mathcal{S}}( L^{\circ}\otimes \R)))\otimes \K=\Z(\End_{\K[\mathcal{S}]}(\ind_{\G^{\circ}}^{\mathcal{S}}( \pi^{\circ})))\simeq \K[\mathcal{T}/\G^{\circ}]\]
by Lemmas \ref{scalarextandpros} and \ref{centrelemma}.  Hence, as~$\R[\mathcal{T}/\G^{\circ}]$ is a normal Noetherian domain and~$\End_{\R[\mathcal{S}]}(\ind_{\G^{\circ}}^{\mathcal{S}}( L^{\circ}\otimes \R))$ (hence~$\Z(\End_{\R[\mathcal{S}]}(\ind_{\G^{\circ}}^{\mathcal{S}}( L^{\circ}\otimes \R)))$) a finite~$\R[\mathcal{T}/\G^{\circ}]$-module (as~$[\mathcal{S}:\mathcal{T}]=m^2$ is finite), the inclusion of Equation \ref{inclusion} is an isomorphism.
\end{proof}

Let~$[\M,\rho]_{\G}\in\mathfrak{B}_{\Qbar}(\G)$ be a~$\Qbar$-inertial class.  Choose a representative~$(\M,\rho)$ such that~$\rho$ has finite order central character, and is hence defined over a number field~$\K/\mathbb{Q}$ by Proposition \ref{integralreps}.  By extending scalars if necessary, we assume that~$\K$ contains~$\sqrt{q}$ and is sufficiently large for the decomposition of~$\rho_{|\M^{\circ}}$.  

The finite group~$\H_{\rho}$ only depends on the~$\M$-inertial class of~$\rho$ and we write~$\H_{(\M,\rho)}$ as before.  Recall, from Bernstein's decomposition, we put
\[\W_{(\M,\rho)}=\{g\in \G: \M^g=\M,\text{ and  }(\M,\rho^g)\in[\M,\rho]_{\M}\}/\M.\]

As in the cuspidal case above, we choose our lattice~$L^{\circ}$ in~$\rho^{\circ}$, where~$\rho^{\circ}$ is an absolutely irreducible~$\K[\M^{\circ}]$-submodule of~$\rho_{|\M^{\circ}}$.  Choose a parabolic subgroup~$\P$ of~$\G$ with Levi factor~$\M$.  We set~$\R=\cO_\K[1/\N_\G]$ and let
\[P_{(\M,\rho)}=i_{\M, \P}^\G(\ind_{\M^{\circ}}^{\M}(L^{\circ}\otimes\R)).\]

\begin{lemma}\label{lemmacentreprojK}
The~$\K[\G]$-module~$P_{(\M,\rho)}\otimes\K\simeq i_{\M, \P}^\G(\ind_{\M^{\circ}}^{\M}(\rho^{\circ}))$ is finitely generated projective, with
\[\Z(\End_{\K[\G]}(P_{(\M,\rho)}\otimes\K))\simeq (\K[\M/\M^\circ]^{\H_{(\M,\rho)}})^{\W_{(\M,\rho)}}.\]
\end{lemma}
\begin{proof}
As~$\Kbar/\K$ is faithfully flat, this follows from Bernstein's Theorem \ref{BernsteinDeligne} and from Lemma \ref{centrelemma}.
\end{proof}

We now prove an integral form of this after inverting~$\N_\G$:

\begin{lemma}\label{lemmacentreproj2}
Let~$\R=\cO_\K[1/\N_\G]$.  The~$\R[\G]$-module~$P_{(\M,\rho)}=i_{\M, \P}^\G(\ind_{\M^{\circ}}^{\M}(L^{\circ}\otimes\R))$
is finitely generated projective, with
\[\Z(\End_{\R[\G]}(P_{(\M,\rho)}))\simeq (\R[\M/\M^\circ]^{\H_{(\M,\rho)}})^{\W_{(\M,\rho)}}.\]
\end{lemma}

\begin{proof}
By Lemma \ref{lemmacuspend}, second adjunction and finiteness of parabolic induction \cite{DHKMfiniteness}, 
~$P_{(\M,\rho)}$ is~finitely generated projective.
%

As~$P_{(\M,\rho)}$ is finitely generated projective, we have an embedding 
\begin{equation}\label{Equation1}
\End_{\R[\G]}(P_{(\M,\rho)})\hookrightarrow \End_{\R[\G]}(P_{(\M,\rho)})\otimes \K \simeq \End_{\K[\G]}(P_{(\M,\rho)}\otimes\K).\end{equation}
Moreover, the isomorphism $(\K[\M/\M^\circ]^{\H_{(\M,\rho)}})^{\W_{(\M,\rho)}}
\xrightarrow{\sim}\Z(\End_{\K[\G]}(\P_{(\M,\rho)}\otimes\K))$ of Lemma
\ref{lemmacentreprojK} is induced by the following composed map,
\[(\K[\M/\M^\circ]^{\H_{(\M,\rho)}})^{\W_{(\M,\rho)}}\hookrightarrow \End_{\K[\M]}(\ind_{\M^{\circ}}^{\M}(\rho^{\circ}))\hookrightarrow \End_{\K[\G]}(P_{(\M,\rho)}\otimes\K),\]
where the second map is given by functoriality for parabolic induction.  
Therefore, the analogous map over~$\R$ 
\[(\R[\M/\M^\circ]^{\H_{(\M,\rho)}})^{\W_{(\M,\rho)}}\hookrightarrow\End_{\R[\M]}(\ind_{\M^{\circ}}^{\M}(L^{\circ}\otimes\R))\hookrightarrow \End_{\R[\G]}(P_{(\M,\rho)}),\]
 also given by functoriality of parabolic induction, certainly induces an embedding
\[(\R[\M/\M^\circ]^{\H_{(\M,\rho)}})^{\W_{(\M,\rho)}}\hookrightarrow  \Z(\End_{\R[\G]}(P_{(\M,\rho)})).\]
To see that this embedding is an isomorphism, it then suffices to see that the target is
integral over the source, since the source is integrally closed. In turn, it is enough to
show that $\End_{\R[\G]}(P_{(\M,\rho)})$ is finitely generated as a module over
$(\R[\M/\M^\circ]^{\H_{(\M,\rho)}})^{\W_{(\M,\rho)}}$, or even over
$\R[\M/\M^\circ]^{\H_{(\M,\rho)}}$. We have already seen that
$\End_{\R[\M]}(\ind_{\M^{\circ}}^{\M}(L^{\circ}\otimes\R))$ is finitely generated as a
module over $\R[\M/\M^\circ]^{\H_{(\M,\rho)}}$. On the other hand, the Geometric Lemma
implies that $\End_{\R[\G]}(P_{(\M,\rho)})$ is finitely generated as a module over 
$\End_{\R[\M]}(\ind_{\M^{\circ}}^{\M}(L^{\circ}\otimes\R))$.
\end{proof}

For every~$\Qbar$-inertial class of supercuspidal support~$[\M,\rho]_\G$~for~$\G$, we follow the same construction, choosing representatives~$(\M,\rho)$ with finite order central characters, and defining a collection of projective modules
\begin{equation}\label{Integralprog}
P_{(\M,\rho)}=i_{\M,\P_{\M}}^{\G}(\ind_{\M^{\circ}}^{\M}(L_{\rho}^{\circ})\otimes \cO_{\K_\rho}[1/\N_{\G}]),
\end{equation}
where the number field~$\K_{\rho}$ and the lattice~$L_{\rho}^{\circ}$ depend on~$\rho$, and the parabolic~$\P_{\M}$ on~$\M$.  We choose our parabolic subgroups compatibility so that~$\P_{\M^g}=\P_{\M}^g$.  To consider all of these over the same base ring we extend scalars to~$\Zbar[1/\N_{\G}]$, where we have:

\begin{lemma}
The~$\overline{\mathbb{Z}}[1/\N_\G][\G]$-modules~$P_{(\M,\rho)}\otimes \overline{\mathbb{Z}}[1/\N_\G]$, are
\begin{enumerate}
\item \emph{finitely generated projective};
\item \emph{mutually disjoint}, i.e.~if~$[\M,\rho]_\G$ and~$[\M',\rho']_\G$ define different classes in~$\mathfrak{B}_{\Qbar}(\G)$ then
\[\Hom_{\overline{\mathbb{Z}}[1/\N_\G][\G]}(P_{(\M,\rho)}\otimes \overline{\mathbb{Z}}[1/\N_\G],P_{(\M',\rho')}\otimes \overline{\mathbb{Z}}[1/\N_\G])=0.\]
\item \emph{exhaust} the category, i.e.~for any~$\overline{\mathbb{Z}}[1/\N_\G][\G]$-module~$\pi$ there exists an inertial class~$[\M,\rho]_\G\in\mathfrak{B}_{\Qbar}(\G)$ with
\[\Hom_{\overline{\mathbb{Z}}[1/\N_\G][\G]}(P_{(\M,\rho)}\otimes \overline{\mathbb{Z}}[1/\N_\G],\pi)\neq 0.\] \end{enumerate}
\end{lemma}

\begin{proof}
\begin{enumerate}
\item It follows from Lemma \ref{lemmacentreproj2} that the~$P_{(\M,\rho)}\otimes \overline{\mathbb{Z}}[1/\N_\G]$ are finitely generated projective.
\item Moreover, because~$P_{(\M,\rho)}\otimes \overline{\mathbb{Z}}[1/\N_\G]$ is finitely generated projective it is torsion free (or by Lemma \ref{scalarextandpros}), we have
\[\Hom_{\overline{\mathbb{Z}}[1/\N_\G][\G]}(P_{(\M,\rho)}\otimes \overline{\mathbb{Z}}[1/\N_\G],P_{(\M',\rho')}\otimes \overline{\mathbb{Z}}[1/\N_\G])\hookrightarrow\Hom_{\mathbb{C}[\G]}(P_{(\M,\rho)}\otimes \mathbb{C},P_{(\M',\rho')}\otimes \mathbb{C})\]
which is zero by the Bernstein decomposition \ref{BernsteinDeligne}.
\item As~$P_{(\M,\rho)}\otimes \overline{\mathbb{Z}}[1/\N_\G]$ is projective and every~$\overline{\mathbb{Z}}[1/\N_\G][\G]$-module has an irreducible subquotient, we can reduce to the case where~$\pi$ is irreducible.

First assume that~$\pi$ is torsion free, then by Lemma \ref{scalarextandpros} 
\[\Hom_{\overline{\mathbb{Z}}[1/\N_\G][\G]}(P_{(\M,\rho)}\otimes \overline{\mathbb{Z}}[1/\N_\G],\pi)\otimes\mathbb{C}\simeq\Hom _{\mathbb{C}[\G]}(P_{(\M,\rho)}\otimes \mathbb{C},\pi\otimes\mathbb{C}),\]
so~$\Hom_{\overline{\mathbb{Z}}[1/\N_\G][\G]}(P_{(\M,\rho)}\otimes \overline{\mathbb{Z}}[1/\N_\G],\pi)$ is non-zero if and only if~$\pi\otimes\mathbb{C}\in\Rep_{\mathbb{C}}(\G)_{[\M,\rho]_{\G}}$, and exhaustion follows as~$P_{(\M,\rho)}\otimes \mathbb{C}$ exhaust over~$\mathbb{C}$ from Bernstein's decomposition \ref{BernsteinDeligne}.  

Now assume that there exists a banal prime~$\ell$ such that~$\ell\pi=0$.  Hence~$\pi$ identifies with a simple~$\Fl[\G]$-module.  There exists a parabolic~$\P_{\M}=\M\N$ and an irreducible cuspidal~$\Fl[\M]$-module~$\tau$ such that~$\pi$ is a quotient of~$i_{\M,\P_{\M}}^{\G}(\tau)$.    By Proposition \ref{mainliftingreductionprop}, there exists a number field~$\K$ and an absolutely irreducible integral cuspidal~$\widetilde{\tau}$ with reduction~$r_{\ell}(\widetilde{\tau})\otimes\Fl\simeq \tau$.  There exists $(\M^g,\eta)\in[\M,\widetilde{\tau}]_{\G}$ in our chosen collection of representatives for the~$\Qbar$-inertial classes of~$\G$, with associated projective
\[P_{(\M^g,\eta)}=i_{\M^g,\P_{\M^g}}^{\G}(\ind_{\M_{\eta}^{\circ}}^{\M}(L_\eta^{\circ})\otimes\cO_{\K_{\eta}}[1/\N_{\G}]).\] 
Moreover enlarging~$\K_{\eta}$ if necessary, so that it contains~$\K$, the~$\K_{\eta}[\G]$-module~$\widetilde{\tau}\otimes\K_{\eta}$ is a quotient of~$\ind_{\M_{\eta}^{\circ}}^{\M}(L_\eta^{\circ})\otimes\K_{\eta}$.  The image of~$(\ind_{\M^{\circ}}^{\M}(L_\eta^{\circ}))^{g^{-1}}\otimes\cO_{\K_{\eta}}$ in~$\widetilde{\tau}\otimes\K_{\eta}$ is an~$\cO_{\K_{\eta}}$-lattice~$L_{\tau}$ in~$\widetilde{\tau}$ (\cite[I 9.3]{Vig96}).  Hence~$i_{\M,\P_{\M}}^{\G}(\tau)$ and hence~$\pi$ are quotients of~$P_{(\M^g,\eta)}$.
\end{enumerate}
\end{proof}
Hence, we obtain a decomposition of the category and from the last section a description of the centre of each factor.  Moreover, this is the finest such decomposition - the \emph{block decomposition} over~$\overline{\mathbb{Z}}[1/\N_\G]$ -- the centres of the categories (indexed by~$\Qbar$-inertial classes of supercuspidal supports) have no non-trivial idempotents.  We record this as a theorem:

\begin{theorem}\label{banaldecomptheorem}
\begin{enumerate}\item The category~$\Rep_{\overline{\mathbb{Z}}[1/\N_\G]}(\G)$ decomposes as
\[\Rep_{\overline{\mathbb{Z}}[1/\N_\G]}(\G)=\prod_{[\M,\rho]_\G\in\mathfrak{B}_{\Qbar}(\G)}
  \Rep_{\overline{\mathbb{Z}}[1/\N_\G]}(\G)_{[\M,\rho]_{\G}},\]
where~$\Rep_{\overline{\mathbb{Z}}[1/\N_\G]}(\G)_{[\M,\rho]_{\G}}$ is the direct factor subcategory generated by the finitely generated projective~$P_{(\M,\rho)}\otimes \overline{\mathbb{Z}}[1/\N_\G]$.  
\item Moreover, the choice
of~$P_{(\M,\rho)}$ identifies the
centre~$\mathfrak{Z}_{\G,\Zbar[1/\N_{\G}],[\M,\rho]_{\G}}$
of~$\Rep_{\overline{\mathbb{Z}}[1/\N_\G]}(\G)_{[\M,\rho]_{\G}}$ with~$(\overline{\mathbb{Z}}[1/\N_\G][\M/\M^\circ]^{\H_{(\M,\rho)}})^{\W_{(\M,\rho)}}$.\end{enumerate}
\end{theorem}

\begin{remark} \label{rk_definition_block}
  The block $\Rep_{\overline{\mathbb{Z}}[1/\N_\G]}(\G)_{[\M,\rho]_{\G}}$ is ``defined over
  $\cO_{\K_{\rho}}[1/\N_{\G}]$'', in the sense that $\P_{(\M,\rho)}$ generates a direct factor
  subcategory $\Rep_{\cO_{\K_{\rho}}[1/\N_\G]}(\G)_{[\M,\rho]_{\G}}$ of
  $\Rep_{\cO_{\K_{\rho}}[1/\N_\G]}(\G)$.
\end{remark}

\section{Gelfand--Graev representations and their endomorphism algebras}
For this section, we suppose that~$\mathbf{G}$ is~$\F$-quasi-split.  Choose a maximal~$\F$-split torus~$\mathbf{S}$ of~$\mathbf{G}$, and a Borel subgroup~$\mathbf{B}$ with Levi factor~$\mathbf{T}=\C_{\mathbf{G}}(\mathbf{S})$ and unipotent radical~$\mathbf{U}$.  We let~$\T=\mathbf{T}(\F)$,~$\B=\mathbf{B}(\F)$, and~$\U=\mathbf{U}(\F)$.  We write~$\Phi$ for the set of roots of~$\mathbf{S}$ in~$\mathbf{G}$,~$\Delta$ for the set of simple roots determined by~$\mathbf{B}$, and~$\Phi^+$ the set of positive roots determined by~$\Delta$.  For~$\alpha\in\Phi^+$ we let~$\mathbf{U}_{\alpha}$ denote the root subgroup corresponding to~$\alpha$, thus we have an isomorphism~$\prod_{\alpha\in\Phi^+_{\mathrm{nd}}}\mathbf{U}_\alpha\rightarrow \mathbf{U}$ of~$\F$-varieties, where~$\Phi^+_{\mathrm{nd}}$ denotes the subset of~$\Phi^+$ of non-divisible roots.

Choose in addition a pinning of $\mathbf{G}$, compatible with our choices of $\mathbf{T}$ and $\mathbf{B}$.  Equivalently we fix an isomorphism, for each absolute root $\alpha$ of $\mathbf{U}_{\overline{\F}}$, of $(\mathbf{U}_{\overline{\F}})_{\alpha}$ with the additive group over $\overline{\F}$.  

\begin{definition}
Let $\R_0$ denote the ring $\mathbb{Z}[1/p,\mu_{p^{\infty}}]$, and let~$\R$ be an~$\R_0$-algebra. 
\begin{enumerate}
\item A character~$\psi:\U\rightarrow \R^\times$ of~$\U=\mathbf{U}(\F)$ is called \emph{non-degenerate}, if~it is non-trivial on~$\U_{\alpha}=\mathbf{U}_{\alpha}(F)$, for all~$\alpha\in\Delta$.   
\item A \emph{Whittaker datum} for~$\G$ (over~$\R$) is a pair~$(\U',\psi)$ such that~$\U'$ is the~$\F$-points of the unipotent radical of a Borel subgroup of~$\mathbf{G}$ and~$\psi:\U'\rightarrow\R^\times$ is a non-degenerate character.
\item A simple~$\R[\G]$-module $(\pi,V)$ is called~\emph{$\psi$-generic}, for a
  non-degenerate character~$\psi:\U\rightarrow\R^\times$,
if the module of $(\U,\psi)$-coinvariants $V_{\psi}:=V/\langle
\pi(u)v-\psi(u)v, \, u\in\U, v\in V\rangle$ is non zero.
\end{enumerate}
\end{definition} 

From a Whittaker datum $(\U,\psi)$ over $\R$ we may construct the
smooth $\R[\G]$-module
$$\ind_{\U}^{\G}(\psi) :=\{{\rm smooth,\, compactly\, supported\, mod\, U},\, f : \G\rightarrow \R, \, \forall u\in\U, f(ug)=\psi(u)f(g) \}$$
This module depends, up to isomorphism, only on the $\G$-conjugacy class of $(\U,\psi)$.

Fix an additive character $\psi_\F: \F^+ \rightarrow \R_{0}^{\times}$, trivial on $\cO_{\K}$ but not on $\varpi^{-1} \cO_{\K}$, where $\varpi$ is a uniformizer of $\F$.  Our chosen pinning yields a $\Gal(\overline{\F}/\F)$-equivariant identification of ${\mathbf U}^{\ab}_{\overline{\F}}$ with the product $\prod\limits_{\alpha \in \Delta_{\overline{\F}}} {\mathbf G}_a$, where $\Delta_{\overline{\F}}$ denotes the set of {\em absolute} simple roots (that is, the set of positive simple roots of ${\mathbf G}_{\overline{\F}}$ determined by ${\mathbf B}_{\overline{\F}}$.) Let $\iota_{\alpha}$ denote the map ${\mathbf U} \rightarrow {\mathbf G}_a$ determined by $\alpha$ and our chosen pinning.  We can then define a character $\psi$ of $\U$ by the formula:
$$\psi(u) = \psi_{\F} \left(\sum\limits_{\alpha \in \Delta_{\overline{\F}}} \iota_{\alpha}(u) \right),$$
noting that the sum on the right hand side is an element of $\F$ since it is fixed under $\Gal(\overline{\F}/\F)$.

The character $\psi$ is non-degenerate.  Moreover, this construction gives a bijection between pinnings of $\mathbf{G}$ (for a fixed choice of $\mathbf{T}, \mathbf{B}$) and Whittaker data of the form $(\U,\psi)$.  In particular both of these sets are torsors under the conjugation action of the group $(\mathbf{T}/\mathbf{Z})(\F)$.

A choice of Whittaker datum for ${\mathbf G}$ also determines Whittaker data for Levi subgroups of ${\mathbf G}.$  Let~$\M$ be a \emph{standard} Levi subgroup of $\G$ (meaning that it contains~$\T$ and is a Levi factor of a parabolic subgroup containing~$\B$, i.e., of a \emph{standard} parabolic).  Then by \cite[2.2 Proposition]{BHGen},~$\U_{\M}=\U\cap \M$ is the unipotent radical of a Borel subgroup of~$\M$, and if~$\psi:\U\rightarrow \R^\times$ is a non-degenerate character of~$\G$ then~$\psi_{\M}=\psi_{|\U_{\M}}:\U_\M\rightarrow \R^\times$ is a non-degenerate character of~$\M$.

\subsection{Rings of definition}\label{sec:rings-definition}

In this subsection we study various rings of definition for the
representation $\ind_{\U}^{\G} (\psi)$. In particular, our objective is to prove:   

\begin{proposition}\label{ringofdefprop} There exists a finite Galois extension
  $\K/{\mathbb Q}$ contained in~$\mathbb{Q}(\mu_{p^3})$, and a  $\cO_\K[1/p][\G]$-submodule
  $W_{\U,\psi}\subset \ind_{\U}^{\G} (\psi)$, such that the natural map $W_{\U,\psi} \otimes_{\cO_{\K}[1/p]} {\mathbb
    Z}[1/p,\mu_{p^{\infty}}]\rightarrow \ind_{\U}^{\G} (\psi)$ is an isomorphism.  Moreover, the field $\K$ may be taken as follows:
\begin{enumerate}
\item If one half the sum of the positive coroots of ${\mathbf G}$ (considered as a cocharacter of ${\mathbf T}/{\mathbf Z}$) lifts to an integral cocharacter of ${\mathbf T}$, then one can take $\K = {\mathbb Q}$.
\item If (1) does not hold, and either $p$ is odd or $\F$ has characteristic $p$, then one can take $\K = {\mathbb Q}(\mu_p)$.
\item If (1) does not hold, $p= 2$, and $\F$ has characteristic zero, then one can take $\K = {\mathbb Q}(\mu_8)$.
\end{enumerate}
\end{proposition}

\begin{remark}
  The models $W_{\U,\psi}$ in the proposition may not be unique, in
  particular in case (1). However, in cases (2) and (3) our proof
  will arguably provide  ``natural'' models.
\end{remark}

Let $\K_0=\mathbb{Q}(\mu_{p^{\infty}})$ denote the field of fractions of $\R_0$; we begin
by studying the action of $\Gal(\K_0/{\mathbb Q})$ on the set of non-degenerate characters
of $\U$. From the last section, we already know that for each $\sigma \in
\Gal(\K_0/{\mathbb Q})$, there is a unique element $t_{\sigma}\in
(\mathbf{T}/\mathbf{Z})(\F)$ such that $\sigma(\psi)^{t_{\sigma}} = \psi$.   Here
${\sigma}(\psi)$ denotes the image of $\psi$ under the Galois action of $\sigma$ while
$\psi^{t_{\sigma}}$ denotes its image under the adjoint action  of $t_{\sigma}$, defined
by $\psi^{t_{\sigma}}(u)=\psi({\rm Ad}_{t_{\sigma}}(u))$. 

\begin{lemma} \label{lemmadescent1}
The map $\sigma \mapsto t_{\sigma}$ is a continuous group homomorphism
from $\Gal(\K_0/{\mathbb Q})$ to $({\mathbf T}/{\mathbf Z})(\F)$.
\end{lemma}

\begin{proof} It is clearly a group homomorphism by uniqueness of $t_{\sigma}$. To prove
  continuity, we will compute $t_{\sigma}$ explicitly.
  We denote by $\sigma\mapsto a_{\sigma}$ the composition of
  the cyclotomic character
  $\Gal(\K_{0}/\mathbb{Q})\xrightarrow{\sim}\mathbb{Z}_{p}^{\times}$ and the natural map
  ${\mathbb Z}_p^{\times} \rightarrow \F^{\times}$.  We then have
  $\psi_\F(a_{\sigma} x) = \sigma(\psi_\F(x))$ for all $x$.  Note that the map
  $\sigma \mapsto a_{\sigma}$ is continuous and valued in $\cO_F^{\times}.$

  Let $\beta$ be half the sum of the positive absolute coroots in
  $X_{*}(\mathbf{T})_{\mathbb{Q}}$.  Note that $\beta$ is fixed under
  $\Gal(\overline{\F}/\F)$, so we can regard $\beta$ as a cocharacter of
  ${\mathbf T}/{\mathbf Z}$,
  and the map $\sigma \mapsto \beta(a_{\sigma})$ is certainly a
  continuous group morphism from $\Gal(\K_0/{\mathbb Q})$ to $({\mathbf T}/{\mathbf Z})(\F)$.
  But since $\langle\beta,\alpha\rangle = 1$ for all
  $\alpha \in \Delta_{\overline{\F}}.$  we have $\sigma(\psi)^{\beta(a_{\sigma})} = \psi$, hence $t_{\sigma} = \beta(a_{\sigma}).$
\end{proof}
The next lemma studies when this morphism can be lifted to a morphism valued in $\mathbf{T}(F)$.

\begin{lemma} \label{lemmadescent2}
  Under hypothesis (1), (2) or (3) of Proposition \ref{ringofdefprop}, and with the
  notation $\K$ therein, the map  $\sigma
  \mapsto t_{\sigma}$ can be lifted to a continuous morphism $\sigma\mapsto \tilde
  t_{\sigma}$, $\Gal(\K_0/\K) \rightarrow \mathbf T(\F)$.
\end{lemma}

\begin{proof}
  (1)  If $\beta$ lifts to a cocharacter $\tilde \beta$ of ${\mathbf T}$, then we can put
  ${\tilde t}_{\sigma}:=\tilde\beta(a_{\sigma})$ and this defines a continuous morphism
  $\Gal(\K_0/{\mathbb Q})$ to ${\mathbf T}(\F)$ as desired.
  However, this will not always be the
  case; for example when ${\mathbf G} = \SL_2$. In general, there is a natural obstruction to the
  existence of a lift $\tilde\beta$ in the group
  ${\rm Ext}^{1}_{\mathbb{Z}\Gal(\overline{\F}/\F)}(X^{*}(\mathbb{Z}),\mathbb{Z})$ 
  and it  is a $2$-torsion element, since $2\beta\in X_{*}(\mathbf{T})$.

  (2) If $p$ is odd and $\K=\mathbb{Q}(\mu_{p})$, then $\Gal(\K_{0}/\K)$ is isomorphic
  to $\mathbb{Z}_{p}$. Since $2$ is invertible in $\mathbb{Z}_{p}$, this means that the
  map $\sigma\mapsto \sigma^{2}$ is a continuous automorphism of
  $\Gal(\K_{0}/\K)$. We can then put 
  $\tilde t_{\sigma}:=2\beta(a_{\sqrt\sigma})$ and get the desired morphism. On the
  other hand, if  $\F$ has characteristic $p$, then $a_{\sigma}=1$, hence also
  $t_{\sigma}=1$ for all $\sigma\in\Gal(\K_{0}/\K)$, so we may just put $\tilde
  t_{\sigma}:=1$ in this case.

  (3) Suppose $p=2$ and 
  set $\K=\mathbb{Q}(\mu_{8})$. Then
  $\Gal(\K_{0}/\mathbb{Q}(\mu_{4}))\simeq \mathbb{Z}_{2}$ and
  $\Gal(\K_{0}/\K)$ is $2\mathbb{Z}_{2}$ therein. So any
  $\sigma$ in $\Gal(\K_{0}/\K)$ has a unique ``square root''
  $\sqrt\sigma$ in $\Gal(\K_{0}/\mathbb{Q}(\mu_{4}))$ and, as above, we can take
  $\tilde t_{\sigma}:=2\beta(a_{\sqrt\sigma})$.
  \end{proof}

\begin{proof}[Proof of Proposition \ref{ringofdefprop}]
    Let $\K$ be as in Proposition \ref{ringofdefprop} and let us choose a lift
    $\sigma\mapsto \tilde t_{\sigma}$ as in the last lemma. We then get a
    semi-linear action of $\Gal(\K_0/\K)$ on $\ind_{\U}^{\G} (\psi)$, defined by :
$$(\sigma, f) \mapsto  \left(\tilde T_{\sigma} f : g \mapsto \sigma(f({\tilde t}_{\sigma} g))\right)$$
for $\sigma$ in $\Gal(\K_0/\K)$ and $f$ a left $(\U,\psi)$-equivariant function
$\G \rightarrow \R_{0}=\cO_{\K_{0}}[1/p]$.  This action is continuous for the discrete topology on
$\ind_{\U}^{\G} (\psi)$, hence it defines an effective descent datum for the pro\'etale
Galois cover $\Spec(\cO_{\K_{0}}[1/p])\rightarrow\Spec(\cO_{\K}[1/p])$, and the fixed
points of this action are an $\cO_{\K}[1/p][\G]$-submodule $W_{\U,\psi}$ of
$\ind_{\U}^{\G} (\psi)$ satisfying the requirements of the proposition.

\end{proof}

\subsection{Basic properties of GGRs}

Let~$\R$ be a~$\mathbb{Z}[1/p,\mu_{p^\infty}]$-algebra.  Let~$(\U,\psi)$ be a Whittaker datum over~$\R$. 

For any~$\R$-algebra~$\R'$, we write~$\psi_{\R'}$ for the character~$\psi\otimes\R'$, and have~$\ind_{\U}^{\G}(\psi_{\R'})\simeq \ind_{\U}^{\G}(\psi)\otimes\R'$.  
In the special case~$\R=\mathbb{C}$, Chan and Savin \cite{ChanSavin} show that~$\ind_{\U}^{\G}(\psi_{\mathbb{C}})$ is flat, we now consider the general case for the module~$W_{\U,\psi}$.

\begin{proposition} Let $\V$ be a right $\R[\G]$-module.  Then we have a natural isomorphism of $\R$-modules:
$$\V \otimes_{\R[\G]} \ind_{\U}^{\G} (\psi )\simeq \V_{\psi},$$
where $\V_{\psi}$ denotes the $(\U,\psi)$-coinvariants of $\V$.
\end{proposition}
\begin{proof}
Let $N$ be an arbitrary $\R$-module.  We then have an isomorphism:
$$\Hom_\R(\V \otimes_{\R[\G]} \ind_{\U}^{\G} (\psi), N) \simeq
\Hom_{\R[\G]}(\V, \Hom_{\R}(\ind_{\U}^{\G} (\psi), N)')$$
by Hom-tensor adjunction, where $\Hom_\R(\ind_{\U}^{\G} (\psi), N)'$ consists of smooth $\R$-linear maps $\phi$ from $\ind_{\U}^{\G} (\psi)$ to $N$, which has a right $\G$-action given by $(\phi g)(f) = \phi(gf)$.  This space of homomorphisms of right $\G$-modules is isomorphic to the space $\Hom_{\R[\G]}(\V', \Hom_{\R}(\ind_{\U}^{\G} (\psi), N))$, where $\V'$ is the module $\V$ made into a left $\G$-module via the map $g \mapsto g^{-1}$, and similarly $\Hom_{R}(\ind_{\U}^{\G} (\psi), N)$ is considered as a left $\G$-module in the usual way.

Integration over $\U \backslash \G$ defines a perfect pairing:
$$\ind_{\U}^{\G} (\psi) \times \Ind_{\U}^{\G} (\psi^{-1}_N) \rightarrow N,$$
where $\psi^{-1}_N$ is the $\R[\U]$-module whose underling $\R$-module is $N$, on which $\U$ acts via $\psi^{-1}$, and therefore identifies $\Hom_{\R}(\ind_{\U}^{\G} (\psi), N)$ with $\Ind_{\U}^{\G} (\psi^{-1}_N)$.  But we then have an isomorphism:
$$\Hom_{\R[\G]}(\V', \Ind_{\U}^{\G} (\psi^{-1}_N)) \simeq \Hom_{\R}(\V'_{\psi^{-1}}, N).$$
Since this isomorphism exists for all $N$ and is functorial in $N$, by Yoneda's lemma we have an isomorphism:
$$\V \otimes_{\R[\G]} \ind_{\U}^{\G} (\psi) \simeq \V'_{\psi^{-1}},$$
and the claim follows by observing that $\V'_{\psi^{-1}}$ is naturally isomorphic to $\V_{\psi}.$
\end{proof}

\begin{remark}
  One can define explicitly the morphism $\V_{\psi} \rightarrow \V \otimes_{\R[\G]}
  \ind_{\U}^{\G} (\psi )$ as follows. Denote by $C^{\infty}_{c}(\G,\R)$ the $\R$-module of
  smooth compactly supported $\R$-valued functions on $\G$, and by $\int_{\psi}$ the averaging map
  $C^{\infty}_{c}(\G,\R)  \longrightarrow   \ind_{\U}^{\G} (\psi )$ defined by $f\mapsto
  (g\mapsto \int_{\U}f(ug)\psi(u)^{-1}du)$ after fixing a Haar measure on $\U$. If we
  identify $C^{\infty}_{c}(\G,\R)$ with the Hecke algebra (via some choice of Haar measure
  on $\G$), then the action map induces an isomorphism
  ${a : } \V\otimes_{\R[\G]} C^{\infty}_{c}(\G,\R) \xrightarrow{\sim} \V$ and the 
   following composition 
  $$ \V  {\buildrel{a^{-1}}\over\longrightarrow} \V\otimes_{\R[\G]} C^{\infty}_{c}(\G,\R)
  \buildrel{{\rm id}\otimes\int_{\psi}}\over\longrightarrow \V \otimes_{\R[\G]}  \ind_{\U}^{\G} (\psi )$$
factors over $\V_{\psi}$, providing the desired isomorphism.
\end{remark}

Since $\U$ is a colimit of pro-$p$ groups, taking $(\U,\psi)$-coinvariants is exact, so we
deduce that $\ind_{\U}^{\G} (\psi)$ is flat as a ${\mathbb
  Z}[1/p,\mu_{p^{\infty}}][\G]$-module.  Since one can check flatness after a faithfully
flat base change, we immediately deduce (recall the notation  $\K$ from Proposition
\ref{ringofdefprop}) :

\begin{corollary}\label{WhittakerFlat}
Any model~$W_{\U,\psi}$ as in Proposition \ref{ringofdefprop} is flat as an $\cO_{\K}[1/p][\G]$-module.
\end{corollary}

Gelfand--Graev representations behave well with respect to parabolic restriction, with the proof of Bushnell--Henniart carrying over without change to coefficients in~$\R$:

\begin{proposition}[{\cite[2.2 Theorem]{BHGen}}]\label{pararestGGR}
Let~$\P$ be a standard parabolic subgroup with Levi decomposition~$\P=\M\N$ where~$\M$ is a standard Levi subgroup of~$\P$.  Let~$\P^{\circ}=\M\N^{\circ}$ denote the opposite parabolic to~$\P$.  There is a unique isomorphism
\[r:r^{\G}_{\M,\P^{\circ}}(\ind_{\U}^{\G}(\psi))\xrightarrow{\simeq} \ind_{\U_\M}^{\M}(\psi_\M)\]
characterized by the following property: for any compact open subgroup~$\A$ of~$\N^{\circ}$, and any (right)~$\A$-invariant element $f$ of~$\ind_{\U}^{\G}(\psi)$ supported on~$\U\M\A$, the element~$r(f)$ of~$\ind_{\U_\M}^{\M}(\psi_\M)$ is the function on~$\M$ given by
\[r(f)(m)=\mu(\A)\delta_{\P^{\circ}}^{1/2}(m)f(m),\]
for all~$m\in\M$, where~$\mu$ is a Haar measure on~$\N^{\circ}$ and~$\delta_{\P^{\circ}}$ is the modulus character of~$\P^{\circ}$.
\end{proposition}

The level decomposition gives us the canonical direct sum decomposition~$\ind_{\U}^\G(\psi)=\bigoplus \ind_{\U}^{\G}(\psi)_n$.  In the special case~$\R=\mathbb{C}$, Bushnell--Henniart prove in \cite[Section 7]{BHGen} that~$\ind_{\U}^{\G}(\psi_{\mathbb{C}})_n$ is a finitely generated~$\mathbb{C}[\G]$-module.  Their argument carries over, nearly without modification, in the context of~$\R[\G]$-modules (even for $\R$ non-Noetherian!):

\begin{proposition}\label{Whittakerfg}
The module~$\ind_{\U}^{\G}(\psi)_n$ is a finitely generated~$\R[\G]$-module.
\end{proposition}

In particular this holds for $\R = \R_0$.  We can thus descend our result to the depth $n$ summand
$\W_{\U,\psi,n}$ of $\W_{\U,\psi}:$
\begin{corollary}\label{Whittakerfg2}
Any model $\W_{\U,\psi,n}$ as in Proposition \ref{ringofdefprop} is finitely generated as an $\cO_\K[1/p][\G]$-module.
\end{corollary}
\begin{proof}
We have an isomorphism
$\W_{\U,\psi,n} \otimes_{\cO_{\K}[1/p]} \R_0 \simeq \ind_{\U}^{\G}(\psi)_n$, so we can
find a finite subset of $\W_{\U,\psi,n}$ that generates $\ind_{\U}^{\G}(\psi)_n$ as a
$\R_{0}[\G]$-module. Let $\W'\subset \W_{\U,\psi,n}$ be the $\R[\G]$-submodule generated
by this set. Then we have $(\W_{\U,\psi,n}/\W')\otimes_{\cO_{\K}[1/p]} \R_0 =0$, hence
also $\W_{\U,\psi,n}/\W'=0$ by faithful flatness.
\end{proof}

\begin{corollary}\label{WhittakerfgProj}
For any $\cO_\K[1/p]$-algebra $\R$, the module~$\W_{\U,\psi,n} \otimes_{\cO_\K[1/p]} \R$ is a finitely generated projective~$\R[\G]$-module.  In particular,~$\ind_{\U}^{\G}(\psi)$ is a projective~$\R_0[\G]$-module.
\end{corollary}
\begin{proof}
By \cite[Corollary 1.4]{DHKMfiniteness}, Hecke algebras over the Noetherian ring $\cO_\K[1/p]$ are Noetherian.  It follows that any finitely generated flat $\cO_{\K}[1/p][\G]$-module is projective; in particular this applies to $\W_{\U,\psi,n}$.  The case of arbitrary $\R$ then follows by base change.
\end{proof}

\begin{remark}
\begin{enumerate}
\item As~\cite{DHKMfiniteness} uses the main result of~\cite{FarguesScholze}, it might appear on first look that our proof depends on their high-tech machinery.  However, for the Corollary we do not actually need to know Noetherianity of Hecke algebras. A ring~$\mathcal{H}$ with the property that every finitely generated flat right $\mathcal{H}$-module is projective is called a \emph{right $\S$-ring}, cf.~\cite{MR2067618} for details on the theory of~$\S$-rings.  In particular it is proven there that an arbitrary subring of a right $\S$-ring is a right $\S$-ring.  Since Hecke algebras over $\mathbb{C}$ are Noetherian (and therefore right $\S$-rings), it follows from this that Hecke algebras over $\cO_{\K}[1/p]$ are also right $\S$-rings, and this is all that is necessary for the proof of the Corollary.
\item Hansen has recently provided another very nice proof that~$\ind_{\U}^{\G}(\psi)$ is a projective~$\mathbb{Z}[1/p,\mu_{p^{\infty}}][\G]$-module in \cite{Hansen}, using ``Rodier approximation''.
\end{enumerate}
\end{remark}

\subsection{Endomorphism Rings}
We now turn to the question of descending the endomorphism ring of
$\ind_{\U}^{\G}(\psi)$. For a $\R_{0}$-valued function $f$ on $\G$ and
$\sigma\in\Gal(\K_{0}/\mathbb{Q})$, define $\sigma(f)$ by
$\sigma(f)(g)=\sigma(f(g))$, and for $\gamma\in {\rm Aut}(\G)$, define
$f^{\gamma}(g)$ by $f(\gamma(g))$. Then, with the notation of  Lemma \ref{lemmadescent1}, the map
 $T_{\sigma}:\,f\mapsto \sigma(f)^{t_{\sigma}}$ takes $\ind_{\U}^{\G} (\psi)$ into
 itself. 

 \begin{lemma}
The map $(\sigma,\varphi)\mapsto T_{\sigma}\circ \varphi\circ T_{\sigma}^{-1}$   defines a  semi-linear
action of $\Gal(\K_{0}/\mathbb{Q})$ on  the $\R_{0}$-algebra $\End_{\R_0[\G]}(\ind_{\U}^{\G} (\psi))$.
Moreover, its restriction to $\Gal(\K_{0}/\K)$ coincides with the action coming from the
base change isomorphism
$W_{\U,\psi}\otimes_{\cO_{\K}[1/p]}\R_{0}\xrightarrow\sim \ind_{\U}^{\G} (\psi)$
of Proposition \ref{ringofdefprop}.
 \end{lemma}
 \begin{proof}
   The first assertion is a straightforward computation. For the second one, with the
   notation of the last paragraph of  Subsection \ref{sec:rings-definition}, we need to show that
   $T_{\sigma}\circ \varphi\circ T_{\sigma}^{-1} = \tilde T_{\sigma}\circ \varphi\circ
   \tilde T_{\sigma}^{-1}$ for all $\sigma\in\Gal(\K_{0}/\K)$ and
   $\varphi\in \End_{\R_0[\G]}(\ind_{\U}^{\G} (\psi))$.
   By construction,
   we have $\tilde T_{\sigma}= \tilde t_{\sigma}\circ T_{\sigma}$ as
   endomorphisms of the $\mathbb{Z}[1/p]$-module $\ind_{\U}^{\G} (\psi)$, and where
   $\tilde t_{\sigma}$ denotes the action of $\tilde t_{\sigma}$ on that module. So the desired equality
   boils down to the fact that $T_{\sigma}\circ \varphi\circ T_{\sigma}^{-1}$ commutes to
   the action of $\G$.
    \end{proof}

Beware that this semi-linear action is not continuous for the discrete topology on the
$\R_{0}$-algebra $\End_{\R_0[\G]}(\ind_{\U}^{\G} (\psi))$.
This is because $W_{\U,\psi}$ is not finitely generated, so the natural map:
\[\End_{\cO_{\K}[1/p][\G]}(W_{\U,\psi}) \otimes_{\cO_{\K}[1/p]} \R_0 \rightarrow \End_{\R_0[\G]}(\ind_{\U}^{\G} (\psi))\]
fails to be an isomorphism.

\begin{lemma}
  For any depth $n$ and any element $\sigma\in\Gal(\K_{0}/\mathbb{Q})$, the endomorphism
  $T_{\sigma}$ takes the summand $\ind_{\U}^{\G} (\psi)_{n}$ into itself.  Moreover, the map
  $(\sigma,\varphi)\mapsto T_{\sigma}\circ 
  \varphi\circ T_{\sigma}^{-1}$   defines a  \emph{continuous} semi-linear 
action of $\Gal(\K_{0}/\mathbb{Q})$ on  the $\R_{0}$-algebra $\End_{\R_0[\G]}(\ind_{\U}^{\G} (\psi)_{n})$.
\end{lemma}
\begin{proof}
  The map $f\mapsto \sigma(f)$ is a $\mathbb{Z}[1/p][\G]$ endomorphism of
  $\ind_{\U}^{\G} (\psi)$, hence it commutes with the action of the centre $\mathfrak
  Z_{\G,\mathbb{Z}[1/p]}$ and in particular with the idempotents that give the depth decomposition.
  On the other hand, the map $f\mapsto f^{t_{\sigma}}$ is $\G$-equivariant if one twists
  the action of $\G$ by the  automorphism $t_{\sigma}$. This automorphism also acts on
  $\mathfrak{Z}_{\G,\mathbb{Z}[1/p]}$ and we have
  $(zf)^{t_{\sigma}}=z^{t_{\sigma}}f^{t_{\sigma}}$ for all $z\in \mathfrak
  Z_{\G,\mathbb{Z}[1/p]}$ and $f\in \ind_{\U}^{\G} (\psi)$. But by construction, the
  depth decomposition is invariant under ${\rm Aut}(\mathbf{G})(\F)$, hence in particular
  the map $f\mapsto f^{t_{\sigma}}$ also commutes with the idempotents that give the depth
  decomposition.

  Hence the map $(\sigma,\varphi)\mapsto T_{\sigma}\circ\varphi\circ T_{\sigma}^{-1}$
  defines a semi-linar action of $\Gal(\K_{0}/\mathbb{Q})$ on  the $\R_{0}$-algebra
  $\End_{\R_0[\G]}(\ind_{\U}^{\G} (\psi)_{n})$ and, as in the last lemma, the restriction
  of this action to $\Gal(\K_{0}/\K)$ coincides with the action coming from the base
  change isomorphism $W_{\U,\psi,n}\otimes_{\cO_{K}[1/p]}\R_{0} \xrightarrow\sim
  \ind_{\U}^{\G} (\psi)_{n}$. Therefore, the natural map:
    \[\End_{\cO_{\K}[1/p][\G]}(W_{\U,\psi,n}) \otimes_{\cO_{\K}[1/p]} \R_0
      \rightarrow \End_{\R_0[\G]}(\ind_{\U}^{\G} (\psi)_{n})\]
    is $\Gal(\K_{0}/\K)$-equivariant for the natural action on the LHS. 
  Since $W_{\U,\psi,n}$ is  finitely presented, this map is an isomorphism, and it follows
  that the action we have defined on $\End_{\R_0[\G]}(\ind_{\U}^{\G} (\psi)_{n})$ is continuous.
\end{proof}

Let us now take Galois invariants and put
$$\mathfrak{E}_{\G,n}:= \End_{\R_0[\G]}(\ind_{\U}^{\G} (\psi)_n)^{\Gal(\K_{0}/\QQ)}.$$
By descent along the pro\'etale cover
$\Spec(\R_{0})\rightarrow \Spec(\mathbb{Z}[1/p])$,
the ${\mathbb Z}[1/p]$-algebra $\mathfrak{E}_{\G,n}$ is a model for $\End_{\R_0[\G]}(\ind_{\U}^{\G} (\psi)_n)$.
Moreover, for any $g \in
(\mathbf{G}/\mathbf{Z})(\F)$, the isomorphism of $\End_{\R_0[\G]}(\ind_{\U}^{\G}
(\psi)_n)$ with $\End_{\R_0[\G]}(\ind_{\U^g}^{\G} (\psi^g)_n)$ induced by the endofunctor $\pi \mapsto
\pi^g$ is compatible with the actions of $\Gal(\K_{0}/\QQ)$ on source and target.  Thus the
ring $\mathfrak{E}_{\G,n}$ is independent of the choice of Whittaker datum. 

\begin{remark}
The canonical map $\mathfrak{Z}_{\G,\R_0} \rightarrow \End_{\R_0[\G]}(\ind_{\U}^{\G} (\psi))$
is certainly $\Gal(\K_{0}/\K)$-equivariant with respect to the Galois actions we have defined on the two
rings, but \emph{may not be} $\Gal(\K_{0}/\mathbb{Q})$-equivariant in general.
Indeed, the action on $\mathfrak{Z}_{\G,\R_0}$ arises from the endofunctors $\pi
\mapsto \pi^{\sigma}$, whereas the action on $\End_{\R_0[\G]}(\ind_{\U}^{\G} (\psi))$
arises from the endofunctors $\pi \mapsto (\pi^{\sigma})^{t_{\sigma}}.$  The difference
between the two actions of $\sigma$ is thus given by the automorphism of
$\mathfrak{Z}_{\G,\R_0}$ induced by the endofunctor $\pi \mapsto \pi^{t_{\sigma}}$.  In
particular, if we let $\mathfrak{Z}_{\G,\R_0}^{\ad}$ denote the subring of
$\mathfrak{Z}_{\G,\R_0}$ stable under the automorphisms induced by the endofunctors $\pi
\mapsto \pi^g$ for $g \in ({\mathbf G}/{\mathbf Z})(\F)$, then the action of
$\mathfrak{Z}_{\G,\R_0}$ restricts to a $\Gal(\K_{0}/\QQ)$-equivariant map: 
$$\mathfrak{Z}_{\G,\R_0}^{\ad} \rightarrow \End_{\R_0[\G]}(\ind_{\U}^{\G} (\psi)_n).$$  

\end{remark}

We summarize the above discussion as follows:

\begin{theorem}\label{theoremmathfrakE}
  For each $n$, the ${\mathbb Z}[1/p]$-algebra $\mathfrak{E}_{\G,n}$ is commutative, flat
  and finitely generated and, there is a canonical isomorphism
\[\mathfrak{E}_{\G,n} \otimes_{{\mathbb Z}[1/p]} \R \simeq \End_{\R[\G]}(W_{\U',\psi',n} \otimes_{\cO_{\K}[1/p]} \R)\]
for each $\cO_K[1/p]$-algebra $\R$ and each Whittaker datum $(\U',\psi')$ of $\G$.  Moreover there is a natural map 
\[\mathfrak{Z}_{\G,\mathbb{Z}[1/p]}^{\ad} \rightarrow \mathfrak{E}_{\G,n}\]
that, after base change to $\cO_{\K}[1/p]$ and the identifications above, coincides with the map 
\[\mathfrak{Z}_{\G,\cO_{\K}[1/p]}^{\ad} \rightarrow \End_{\cO_{\K}[1/p][\G]}(W_{\U',\psi',n})\]
arising from the action of the Bernstein centre.
\end{theorem}

\subsection{An isomorphism of Bushnell and Henniart in the banal setting}
The aim of this section is to point out that in the banal setting, we have an analogue of a theorem of Bushnell--Henniart relating the ``$\psi$-generic blocks'' of the Bernstein centre with the endomorphisms of the Gelfand--Graev representation defined by~$\psi$.

\begin{definition}
In the case~$\R$ is an algebraically closed field, we say that an inertial class~$[\M,\pi]_\G\in\mathfrak{B}_{\R}(\G)$ is~\emph{$\psi$-generic} if
\[\dim_{\R}(\Hom_{\R[\M]}(\ind_{\U_\M}^\M(\psi_{\M,\R}),\pi'))=1\]
for all~$\pi'\in[\M,\pi]$. 
\end{definition}
If~$\Rep_{\mathbb{C}}(\G)_{[\M,\pi]_\G}$ contains a~$\psi$-generic representation then~$[\M,\pi]_{\G}$ is \emph{$\psi$-generic} (cf.~\cite[4.5 (1)]{BHGen}), and it follows that this also holds over~$\Qbar$.

\begin{lemma}
  Let~$\K/\mathbb{Q}(\mu_{p^{\infty}})$ be a finite extension, and $\pi$ be an absolutely irreducible $\psi$-generic
  cuspidal~$\K[\G]$-module with finite order central character.
  The $\cO_{\K}[1/\N_{\G}][\G]$-module
$\ind_{\U}^{\G}(\psi_{\cO_{\K}[1/\N_{\G}]})_{[\G,\pi]}$ is
  a finitely generated projective generator of~$\Rep_{\cO_{\K}[1/\N_{\G}]}(\G)_{[\G,\pi]}$. 
\end{lemma}
\begin{proof}
  Set $Q:=\ind_{\U}^{\G}(\psi_{\cO_{\K}[1/\N_{\G}]})_{[\G,\pi]}$. By
  Corollary \ref{WhittakerfgProj},
  we already know that $Q$ is projective and finitely generated, so it
  remains to see it is a generator, i.e. that any simple object of
  $\Rep_{\cO_{\K}[1/\N_{\G}]}(\G)_{[\G,\pi]}$ is a quotient of $Q$. Note
  that the set of simple quotients of $Q$ is stable under unramified twisting, since
  $\U \subset \G^{\circ}$. On the other hand, for each algebraically closed field $L$ over
  $\cO_{\K}[1/N_{\G}]$, the set of irreducible $L[G]$-modules in
  $\Rep_{\cO_{\K}[1/\N_{\G}]}(\G)_{[\G,\pi]}$ form a single unramified
  orbit, by construction. Therefore, it is enough to show that for each such $L$, the $L[\G]$-module
  $Q\otimes L$ has a simple quotient. Since it is finitely generated, it suffices to prove
  it is non-zero. By \cite[9.2]{BHGen}, we know that $Q\otimes \mathbb{C}$ is non-zero,
  hence $Q\neq 0$. Moreover, since $Q$ is a direct factor a space of
  $\cO_{\K}[1/\N_{\G}]$-valued functions, it is certainly $\ell$-adically
  separated for each banal $\ell$, meaning that $Q\otimes L$ is non-zero for all $L$ as above. 
\end{proof}

\begin{corollary}
Let~$[\G,\pi]_{\G}\in\mathfrak{B}_{\Qbar}(\G)$ be a~$\psi$-generic inertial class.  
The natural map
\[\mathfrak{Z}_{\G,\Zbar[1/\N_{\G}],[\G,\pi]_{\G}}\rightarrow
  \End_{\Zbar[1/\N_{\G}][\G]}(\ind_{\U}^{\G}(\psi_{\Zbar[1/\N_{\G}]})_{[\G,\pi]_{\G}})\]
is an isomorphism.
\end{corollary}

Now, following Bushnell and Henniart, we can extend this to all~$\psi$-generic inertial classes:

\begin{theorem}\label{BHBanalcase}
Let~$[\M,\pi]_{\G}\in\mathfrak{B}_{\Qbar}(\G)$ be a~$\psi$-generic inertial class.  Then the canonical map
\[\mathfrak{Z}_{\G,\Zbar[1/\N_{\G}],[\M,\pi]_{\G}}\rightarrow
  \End_{\Zbar[1/\N_{\G}][\G]}(\ind_{\U}^{\G}(\psi_{\Zbar[1/\N_{\G}]})_{[\M,\pi]_{\G}})\]  
is an isomorphism.  
\end{theorem}

\begin{proof}
The map becomes an isomorphism after extending scalars to~$\Qbar$ by \cite[4.3]{BHGen}.
As both rings are~torsion free, it follows that the map is injective. Since the source
$\mathfrak{Z}_{\G,\Zbar[1/\N_{\G}],[\M,\pi]_{\G}}$ is a normal domain, it thus suffices to
prove that the map is integral.

By functoriality of parabolic restriction, we have a commutative diagram:
\[\begin{tikzcd}[row sep=4ex,column sep=2ex]
\mathfrak{Z}_{\G,\Zbar[1/\N_{\G}],[\M,\pi]_{\G}}\arrow[r] \arrow[d] & \End_{\Zbar[1/\N_{\G}][\G]}(\ind_{\U}^{\G}(\psi_{\Zbar[1/\N_{\G}]})_{[\M,\pi]_{\G}})\arrow[d] \\
\mathfrak{Z}_{\M,\Zbar[1/\N_{\G}],[\M,\pi]_{\M}}\arrow[r,"\sim"]&  \End_{\Zbar[1/\N_{\G}][\M]}(\ind_{\U}^{\M}(\psi_{\Zbar[1/\N_{\G}],|\U_{\M}})_{[\M,\pi]_{\M}})
\end{tikzcd}\]
where the right vertical map follows from Proposition \ref{pararestGGR} and projection
onto the~$[\M,\pi]_{\M}$-block, the lower horizontal map is an isomorphism by the previous
corollary, and the left vertical map is the canonical inclusion.  Moreover, the right
vertical map is injective as it is injective over~$\Qbar$ by \cite[9.3 Lemma 1]{BHGen} and
both rings are torsion free.  Hence,
as~$\overline{\mathbb{Z}}[1/\N_\G][\M/\M^\circ]^{\H_{[\M,\pi]_\M}}$ is finite over the
 ring~$\mathfrak{Z}_{\G,\Zbar[1/\N_{\G}],[\M,\pi]_{\G}}$, the map of the theorem 
 is integral.
\end{proof}

 \section{Langlands parameters and the local Langlands correspondence}\label{LLcompletesection}

\subsection{The~$L$-group}
Let~$\W_\F$ denote the \emph{Weil group of~$\F$}, the topological group obtained from the absolute Galois group of~$\F$ by discretizing its unramified quotient.  Let~$\I_\F$ and~$\P_\F$ denote the inertia and wild inertia subgroups of~$\W_\F$.

Let~$\widehat{\G}$ denote the \emph{dual group} to the underlying algebraic group of~$\G$ considered as a~$\mathbb{Z}[1/p]$-group scheme.  We fix~$\LG=\widehat{\G}\rtimes \W$ a finite form of the \emph{Langlands dual group}, considered as a (possibly non-connected)~$\mathbb{Z}[1/p]$-group scheme, where~$\W$ is a finite quotient of the Weil group~$\W_\F$ through which the action of~$\W_\F$ on~$\widehat{{\G}}$ factors. 

\subsection{Langlands parameters}

Suppose~$\Kbar$ is an algebraically closed field of characteristic zero.  A morphism~$\phi:\W_\F\rightarrow \LG(\Kbar)$ is called an \emph{$L$-homomorphism} if the composition~$\W_\F\xrightarrow{\phi}\LG(\Kbar)\rightarrow\W$ is the natural projection.

\begin{definition}
\begin{enumerate}
\item A \emph{semisimple parameter} for~$\G$ over~$\Kbar$ is an~$L$-homomorphism~$\phi:\W_\F\rightarrow \LG(\Kbar)$ with open kernel whose image consists of semisimple elements in~$\LG(\Kbar)$.  Write~$\Phi_{ss,\Kbar}(\G)$ for the set of~$\widehat{\G}(\Kbar)$-conjugacy classes of semisimple parameters for~$\G$ over~$\Kbar$.  
\item A \emph{Weil-Deligne parameter} for~$\G$ over~$\Kbar$ is a pair~$(r,N)$ where~$r:\W_\F\rightarrow \LG(\Kbar)$ is a semisimple parameter, and~$N\in \mathrm{Lie}(\widehat{\G}(\Kbar))$ is a nilpotent element satisfying~$\mathrm{Ad}(r)(w)N=|w|N$.  Write~$\Phi_{\mathrm{WD},\Kbar}(\G)$ for the set of~$\widehat{\G}(\Kbar)$-conjugacy classes of Weil-Deligne parameters over~$\Kbar$ for~$\G$.
\item An~\emph{$\SL_2$-parameter} for~$\G$ over~$\Kbar$ is a morphism~$\psi:\W_\F\times\SL_2(\Kbar)\rightarrow \LG(\Kbar)$ such that~$\psi|_{\W_\F}$ is a semisimple parameter, and~$\psi|_{\SL_2(\Kbar)}$ is an algebraic morphism.  Write~$\Phi_{\SL_2,\Kbar}(\G)$ for the set of~$\widehat{\G}$-conjugacy classes of $\SL_2$-parameters over~$\Kbar$ for~$\G$.
\item 
An~$\ell$-adic \emph{Langlands parameter} for~$\G$ (or a Langlands parameter for~$\G$ over~$\Ql$), for some~$\ell\neq p$, is an~$\ell$-adically continuous Frobenius-semisimple~$L$-homomorphism~$\phi_{\ell}:\W_\F\rightarrow \LG(\Ql)$.  Write~$\Phi_{\ell\text{-adic}}(\G)$ for the set of~$\widehat{\G}(\Ql)$-conjugacy classes of Langlands parameters over~$\Ql$ for~$\G$.
\end{enumerate}
\end{definition}

\begin{remark} \label{rk:Frob_ss}
 In (4), \emph{Frobenius-semisimple} means that there exists a lift of
 Frobenius $w$ in $\W_{\F}$ such that $\phi_{\ell}(w)$ is a
 semisimple element  in $\LG(\Ql)$. This is equivalent to asking that \emph{for any} lift
 of Frobenius $w$, and more generally for any $w\in \W_{\F}\setminus \I_{\F}$, the element
 $\phi_{\ell}(w)$ is semisimple. More generally, as in \cite[8.5]{Deligne}, if
 $\phi_{\ell}:\W_\F\rightarrow \LG(\Ql)$ is 
 any $\ell$-adically continuous $L$-homomorphism, then there are :
 \begin{itemize}
 \item a unique Frobenius-semisimple $\ell$-adically continuous  $\phi_{\ell,\Fr-\rm ss}:\W_\F\rightarrow \LG(\Ql)$,
   and 
 \item a unique unipotent element $u\in \LG(\Ql)$,
 \end{itemize}
 such that, for any $w\in \W_{\F}\setminus \I_{\F}$, the Jordan decomposition of $\phi_{\ell}(w)$ is
$\phi_{\ell,\Fr-\rm ss}(w)u^{\nu(w)}$ where~$\nu : \W_{\F}/\I_{\F}\longrightarrow \mathbb{Z}$ takes
a lift of Frobenius to $1$. In particular, $u$ commutes with $\phi_{\ell}(\W_{\F})$.
\end{remark}

\begin{remark}
  In (1), note that $\phi(\I_{\F})$ consists of finite order, hence semisimple, elements. So
  the condition that $\phi(\W_{\F})$ consists of semi-simple elements is equivalent to the
  Frobenius-semisimplicity condition of the previous remark. As explained in \cite[Remark
  6.9 ii)]{DHKM}, this is also equivalent to asking that 
  the associated $1$-cocycle~$\phi^\circ:\W_\F\rightarrow \widehat{\G}(\Kbar)$ has
Zariski-closed orbit under~$\widehat{\G}(\Kbar)$ in~$\Z^{1}(\W_{\F},\widehat \G_{\Kbar})$.
\end{remark}

\begin{remark} \label{rk:dictionnary}
(1)   After fixing  choices of a Frobenius lift and of a generator of $\I_{\F}/\P_{\F}$,
  Grothendieck's monodromy theorem gives a  bijection between the set of
  Weil-Deligne representations over~$\Ql$ and the set of Langlands parameters
  over~$\Ql$. This bijection may depend on choices but it induces a canonical
  bijection~$\Phi_{\ell\text{-adic}}(\G)\rightarrow \Phi_{\mathrm{WD},\Ql}(\G)$ on
  conjugacy classes, see \cite[Lemme 8.4.3]{Deligne}. Note that this construction works without the Frobenius-semisimple
  hypothesis and ``commutes'' with Frobenius semisimplification.
  
(2)  On the other hand, there is an obvious map that associates a Weil-Deligne parameter to an~$\SL_{2}$-parameter,
$\psi\mapsto (r,N)$ defined
  by~$r(w)=\psi\left( w,\left(\begin{smallmatrix} q
        &0\\0&q^{-1}\end{smallmatrix}\right)^{\nu(w)/2}\right)$
  and~$N=d\psi\left(\begin{smallmatrix}0&1\\0&0\end{smallmatrix}\right)$.
  Thanks to the Jacobson-Morozov theorem, this map induces a
  bijection~$\Phi_{\SL_2,\Kbar}(\G)\rightarrow \Phi_{\mathrm{WD},\Kbar}(\G)$ between conjugacy
  classes. Note that this map again makes sense if we drop the Frobenius semisimplicity
  condition on both sides, but in general it fails to be injective or surjective on conjugacy
  classes (see \cite[Example 3.5]{MIY}). 
\end{remark}
 \subsection{Moduli of Langlands parameters}
Fix an arithmetic Frobenius~$\Fr$ in~$\W_\F$, and a progenerator~$s$ of the tame inertia group~$\I_\F/\P_\F$.  Denote by~$\W_\F^0$ the topological group obtained from~$\W_\F$ by \emph{discretization} of~$\W_\F/\P_\F$ with respect to these choices:~$\W_\F^0$ is the preimage under the quotient map~$\W_\F\rightarrow\W_\F/\P_\F$ of~$\langle\Fr,s\rangle$ and is endowed with the topology that extends the natural (profinite) topology on~$\P_\F$ and induces the discrete topology on~$\langle\Fr,s\rangle$.

%
Fix an exhaustive filtration~$\P_\F^e$ of~$\P_\F$ by open normal subgroups of~$\W_\F$.  And consider the functor
\begin{align*}
\underline{\Z}^1(\W_\F^0,\widehat{\G}):\mathbb{Z}[1/p]\text{-algebras}&\rightarrow \mathcal{S}et;\\
\R&\mapsto \Z^1(\W_\F^0,\widehat{\G}(\R));
\end{align*}
where~$\Z^1(\W_\F^0,\widehat{\G}(\R))$ denotes the set of $1$-cocycles~$\W_\F^0\rightarrow
\widehat{\G}(\R)$ continuous with respect to the discrete topology on $\widehat{\G}(\R)$.  Then~$\underline{\Z}^1(\W_\F^0,\widehat{\G})=\lim_{\rightarrow}\underline{\Z}^1(\W_\F^0/\P_\F^e,\widehat{\G})$ is an ind-affine scheme (where~$\underline{\Z}^1(\W_\F^0/\P_\F^e,\widehat{\G})$ is defined analagously).  Let~$\mathfrak{R}_{\LG}^e$ denote the ring of functions of the affine group scheme~$\underline{\Z}^1(\W_\F^0/\P_\F^e,\widehat{\G})$, and~$\mathfrak{R}_{\LG}=\lim_{\leftarrow}\mathfrak{R}_{\LG}^e$.

\begin{definition}
Let~$\R$ be a~$\mathbb{Z}[1/p]$-algebra.  A morphism~$\phi:\W_\F\rightarrow \LG(\R)$ is
\emph{$\ell$-adically continuous} if there exist~$f:\R'\rightarrow \R$
with~$\R'$~$\ell$-adically separated, and~$\phi':\W_\F\rightarrow \LG(\R')$ satisfying: 
\begin{enumerate}
\item $\LG(f)\circ \phi'=\phi$;
\item $\LG(\pi_{n})\circ\phi':\W_\F\rightarrow \LG(\R'/\ell^n\R')$ is continuous for
  all~$n$, where $\pi_{n}$ denotes the projection $\R'\rightarrow \R'/\ell^{n}\R'$.
\end{enumerate}
\end{definition}

\begin{theorem}[{\cite[Theorems 4.1, 4.18, Corollary 4.2]{DHKM}}]\label{Modulitheorem}
\begin{enumerate}
\item The scheme $\underline{\Z}^1(\W_\F^0,\widehat{\G})$  is a reduced, flat, locally complete intersection  of relative dimension~$\mathrm{dim}(\widehat{\G})$ over~$\mathbb{Z}[1/p]$.
\item\label{Moduliuniversal} For all~$\ell\neq p$, $\mathfrak{R}_{\LG}^{e}$ is~$\ell$-adically separated and the universal cocycle~$\phi^e_{\univ}:\W_\F^{0}\rightarrow \widehat{\G}(\mathfrak{R}_{\LG}^{e})$
extends uniquely to an~$\ell$-adically continuous cocycle~$\phi_{\univ,\ell}^{e}:\W_\F\rightarrow \widehat{\G}(\mathfrak{R}_{\LG}^{e}\otimes \mathbb{Z}_\ell)$ which is universal for all~$\ell$-adically continuous~cocycles which are trivial on~$\P_\F^e$.
\item The GIT quotient~$\underline{\Z}^1(\W_\F^0,\widehat{\G})\sslash \widehat{\G}$ is, up to canonical isomorphism, independent of the choice of~$\Fr,s$.
\end{enumerate}
\end{theorem}

\subsection{Central characters}
We write~$\Pi_{\Kbar}(\G)$ for the set of isomorphism classes of irreducible (smooth)~$\Kbar[\G]$-modules.  Let~$\mathbf{G}_{\rad}$ denote the greatest central torus of~$\mathbf{G}$ and~$\G_{\rad}=\mathbf{G}_{\rad}(\F)$, then~$\G_{\rad}\rightarrow \G$ induces a surjective morphism\[\det:\LG(\Kbar)\rightarrow\LG_{\rad}(\Kbar),\] generalizing the determinant map when~$\mathbf{G}=\mathbf{GL}_n$, whence by composing with local Langlands for tori (cf.~\cite[9.2]{Borel}) we obtain a map~$\Phi(\G)\rightarrow\Pi(\G_\rad)$.  From this, we obtain a map~$\omega_{-}:\Phi(\G)\rightarrow\Pi(\Z_\G)$, using a construction of Langlands \cite[p17-18]{Langlands} in cases where~$\G_{\rad}\subsetneq \Z_\G$.

\subsection{The local Langlands correspondence}\label{LLcorresp}
In this setting it is expected that we have a \emph{local Langlands correspondence}.  We
include here one formulation together with some of the expected properties we will use
later (we use $\SL_2$-parameters with~$\Kbar=\mathbb{C}$  for convenience in citing the
literature in the next section).

We call an~$\SL_2$-parameter~$\phi$ for~$\G$ over~$\mathbb{C}$ \emph{tempered} if the projection of~$\phi(\W_\F)$ to~$\widehat{\G}(\mathbb{C})$ is bounded, and say that a tempered parameter~$\phi$ is \emph{discrete} if~$\C_{\widehat{\G}}(\phi)/\Z(\widehat{\G})^{\W_{\F}}$ is finite.

Let~$\M$ denote a Levi subgroup of~$\G$, then the inclusion~$\M\hookrightarrow\G$ induces a unique up to~$\widehat{\G}(\mathbb{C})$-conjugacy embedding
\[\iota_{\M,\G}:\LM(\mathbb{C})\hookrightarrow \LG(\mathbb{C}).\]

Let~$\M$ be a standard Levi subgroup of~$\G$, and~$\P$ the standard
parabolic subgroup of~$\G$ with Levi factor~$\M$.   Denote
by~$\mathfrak{A}_{\M}^*$ the dual of the real Lie
algebra~$\mathfrak{A}_\M$ of the split component of the centre of~$\M$
and by~$(\mathfrak{A}_{\M}^*)^+$ the positive Weyl chamber
of~$\mathfrak{A}_{\M}^*$ with respect to the standard parabolic~$\P$.
Define a homomorphism~$\H_{\M}:\M\rightarrow \mathfrak{A}_\M$ by
demanding that~$|\chi(m)|=q^{-\langle \chi, \H_{\M}(m)\rangle}$ for
all rational characters~$\chi$ of~$\M$ and~$m\in\M$, where
here~$\langle~,~\rangle:\mathfrak{A}_\M^*\times\mathfrak{A}_\M\rightarrow
\mathbb{R}$ denotes the natural pairing.  Given an
element~$\lambda\in\mathfrak{A}_{\M}^*$ we have an unramified
character~$\chi_{\lambda}$ of~$\M$ defined by~$\chi_{\lambda}(m)=q^{-\langle \lambda,\H_{\M}(m)\rangle}$.

\begin{conjecture}[The local Langlands correspondence]\label{LLcorresp}
There is a \emph{natural} map
\[\mathcal{LL}_\G:\Pi_{\mathbb{C}}(\G)\rightarrow \Phi_{\SL_2,\mathbb{C}}(\G),\]
satisfying a list of desiderata (cf.~\cite{Borel}), including:
\begin{enumerate}
\item (\emph{Relevance}) $\mathcal{LL}_\G$ has finite fibres called~\emph{$L$-packets}, and the image of~$\mathcal{LL}_\G$ consists precisely of the \emph{relevant} parameters (cf.~\cite[3.3, 8.2]{Borel} for the definition of \emph{relevance}, and note that if~$\G$ is~$\F$-quasi-split then all parameters are relevant).
\item\label{LLmain2} (\emph{Preservation of discrete/tempered-ness}) $\pi$ is discrete series (resp.~tempered) if and only if~$\mathcal{LL}_\G(\pi)$ is discrete (resp.~tempered).
\item\label{LLmain3} (\emph{Twisting by a character}) suppose~$\chi:\G\rightarrow\mathbb{C}^\times$ is a character of~$\G$ and~$z_{\chi}:\W_\F\rightarrow \Z_{\widehat{\G}}(\mathbb{C})$ is a choice of $1$-cocycle representing the class in~$\H^1(\W_\F,\Z_{\widehat{\G}}(\mathbb{C}))$ associated to~$\chi$ by the local Langlands correspondence for characters of~$\G$ (\cite[10.2]{Borel}), then for~$\rho\in\Pi_{\mathbb{C}}(\G)$
\[\mathcal{LL}_\G(\rho\otimes \chi)^\circ=\mathcal{LL}_\G(\rho)^{\circ}\cdot z_\chi,\quad\text{in~$\H^1(\W_\F,\widehat{\G}(\mathbb{C}))$.}\]
\item\label{LLmain4} (\emph{Central characters}) For all~$\rho\in\Pi_{\mathbb{C}}(\G)$,~$\omega_\rho=\omega_{\mathcal{LL}_\G(\rho)}$.
\item (\emph{Langlands classification}) For a conjugacy class of parameter in~$\Phi_{\SL_2,\mathbb{C}}(\G)$, there is a representative~$\phi$, a standard Levi subgroup~$\M$ of~$\G$, a tempered~$\SL_2$-parameter~$\phi_\M$ for~$\M$, and $\lambda_0\in(\mathfrak{A}_{\M}^*)^+$, 
such that~$\phi=\iota_{\M,\G}\circ(\phi_{\M})_{\lambda_0}$ where~$(\phi_{\M})_{\lambda_0}=\phi_{\M}\cdot z_{\chi_{\lambda_0}}$.  Then the~$L$-packet~$\Pi_{\phi}$ of~$\phi$ is given by
\[\Pi_{\phi}=\{\J_{\P}^{\G}(\pi_{\lambda_0}):\pi\in\Pi_{\phi_{\M}}\}\]
where~$\P$ is the standard parabolic of~$\G$ with Levi factor~$\M$,~$\pi_{\lambda_0}=\chi_{\lambda_0}\pi$,  and~$\J_{\P}^{\G}(\pi_{\lambda_0})$ is the unique quotient -- the ``Langlands quotient'' -- of~$i_{\P}^{\G}(\pi_{\lambda_0})$.
\end{enumerate} 

\end{conjecture}

\subsection{Generic representations and~$L$-packets}

Suppose~$\G$ is~$\F$-quasi-split and fix a Whittaker datum~$(\U,\psi)$ for~$\G$.  Then it is expected that:

\begin{conjecture}\label{temperedpacketconj}
In each tempered~$L$-packet there is a unique~$\psi$-generic representation.
\end{conjecture}
The existence of a~$\psi$-generic representation in a tempered~$L$-packet is known as the \emph{tempered packet conjecture} and originates in {\cite[Conjecture 9.4]{Shahidi}}. Gross--Prasad and Rallis have conjectured the following criterion for the existence of a generic representation in a general~$L$-packet:

\begin{conjecture}[Gross--Prasad, Rallis]\label{GrossPrasadRallis}
The~$L$-packet of a $\SL_{2}$-parameter~$\phi$ contains a generic representation if and only if~the adjoint $L$-factor~$L(s,\mathrm{Ad}\circ\phi)$ is holomorphic at~$s=1$.
\end{conjecture}

This conjecture has an appealing reformulation in terms of the geometry of the moduli
space of Langlands parameters near $\phi$.  Let $\phi$ be a $\SL_{2}$-parameter over
$\Kbar$, and let $(r,N)$ be its associated Weil-Deligne representation.  We can then
consider the subspace $V_r$ of $\operatorname{Lie}(\widehat{\G})$ consisting of $v \in
\operatorname{Lie}(\widehat{\G})$ such that $\mathrm{Ad}(r)(w)v = |w|v;$ this is the space
of possible monodromy operators $N'$ associated to the semisimple parameter $r$ of $\phi$.
The centralizer $\widehat{\G}_r$ of $r$ acts on $V_r$ with a unique open orbit; we will
say that $\phi$ has {\em maximal monodromy} if $N$ lies in this open orbit. On the other
hand, the pair $(r,N)$ determines a $\L$-homomorphism $\varphi : \W_{\F}^{0}\rightarrow 
\LG(\Kbar)$, i.e. a $\Kbar$-point of $\underline{\Z}^1(\W_\F^0, \widehat{\G})$.  We then have:

\begin{proposition} \label{Davidsmaxorbitprop}
The following are equivalent:
\begin{enumerate}
\item\label{maxorbitprop1} $L(s,\mathrm{Ad} \circ \phi)$ is holomorphic at~$s=1$.
\item\label{maxorbitprop2} $\phi$ has maximal monodromy.
\item\label{maxorbitprop3} The $\Kbar$-point of $\underline{\Z}^1(\W_\F^0, \widehat{\G})$ corresponding to $\phi$ is a smooth point of $\underline{\Z}^1(\W_\F^0, \widehat{\G}).$
\end{enumerate}
\end{proposition}
\begin{proof}
The order of the pole at $s=1$ of $L(s,\mathrm{Ad} \circ \phi)$ is equal to the
multiplicity of the inverse cyclotomic character in the $\W_\F$-representation
$(\mathrm{Ad} \circ r)_{|{\rm Lie}(\hat\G)^{\rm{ad}(N)=0}}$, which is also the
multiplicity of the inverse cyclotomic character in the $\W_\F^{0}$-representation
$\mathrm{Ad} \circ \varphi$, i.e. the dimension of $\H^0(\W_\F^{0}, (\mathrm{Ad} \circ
\varphi)(1)).$ Therefore, the equivalence between (\ref{maxorbitprop1}) and
(\ref{maxorbitprop3}) follows from \cite[Cor. 5.3]{DHKM}.

Let $\mathcal{C}$ denote the $\widehat\G_r$-orbit of $N$ and consider the description of the conormal bundle of $\mathcal{C}$ in $V_r$ given in \cite[Prop 6.3.1]{CFMMX}. The fibre at $N$ of the conormal bundle is the kernel of the map $X\mapsto [N,X]$ on $\H^0(\W_\F^{0}, (\mathrm{Ad} \circ
r)(1))$. This kernel is precisely $\H^0(\W_\F^{0}, (\mathrm{Ad} \circ
\phi)(1))$. It follows that $\mathcal{C}$ is open if and only if $\H^0(\W_\F^{0}, (\mathrm{Ad} \circ
\phi)(1))$ vanishes. Therefore, the equivalence between (\ref{maxorbitprop2}) and
(\ref{maxorbitprop3}) follows from \cite[Cor. 5.3]{DHKM}.
\end{proof}

Thus the Gross-Prasad-Rallis conjecture is equivalent to:

\begin{conjecture}\label{Davidsmaxorbitconj}
Let $r$ be a semisimple parameter.  Then, up to $\widehat{\G}$-conjugacy, there is a unique Langlands parameter $\phi$ whose corresponding $L$-packet contains a generic representation, namely the unique parameter with semisimple parameter $r$ and maximal monodromy.
\end{conjecture}

\subsection{The semisimple local Langlands correspondence}\label{semisimpsection}

We have a natural \emph{semisimplification} map
\begin{align*}(~)_{ss}:\Phi_{\mathrm{WD},\mathbb{C}}(\G) &\rightarrow \Phi_{ss,\mathbb{C}}(\G)\\ 
(r,N)&\mapsto r.\end{align*}
  Via the bijection~$\Phi_{\SL_2,\mathbb{C}}(\G)\rightarrow \Phi_{\mathrm{WD},\mathbb{C}}(\G)$, this defines \emph{semisimplification} map~\[(~)_{ss}:\Phi_{\SL_2,\mathbb{C}}(\G) \rightarrow \Phi_{ss,\mathbb{C}}(\G),\] given by restriction via the embedding~$\W_\F\hookrightarrow \W_\F\times\SL_2(\mathbb{C})$ defined by~$w\mapsto \left(w,\left(\begin{smallmatrix}|w|^{1/2}&0\\0&|w|^{-1/2}\end{smallmatrix}\right)\right)$.  Given~$\phi\in\Phi_{\SL_2,\mathbb{C}}(\G)$, in the literature,~$\phi_{ss}$ is often called the \emph{infinitesimal character} of~$\phi$.
  
  \begin{remark}We also have a natural \emph{semisimplification} map~$\Phi_{\ell\text{-adic}}(\G)\rightarrow \Phi_{ss,\Ql}(\G)$ defined as follows: for~$\phi_{\ell}\in\Phi_{\ell\text{-adic}}(\G)$ choose a minimal parabolic subgroup of~$\LG(\Ql)$ which~$\phi_{\ell}$ factors through then project onto a Levi factor of this parabolic and take its~$\widehat{\G}(\Ql)$-conjugacy class in~$\Phi_{ss,\Ql}(\G)$.  This agrees with the map on Weil-Deligne representations via the bijections~$\Phi_{\ell\text{-adic}}(\G)\rightarrow \Phi_{\mathrm{WD},\Ql}(\G)$, and~$\Phi_{\mathrm{WD},\Ql}(\G)\simeq \Phi_{\mathrm{WD},\mathbb{C}}(\G)$ depending on choosing an isomorphism~$\mathbb{C}\simeq \Ql$. \end{remark}

It is expected that the local Langlands correspondence for~$\G$ is compatible with parabolic induction in the following sense:
\begin{conjecture}[{\cite[Conjecture 5.22]{Haines}}]\label{ParaindLanglands}
Let~$\rho\in\Pi_{\mathbb{C}}(\M)$ and~$\pi$ be an irreducible subquotient of~$i^{\G}_{\P}(\rho)$ where~$\P$ is a parabolic subgroup of~$\G$ with a Levi decomposition~$\P=\M\ltimes \N$.  Then the semisimple parameters~$\iota_{\M,\G}\circ (\mathcal{LL}_{\M}(\rho))_{ss}$ and~$(\mathcal{LL}_{\G}(\pi))_{ss}$ are~$\widehat{\G}(\mathbb{C})$-conjugate.
\end{conjecture}

{This conjecture implies that the semisimplification of a (conjugacy class of) Langlands parameter should factor through the supercuspidal support map.}\footnote{{If  one instead considers the conjectural bijective enhanced local Langlands correspondence then in~\cite{MR3845761} an expected analogue of supercuspidal support map is given on enhanced Langlands parameters.} }   
{More precisely, let}~$\Pi_{sc,\Kbar}(\G)$ denote the set of possible supercuspidal supports of an irreducible representation of~$\G$; that is, the set of~$\G$-conjugacy classes of pairs~$(\M,\rho)$ consisting of a Levi subgroup~$\M$ of a parabolic subgroup of~$\G$ and a simple supercuspidal~$\Kbar[\M]$-module.  Write\[(~)_{sc}:  \Pi_{\Kbar}(\G)\rightarrow \Pi_{sc,\Kbar}(\G)\]  for the supercuspidal support map.  Under~$\mathcal{LL}_{\M}$ for all Levi subgroups of all parabolic subgroups of~$\G$, compatibility of the local Langlands correspondence for~$\G$ with parabolic induction implies there exists a \emph{semisimple local Langlands correspondence} $\mathrm{LL}_{\G}$ defined by the commutative diagram:
\[\begin{tikzcd}
\Pi_{\Kbar}(\G)\arrow[r, "\mathcal{LL}_{\G}"] \arrow[d, "(~)_{sc}"] &\Phi_{\SL_2,\Kbar}(\G)\arrow[d, "(~)_{ss}"] \\
\Pi_{sc,\Kbar}(\G)\arrow[r,"\mathrm{LL}_{\G}"] &\Phi_{ss,\Kbar}(\G)
\end{tikzcd}\]
For ease of notation, when dealing with elements of~$\Pi_{sc,\Kbar}(\G)$ and~$\Phi_{\Kbar}(\G)$, we denote the element (which is a conjugacy class) by any choice of representative.  

{\begin{definition}
We call the fibres of~$(~)_{ss}\circ\mathcal{LL}_\G=\LL_\G\circ (~)_{sc}$ \emph{extended~$L$-packets}.  
\end{definition}}

Conjecture \ref{Davidsmaxorbitconj} implies for quasi-split groups, in particular, that only one~$L$-packet in an extended packet can support generic representations.   We conjecture that:

\begin{conjecture}[Extended packet conjecture]\label{genericpacket}
Suppose~$\G$ is~$\F$-quasi-split, and fix a Whittaker datum~$(\U,\psi)$ for~$\G$.  Then in each extended~$L$-packet there is a unique~$\psi$-generic representation.
\end{conjecture}

\begin{proposition}\label{extendedpacketunderconjectures}
Suppose~$\G$ is~$\F$-quasi-split and Conjecture \ref{LLcorresp} holds for~$\G$.  Then Conjectures \ref{temperedpacketconj} and \ref{Davidsmaxorbitconj} together imply Conjecture \ref{genericpacket}.
\end{proposition}

\begin{proof}
Let~$\phi$ be a semisimple parameter.  By Conjecture \ref{Davidsmaxorbitconj}, an extended $L$-packet corresponding to~$\phi$ contains a generic representation and all such generic representations belong to the same $L$-packet -- namely with Weil--Deligne representation conjugate to~$(\phi,N)$, where $N$ is chosen from the maximal orbit in the space of possible monodromy operators for $\phi$.  In other words, one and only one~$L$-packet in an extended packet contains generic representations.


It remains to show that there is at most one generic representation in the~$L$-packet attached to~$(\phi,N)$.  By the Langlands classification and Conjecture~\ref{LLcorresp} (5), any~$L$-packet consists of the respective Langlands quotients associated to all members of a certain tempered~$L$-packet of some Levi subgroup and a certain ``positive'' unramified character of that Levi. By Conjecture \ref{temperedpacketconj} applied to the Levi subgroup, the tempered~$L$-packet contains a unique generic representation.  A quotient of a parabolic induction being generic implies that the inducing representation is generic too, so we see that in a general~$L$-packet there is at most one generic representation; the unique candidate for genericity being the Langlands quotient of (the twist of) the generic member of the tempered~$L$-packet\footnote{While we don't need it, it is interesting to note that by the ``standard modules conjecture'', now a theorem of Heiermann--Opdam \cite[Corollary 1.2]{HeiermannOpdam}, a Langlands quotient is generic if and only if the inducing data is generic and the Langlands quotient is the entire induced representation.}.
\end{proof}

{We will also need a semisimplified form of Haines' invariance of the local Langlands correspondence under isomorphisms conjecture:
\begin{conjecture}[{Haines, cf.~\cite[Conjecture 5.2.4]{Haines}}]
\label{Hainesisomorphisms:conjecture}
Suppose~$\alpha:\G\to \G'$ is an isomorphism of connected reductive groups,~$\pi\in\Pi_{\mathbb{C}}(\G)$, and ${^\alpha\pi}$ the irreducible representation of $\G'$ obtained by pre-composing with $\alpha^{-1}$. Then the induced isomorphism ${^L\alpha}:{^L\G'}(\mathbb{C})\to {^L\G}(\mathbb{C})$ (which is well-defined up to $\widehat{\G}(\mathbb{C})$-conjugacy) takes the $\widehat{\G'}(\mathbb{C})$ conjugacy class of $\mathcal{LL}_{\G'}({^\alpha\pi})_{ss}$ to the~$\widehat{\G}(\mathbb{C})$-conjugacy class of~$\mathcal{LL}_{\G}(\pi)_{ss}$.\end{conjecture}}

 For descending our results to the smallest possible base ring later, it is useful to make the following conjecture:
\begin{conjecture}\label{fieldautconj}
Suppose~$\LL_\G$ exists,  Let~$\sigma$ be an automorphism of~$\mathbb{C}$ fixing~$\sqrt{q}$, and~$(\M,\rho)\in\Pi_{sc,\Kbar}(\G)$.  Then~$\LL_\G(\rho^\sigma)=\LL_{\G}(\rho)^\sigma$.
\end{conjecture}  

Our formulation of local Langlands in families below uses only some of the basic properties of the expected semisimple local Langlands correspondence, so we work with the following abstract setting:

\begin{definition}\label{semisimplecorrespdef}
Let~$\Kbar$ be an algebraically closed field of charateristic zero.  A \emph{semisimple correspondence for~$\G$} over~$\Kbar$ is a family of maps ~$(\mathscr{C}_\M)$ \[\mathscr{C}_{\M}:\Pi_{sc,\Kbar}(\M)\rightarrow \Phi_{ss,\Kbar}(\M),\]
over all Levi subgroups~$\M$ of~$\G$, which satisfy:
\begin{enumerate}[$(\mathscr{C}1)$]
\item  $\mathscr{C}_\M$ has finite fibres called semisimple~$L$-packets, and is surjective if~$\M$ is~$\F$-quasi-split.\footnote{We could call a semisimple parameter \emph{relevant} if it is the semisimplification of a relevant Langlands parameters and demand in general the image is the set of relevant semisimple parameters.  However, when~$\M$ is not~$\F$-quasi-split it is not clear what this condition gives as it is given by semisimplification of Langlands parameters, (cf.~\cite[p10]{Haines}).}
\item (\emph{supercuspidals}) for all~$(\M,\rho)\in\Pi_{sc,\Kbar}(\M)$ (i.e.~all classes of supercuspidal representations of~$\M$), {any parameter~in~$\mathscr{C}_{\M}(\M,\rho)$} is the semisimplification of a discrete Langlands parameter. 
\item{ (\emph{unramified twisting}) for all~$(\M,\rho)\in\Pi_{sc,\Kbar}(\M)$} and all~$\chi:\M\rightarrow\Kbar^\times$ unramified, 
{\[\mathscr{C}_{\M}(\M,\rho\otimes \chi)^\circ=\mathscr{C}_{\M}(\M,\rho)^{\circ}\cdot z_\chi.\]}
\item {(\emph{central characters}) let~$(\M,\rho)\in\Pi_{sc,\Kbar}(\M)$}, if the centre of~$\M$ is not compact, we require\[\omega_\rho=\omega_{\mathscr{C}_{\M}{(\M,\rho)}}.\]
\item (\emph{supercuspidal supports and genericity}) In the case~$\G$ is~$\F$-quasi-split: For any Whittaker datum~$(\U,\psi)$ for~$\G$, every~semisimple~$L$-packet contains a unique~``$\psi$-generic'' supercuspidal support (Where we call a supercuspidal support~\emph{$\psi$-generic} if the support contains a representative~$(\M,\rho)$ with~$\M$ a standard Levi subgroup and~$\rho$ a~$\psi_\M$-generic representation).
\item (\emph{compatibility with parabolic induction})
For~Levi subgroups~$\M\leqslant\M'$ of~$\G$ and a supercuspidal representation~$\rho$ of~$\M$,{~$\mathscr{C}_{\M'}(\M,\rho)$ is the unique~$\widehat{\M}'(\Kbar)$-conjugacy class containing~$\iota_{\M,\M'}\circ\mathscr{C}_{\M}(\M,\rho)$.}
\item {(\emph{compatibility with field automorphisms})} For all automorphisms~$\sigma$ of~$\Kbar$ fixing~$\sqrt{q}$ and all~$(\M',\rho)\in\Pi_{sc,\Kbar}(\M)$,~$\mathscr{C}_{\M}(\M',\rho^\sigma)=\mathscr{C}_{\M}(\M',\rho)^\sigma$.
\item {(\emph{compatibility with inner automorphisms}) Let~$g \in
(\mathbf{G}/\mathbf{Z})(\F)$,~$(\M,\rho)\in\Pi_{sc,\Kbar}(\M)$, then~$\mathscr{C}_{\G}(\M^g,\rho^g)=\mathscr{C}_{\G}(\M,\rho)$.}
\end{enumerate}
\end{definition}

\begin{remark}
It is important to note, the properties~$(\mathscr{C}1-6)$ are quite mild and do not characterize a semisimple correspondence.  For example, one could interchange the semisimple parameters of two~$\psi$-generic supercuspidal supports on the same (conjugacy class of) Levi subgroup and get another semisimple correspondence.  
\end{remark}

As noted before, the conjectural semisimple local Langlands correspondence defines a \emph{natural} semisimple correspondence.
\begin{proposition}\label{semisimplecorresp}
Suppose~$\mathcal{LL}_\M$ exists (Conjecture \ref{LLcorresp}), together with Conjectures \ref{ParaindLanglands}, \ref{genericpacket}, \ref{Hainesisomorphisms:conjecture}, and \ref{fieldautconj}, for all Levi subgroups~$\M$ of~$\G$.  Then~$(\LL_{\M})=((~)_{ss}\circ \mathcal{LL}_\M\circ (~)_{sc}^{-1})$ is a semisimple correspondence for~$\G$.
\end{proposition}
\begin{proof}
The map~$\LL_{\M}$ is well defined by Conjecture \ref{ParaindLanglands}.  Properties~$(\mathscr{C}1)-(\mathscr{C}4)$ follow at once from Properties~$(1)-(4)$ of Conjecture \ref{LLcorresp}.   A quotient of a parabolic induction being~$\psi$-generic implies that the inducing representation is~$\psi_{\M}$-generic so $(\mathscr{C}5)$ follows from Conjecture \ref{genericpacket}.  Compatibility with parabolic induction~$(\mathscr{C}6)$ follows from Conjecture \ref{ParaindLanglands}, and $(\mathscr{C}7)$ is Conjecture \ref{fieldautconj}.

Finally, we explain how to deduce $(\mathscr{C}8)$ from Conjecture \ref{Hainesisomorphisms:conjecture}. Given $g$ in $(\mathbf{G}/\mathbf{Z})(\F)$ and $(\M,\rho)$ as in $(\mathscr{C}8)$, we can fix a maximal torus $\T$ in $\M$, and write $g = th$, with $h$ in $\G(\F)$ and $t$ in $(\mathbf{T}/\mathbf{Z})(\F)$.  Then we have $\mathscr{C}_\G(\M^g,\rho^g) = \mathscr{C}_\G(\M^t,\rho^t) = \mathscr{C}_\G(\M,\rho^t)$: the first equality because our semisimple correspondence $\mathscr{C}_\G$ is defined on $\G(\F)$-conjugacy classes of supercuspidal pairs, the second because $\M^t = \M$. So it suffices to show $\mathscr{C}_\G(\M,\rho^t) = \mathscr{C}_\G(\M,\rho)$.  Conjecture \ref{Hainesisomorphisms:conjecture} shows that $\mathscr{C}_\M(\M,\rho^t) = \mathscr{C}_\M(\M,\rho)$, and then we are done by compatibility with parabolic induction~$(\mathscr{C}6)$.
\end{proof}

\subsection{Integral parameters}
\begin{definition}
A continuous $L$-homomorphism $\phi_{\ell}:\,\W_{\F}\To{} \LG(\Ql)$ is called \emph{integral} if its image in~$\LG(\Ql)$ has compact closure.  
\end{definition}

A basic application of Bruhat-Tits theory (cf.~ the proof of
\cite[Proposition 2.9]{DHKM}) shows that $\phi_{\ell}$ is integral if and
only if a $\widehat{\G}(\Ql)$-conjugate of it takes values in~$\LG(\Zl)$.     

\begin{proposition}\label{propintegralparameters}
Let~$\phi_{\ell}$ be a continuous $L$-homomorphism $\W_{\F}\To{} \LG(\Ql)$.
\begin{enumerate}
\item\label{iIntegralProp} Suppose $\phi_{\ell}$ is discrete and
  Frobenius semisimple.  Then the following are equivalent: 
\begin{enumerate}
\item  \label{iIntegralProp1}$\phi_{\ell}$ is integral;
\item \label{iIntegralProp2}$\det(\phi_{\ell})$ is integral;
\item \label{iIntegralProp3}$\omega_{\phi_{\ell}}$ is integral.
\end{enumerate}
\item\label{iiIntegralProp}  In general, the following are equivalent:
\begin{enumerate}
\item\label{iiIntegralProp1} $\phi_{\ell}$ is integral;
  \item\label{iiIntegralProp2}  the
 Frobenius-semisimplification~${\phi_{\ell}}_{\Fr-ss}$ of $\phi_{\ell}$ is
  integral.
\item\label{iiIntegralProp2}  the
  semisimplification~${\phi_{\ell}}_{ss}$ of $\phi_{\ell}$ is
  integral. 
\item \label{iiIntegralProp3} the $\Ql$-point of $\underline{\Z}^1(\W_\F, \widehat{\G})\sslash \widehat{\G}$ corresponding to $\phi_{\ell}$ factors through $\Spec (\Zl)$.
\end{enumerate}
\end{enumerate}
\end{proposition}
\begin{proof}
  The equivalence between (2a) and (2b) follows from Remark \ref{rk:Frob_ss} and the fact
  that a unipotent element of $\LG(\Ql)$ is compact (i.e. generates a subgroup that has
  compact closure).
  So we may assume from now on that $\phi_{\ell}$ is Frobenius-semisimple.
  Denote by ${\rm Sp}_{2} : \W_{\F}\longrightarrow\SL_{2}(\Ql)$ the special Langlands
  parameter for ${\rm PGL}_{2}$, which factors over $\W_{\F}/\P_{\F}$ and takes a (chosen) generator of
  $\I_{\F}/\P_{\F}$ to $\left(\begin{smallmatrix}1&1\\0&1\end{smallmatrix}\right)$ and a
  (chosen) lift of Frobenius to
  $\left(\begin{smallmatrix}\sqrt q&0\\0&\sqrt q^{-1}\end{smallmatrix}\right)$.
  By the dictionary between Langlands parameters over $\Ql$ and $\SL_{2}$ parameters
  recalled in Remark \ref{rk:dictionnary}, we can
    factorize $\phi_{\ell} = \psi \circ ({\rm id}\times{\rm Sp}_{2})  : \W_{\F}\To{}
    \W_{\F}\times\SL_{2}(\Ql)\To{\psi} \LG(\Ql)$ for some continuous $L$-homomorphism $\psi$ such
    that $\psi(\I_{\F}\times\{1\})$ is finite and $\psi(\Fr,1)$ is semisimple for any
    Frobenius element $\Fr$ in $\W_{\F}$.
    Since ${\rm Sp}_{2}(\W_{\F})$ has compact closure in $\SL_{2}(\Ql)$,  we see that
    $\overline{\phi_{\ell}(\W_{\F})}$ is compact in $\LG(\Ql)$ if and only if
    $\overline{\psi(\W_{\F}\times\{1\})}$ is compact in $\LG(\Ql)$. In other words, we have
    \begin{center}
      $\phi_{\ell}$ is integral $\Leftrightarrow$ $\psi_{|\W_{\F}\times\{1\}}$ is integral 
$\Leftrightarrow$ $\psi(\Fr,1)$ is a compact element in $\LG(\Ql)$.
\end{center}

Let us now prove (2).  Let $1_{2}:\W_{\F}\longrightarrow\SL_{2}(\Ql)$ be the  Langlands
  parameter of the trivial representation of ${\rm PGL}_{2}$, which factors over $\W_{\F}/\I_{\F}$ and takes  Frobenius to
  $\left(\begin{smallmatrix}\sqrt q&0\\0&\sqrt q^{-1}\end{smallmatrix}\right)$. Then~${\phi_{\ell}}_{\rm ss}:= \psi \circ ({\rm id}\times 1_{2})$ is (a representative of) the
  semi-simplification of ${\phi_{\ell}}$. As above, since~$1_{2}(\W_{\F})$ has compact closure
  in~$\SL_{2}(\Ql)$, we see that
     \begin{center}
      ${\phi_{\ell}}_{\rm ss}$ is integral $\Leftrightarrow$ $\psi_{|\W_{\F}\times\{1\}}$ is integral.
    \end{center} whence the equivalence
    between (2a) and (2c). The implication (2c) $\Rightarrow$ (2d) is clear, so let us
    assume that (2d) holds. Observe that, in order to prove (2c), it suffices to find a
    finite extension $\F'$ of $\F$ such that ${\phi_{\ell}}_{\rm ss}(\W_{\F'})$ has
    compact closure in $\LG(\Ql)$. Let us choose $\F'$ such that ${\phi_{\ell}}_{\rm
      ss}(\I_{\F'})=\{1\}$ and ${\phi_{\ell}}_{\rm ss}(\W_{\F'})\subset \widehat\G(\Ql)$. Denoting by
    $\Fr'$ a Frobenius element in $\W_{\F'}$, we need to show that ${\phi_{\ell}}_{\rm ss}(\Fr')$ is
    compact in $\widehat\G(\Ql)$, knowing that its image in $(\widehat{\G}\sslash\widehat\G)(\Ql)$ belongs to
    $(\widehat{\G}\sslash\widehat\G)(\Zl)$. But we can conjugate ${\phi_{\ell}}$ such that
    ${\phi_{\ell}}_{\rm ss}(\Fr') \in \widehat\T(\Ql)$ and, since the morphism
    $\widehat\T\To{}\widehat{\G}\sslash\widehat\G$ is finite, we then actually  have 
    ${\phi_{\ell}}_{\rm ss}(\Fr') \in \widehat\T(\Zl)$.

    Let us now prove (1). Here, the implication (1a)$\Rightarrow$(1b) is clear and the
    equivalence between (1b) and (1c) follows from the fact that the local Langlands
    correspondence for tori respects integrality. Using the isogeny $\G_{\rm sc}\times
    \G_{\rm rad}\To{} \G$, we see that, to prove (1b)$\Rightarrow$(1a), it suffices to do
    it when $\widehat\G$ is adjoint. That is, we need to prove that, in this case, any
    discrete parameter ${\phi_{\ell}}$ is integral. With the notation $\psi$ introduced above, all we need is to
    check that $\psi(\Fr,1)$ is a compact element. But, since $\psi(\I_{\F}\times\{1\})$ is
    finite, there is an integer $n$ such that $\psi(\Fr,1)^{n}$ belongs to
    $\widehat\G(\Ql)$ and centralizes  $\psi(\I_{\F}\times\{1\})$. It follows that
    $\psi(\Fr,1)^{n}$ belongs to the centralizer in $\widehat\G(\Ql)$ of the image
    $\psi(\W_{\F}\times\SL_{2}(\Ql))$ of $\psi$, which is trivial  since ${\phi_{\ell}}$
    is discrete and $\widehat\G$ is adjoint.  Hence we see that $\psi(\Fr,1)$ even has finite order.
\end{proof}

\begin{remark}\label{remarkpreservesintegrality}
Assume the conjectural local Langlands correspondence for~$\G$ (Conjecture \ref{LLcorresp}) and its Levi subgroups, and fix an isomorphism~$\mathbb{C}\simeq \Ql$ to write the local Langlands correspondence for~$\ell$-adic representations.
\begin{enumerate}
\item Using preservation of central characters and discrete-ness (more precisely, that the parameters of supercuspidal representations are discrete), Proposition \ref{propintegralparameters} (1) and \cite[II 4.12]{Vig96}, show that a supercuspidal representation of~$\G$ is integral if and only if its associated~$\ell$-adic Langlands parameter is integral.  
\item Building on this (under the same hypotheses for all Levi subgroups), using compatibility with parabolic induction, Proposition \ref{propintegralparameters} (2) and \cite[Corollary 1.6]{DHKMfiniteness}, show that an irreducible representation of~$\G$ is integral if and only if its associated~$\ell$-adic Langlands parameter is integral.  
\end{enumerate}
\end{remark}
\section{The semisimple local Langlands correspondence for classical groups}\label{classical}
\subsection{Classical groups}\label{classgroups}
Let~$\E/\F$ be a trivial or quadratic extension of~$p$-adic fields and~$c$ denote the generator of~$\Gal(\E/\F)$, and~$\epsilon\in\{\pm1\}$.  Let~$(\V,h)$ be an~$\epsilon$-hermitian space and
\[\U(\V,h)=\{g\in\GL_\E(\V): h(gv,gw)=h(v,w),~\text{for all~}v,w\in\V\},\]
denote the group of isometries.  
\begin{enumerate}
\item When~$\E/\F$ is quadratic then~$\U(\V,h)$ is a \emph{unitary group}.  
\item When~$\E=\F$ and~$\epsilon=-1$, then~$\U(\V,h)$ is a \emph{symplectic group}
\item When~$\E=\F$ and~$\epsilon=1$, then~$\U(\V,h)$ is an \emph{orthogonal group}.
\end{enumerate}
If~$\dim_\E(\V)=0$ then we interpret~$\G$ to be the trivial group.  Thus, the isometry group~$\U(\V,h)$ is the~$\F$-points of a possibly disconnected reductive group defined over~$\F$.  
\begin{definition}\label{defclassgroup}
By a \emph{classical group}~$\G$ we mean a unitary, symplectic or {odd special orthogonal group}.  
\end{definition}So, in particular, a ``classical group'' in this work is the points of a connected reductive group.

\begin{remark}
{\emph{We do not consider even (special) orthogonal groups 
and have not included them in our definition of a ``classical group''.}}
\end{remark}
%

\subsection{Parameters for classical groups as (conjugate) self-dual parameters}\label{Sectparametersasselfdual}

We can associate to a parameter~$\phi:\W_\F\times\SL_2(\mathbb{C})\rightarrow \LG(\mathbb{C})$ a self-dual or conjugate self-dual representation~$\mathcal{R}\circ \phi:\W_\E\times\SL_2(\mathbb{C})\rightarrow \GL_m(\mathbb{C})$, for some $m=m(\widehat{\G})$, following \cite[Section 8]{GGP12}:

Suppose~$\G$ is symplectic, then~$\LG(\mathbb{C})=\mathrm{SO}_{2n+1}(\mathbb{C})$ and we can compose with the standard representation~$\SO_{2n+1}(\mathbb{C})\rightarrow \GL_{2n+1}(\mathbb{C})$.  This allows us to consider an~$\SL_2$-parameter for~$\G$ as a self-dual parameter~$\phi:\W_\F\times\SL_2(\mathbb{C})\rightarrow \SO_{2n+1}(\mathbb{C})\hookrightarrow  \GL_{2n+1}(\mathbb{C})$ into~$\GL_{2n+1}(\mathbb{C})$.  The case of odd special orthogonal groups is similar. 

Suppose~$\G$ is unitary.  Then~$\LG(\mathbb{C})=\GL_n(\mathbb{C})\rtimes\Gal(\E/\F)$ and restricting~$\phi:\W_\F\times \SL_2(\mathbb{C})\rightarrow \LG(\mathbb{C})$ to~$\W_\E\times\SL_2(\mathbb{C})$ gives the required conjugate self-dual representation~$\mathcal{R}\circ\phi:\W_\E\times\SL_2(\mathbb{C})\rightarrow \GL_n(\mathbb{C})$.

\begin{proposition}[{\cite[Theorem 8.1]{GGP12}}]\label{GGPprop}
Let~$\phi$ and~$\phi'$ be~$\SL_2$-parameters for~$\G$.  Then~$\phi$ and~$\phi'$ are conjugate in~$\widehat{\G}(\mathbb{C})$ if and only if~$\mathcal{R}\circ\phi$ and~$\mathcal{R}\circ\phi'$ are equivalent in~$\GL_m(\mathbb{C})$.  
 \end{proposition}
  
\subsection{The local Langlands correspondence for classical groups}\label{LLforclassgroups}
\emph{The} local Langlands correspondence satisfying properties (1)-(5) of Conjecture \ref{LLcorresp} is known for general linear groups and classical groups in the following cases:
\begin{enumerate}[-]
\item for general linear groups by \cite{HarrisTaylor, Henniart, Scholze} (after Bernstein--Zelevinsky's reduction to the supercuspidal case \cite{Zelevinsky});
\item for symplectic, and odd split special orthogonal groups by \cite{Arthur};
{\item odd non-quasi-split special orthogonal groups by \cite{MR4776199};}
\item for quasi-split unitary groups by \cite{Mok} ;
\item for unitary groups by \cite[Section 1.6]{KMSW} and \cite[Theorem 2.5.1]{ChenZou2}. \end{enumerate}

\begin{remark}
{We do not include even (special) orthogonal groups, where while there is the local Langlands parameterization \cite{Arthur}, \cite[Theorem 4.4]{ChenZou1}, and \cite{MoeglinRenard} some of our framework would need extending (perhaps, mildly in this case) to disconnected reductive $\F$-groups to use it.}
\end{remark}

\subsection{The Plancherel measure}\label{subsection:plancherel_body}
Let $\G$ be a classical group either equal to $\U(\V,h)$ or in the special orthogonal case the subgroup of~$\U(\V,h)$ of isometries of determinant one, $\P$ a maximal parabolic subgroup of $\G$, and $\M$ a
Levi component of $\P$. Recall that the data of $(\M,\P)$ is equivalent to that of a
splitting of $\E$-vector spaces $\V=\W\oplus \V'\oplus \W'$, where $\W,\W'$ are totally
isotropic and $\V'$ is orthogonal to $\W\oplus \W'$, and this provides us with an
identification $\M=\GL(\W)\times\G'$ where~$\G'$ is $\U(\V',h)$ or in the orthogonal case is the subgroup of isometries of determinant one.

Fix~$\psi$ a non-trivial character of~$\F$, and let $\tau\otimes\pi$ be an irreducible
representation of $\M$.
We refer to Appendix~\ref{apdx:plancherel} for the construction of
the Plancherel measure
\[\mu_{\psi}^{\G}(\tau\otimes \pi)\in\mathbb{C}(\M/\M^{\circ})\]
in which the intertwining
operators have been normalized  with respect to a certain choice of Haar measure corresponding to~$\psi$ as in
\cite[Appendix B]{GanIchino14}.

With the aim of comparing this Plancherel measure with certain $\gamma$ factors, we will
slightly change notation as follows. Observe first that the map
$(g_{\W},g_{\V'})\mapsto \text{val}_{\E}(\det(g_\W))$ induces an isomorphism
$\M/\M^{\circ}\To\sim \mathbb{Z}$, and then an isomorphism
$\mathbb{C}[\M/\M^\circ] \simeq \mathbb{C}[(q^{-s})^{\pm 1}]$ where we see $q^{-s}$ as an
indeterminate. 
Accordingly,
we denote by $|\det|^s:\, \GL(\W)\To{} \CC[(q^{-s})^{\pm1}]$ 
the universal unramified character $g_\W\mapsto (q^{-s})^{\text{val}_{\E}(\det(g_\W))}$ of
$\GL(\W)$ and we put $\tau_{s}\otimes\pi:=  \tau.|\det|^{s} \otimes\pi$. Upon identifying
$\mathbb{C}(\M/\M^\circ)$ with $\mathbb{C}((q^{-s})^{\pm 1})$, this is the universal unramified
twist of $\tau\otimes\pi$ appearing in the definition of
$\mu_{\psi}^{\G}(\tau\otimes\pi)$. For this reason, we will denote by 
$\mu^{\G}_{\psi}(\tau_s\otimes\pi)$  the rational function in~$\mathbb{C}(q^{-s})$
corresponding to $\mu^{\G}_{\psi}(\tau\otimes \pi)\in\mathbb{C}(\M/\M^{\circ})$. We will
sometimes drop the subscript $\psi$ or the superscript $\G$ to simplify notation and
simply write~$\mu(\tau_s\otimes \pi)$.

 Similarly, for~$\tau$ and~$\tau'$ irreducible representations of~$\GL_k(\E)$
 and~$\GL_{k'}(\E)$, setting~$\tau_s=\tau\otimes|\det|^s$ and
 $\tau'_t=\tau'\otimes|\det|^t$, denote by~$$\mu(\tau_s\otimes\tau'_t) =
 \mu_{\psi_E}^{\GL_{k+k'}}(\tau\otimes \tau')\in\mathbb{C}(q^{-s},q^{-t})$$ the
 \emph{Plancherel measure} of~$\tau\otimes\tau'$ normalized with respect to our fixed
 character~$\psi_\E=\psi\circ\mathrm{Tr}_{\E/\F}$. Note that this is actually a rational
 function in $\mathbb{C} (q^{-s}q^{t})$. In particular, the hyperplanes $q^{-t}=1$ and
 $q^{-s}=q^{t}$ are not contained in its singular locus, and we simply denote by
 $$\mu(\tau_s\otimes\tau'):=\left.\mu(\tau_{s}\otimes\tau_{t})\right|_{q^{-t}=1}
 \hbox{ and }
 \mu(\tau_s\otimes\tau'_{-s}):=\left.\mu(\tau_{s}\otimes\tau'_{t})\right|_{q^{-t}=q^{s}}$$
the respective rational functions in $\mathbb{C}(q^{-s})$ obtained by restriction.

Later, we will need the following multiplicativity property of the Plancherel measure: 

\begin{proposition}\label{multofplancherel}
\begin{enumerate}
\item\label{multclassfactor} Let~$\M=\M_0\times \M_1$ be a Levi subgroup of~$\G$,
  where~$\M_0=\prod_{i=1}^r\GL_{n_i}(\E)$ is a product of general linear groups and~$\M_1$
  is a classical group.  Let~$\rho$ be an irreducible representation of~$\M$ which
  decomposes with respect to this decomposition as~$\bigotimes_{i=1}^r
  \rho_i\otimes\rho'$.  Let~$\P$ be a parabolic subgroup of~$\G$ with Levi factor~$\M$
  and~$\pi$ an irreducible subquotient of~$i_{\P}^{\G}(\rho)$.  Then, for~$\tau$ an
  irreducible representation of~$\GL_k(\E)$, we have 
\begin{equation*}\mu(\tau_s\otimes \pi) =
  \mu(\tau_s\otimes\rho')\prod_{i=1}^r\left.\mu^{\GL_{k+k_i}(\E)}(\tau_s\otimes
    \rho_i)\right.\left.\mu^{\GL_{k+k_i}(\E)}(\tau_s\otimes
    {\rho_i^c}^{\vee})\right..\end{equation*}
\item\label{multGLnfactor} 
Let $\P'$ be a parabolic of $\GL_m(\E)$ with Levi factor $\GL_{k_1}(\E)\times\cdots \times \GL_{k_l}(\E)$ and suppose $\tau$ is an irreducible subquotient of $i_{\P'}^{\GL_m(\E)}(\tau'_1\otimes \cdots \otimes \tau'_l)$, with $\tau'_i$ an irreducible representation of $\GL_{k_i}(\E)$.  Then, for~$\pi$ an irreducible representation of a classical group~$\G'$, we have
\begin{equation*}
\mu(\tau_s\otimes \pi) = \left(\prod_{1\leqslant i<j\leqslant
    l}\mu^{\GL_{k_i+k_j}(\E)}((\tau_i')_s\otimes
  {((\tau_j')^c)^{\vee}}_{-s})\right)\prod_{1\leqslant i\leqslant l}\mu((\tau_i')_s\otimes
\pi) 
\end{equation*}
\item\label{multGLnmeasure} Let $m = k_1+\cdots + k_r$, and let $n = k_1'+\cdots + k'_l$, and suppose $\tau$ (resp.~$\tau'$) is an irreducible subquotient of $i_{\P'}^{\GL_m(\E)}(\tau_1\otimes\cdots\otimes \tau_r)$ (resp.~$i_{\P''}^{\GL_n(\E)}(\tau_1'\otimes\cdots \otimes \tau_l')$), with $\tau_i$ (resp.~$\tau_i'$) an irreducible representation of $\GL_{k_i}(\E)$ (resp.~$\GL_{k_i'}(\E)$).  Then
\begin{equation*}
\mu^{\GL_{m+n}(\E)}(\tau_s\otimes (\tau')_t) = \prod_{\substack{1\leqslant i\leqslant r \\ 1\leqslant j\leqslant l\\i\leqslant j}}\mu^{\GL_{k_i+k_j'}(\E)}((\tau_i)_s\otimes (\tau_j')_t)
\end{equation*}
\end{enumerate}
\end{proposition}

Parts (\ref{multclassfactor}) and (\ref{multGLnfactor} of this proposition are proved in~\cite[B.5]{GanIchino14} in the special case where~$\pi$ (respectively, $\tau$) is an irreducible \emph{subrepresentation} of~$i_{\P}^{\G}(\rho)$ (respectively, $i_{\P'}^{\GL_m(\E)}(\tau'_1\otimes \cdots \otimes \tau'_l)$).  The general case stated here will be deduced from our results in Appendix~\ref{apdx:j_functions}, specifically Corollary~\ref{cor:multiplicativity_irreducible}; this is carried out in Subsection~\ref{apdx:plancherel}.

\subsection{Gamma factors and a semisimple converse theorem}
On the other hand, let~$(r,N)$ be a Weil-Deligne representation for~$\GL_n(\F)$.  We denote by
\[L(s,(r,N))=\det(1-q^{-s}r(\Fr)_{|\ker(N)^{\I_{\F}}})^{-1},\]
the \emph{$L$-factor} of~$(r,N)$, and by~$\epsilon(s,(r,N),\psi)\in\mathbb{C}[q^{-s}]^\times$ the \emph{epsilon factor} of \cite{Deligne}.  We define the \emph{gamma factor} of~$(r,N)$ by
\[\gamma(s,(r,N),\psi)=\epsilon(s,(r,N),\psi)\frac{L(1-s,(r,N)^*)}{L(s,(r,N))},\]
where~$(r,N)^*$ is the dual Weil-Deligne representation. As above, we view $q^{-s}$ as a formal variable and consider $\gamma(s,(r,N),\psi)$ as an element of $\CC(q^{-s})$.

For a semisimple parameter~$\phi$ for~$\G$, we define its $L$-factor, epsilon, and gamma factor, to be the factors of the Weil--Deligne representation~$(\phi,0)$.  The standard properties we use of the gamma factors are:
\begin{enumerate}
\item The~$\gamma$-factor only depends on the semisimplification of a Weil-Deligne representation: for~$\phi=r=(r,N)_{ss}$, we have~$\gamma(s,\phi,\psi)=\gamma(s,(r,N),\psi)$,;
\item The $\gamma$-factor is multiplicative: for~$\phi\simeq \phi_1\oplus\phi_2$, we have~$\gamma(s,\phi,\psi)=\gamma(s,\phi_1,\psi)\gamma(s,\phi_2,\psi)$.
\end{enumerate}
See~\cite[Corollary 4.5]{HelmMossGamma} for precise references.

Moreover, in this setting, Langlands' conjecture on the Plancherel measure \cite[Appendix II]{MR0579181} is known:
\begin{proposition}\label{Propplanchereldecomp}
 Let~$\pi$ be an irreducible representation of~$\G$,~$k$ a positive integer,~$\tau$ an irreducible representation of~$\GL_k(\E)$,~$\tau'$ an irreducible representation of~$\GL_{k'}(\E)$.     
 \begin{enumerate}
 \item\label{Propplanchereldecomp1} Writing,~$\phi_{\tau}=\mathcal{LL}_{\GL_k(\E)}(\tau)$ and~$\phi_{\tau'}=\mathcal{LL}_{\GL_{k'}(\E)}(\tau')$, we have
 \[\mu(\tau_s\otimes\tau'_r)= \gamma(s-r,\phi_{\tau}\otimes\phi_{\tau'}^*,\psi_\E)\gamma(r-s,\phi_{\tau}^*\otimes\phi_{\tau'},\overline{\psi_{\E}}).\]
 \item\label{Propplanchereldecomp2} Writing,~$\phi=\mathcal{R}\circ \mathcal{LL}_{\G}(\pi)$ and~$\phi_{\tau}=\mathcal{LL}_{\GL_k(\E)}(\tau)$, we have
\[\mu(\tau_s\otimes\pi)=\gamma(s,\phi_{\tau}\otimes\phi^*,\psi_{\E})\gamma(-s,\phi_{\tau}^*\otimes \phi,\overline{\psi_{\E}})\gamma(2s,R\circ \phi_\tau,\psi)\gamma(-2s,R\circ \phi_{\tau}^*,\overline{\psi}),\]
where
\begin{equation*}
R=\begin{cases}
\mathrm{Sym}^2&\text{ if }\G\text{ is odd special orthogonal};\\
\wedge^2&\text{ if }\G\text{ is even orthogonal or symplectic};\\
\mathrm{As}^+&\text{ if }\G\text{ is even unitary};\\
\mathrm{As}^-&\text{ if }\G\text{ is odd unitary}.\\
\end{cases}
\end{equation*}
\end{enumerate}
\end{proposition}
\begin{proof}
Part (\ref{Propplanchereldecomp1}) follows from \cite{Shahidi}, multiplicativity of the Plancherel measure and of gamma factors, and compatibility of the local Langlands correspondence for general linear groups with Langlands--Shahidi gamma factors of pairs.   (See \cite[Proposition 16.6]{Kakuhama} for a similar statement for inner forms of general linear groups.)

{It remains to show (\ref{Propplanchereldecomp2}).  For a tempered representation~$\pi$ of a symplectic, unitary, or quasi-split orthogonal group~$\U(\V,h)$ and a tempered representation~$\tau$ of~$\GL_k(\E)$, Gan--Ichino in \cite[A.7]{GanIchino16} explain that this equality follows from the normalized intertwining relation of \cite[Theorem 2.3.1]{Arthur}, \cite[Lemma 2.2.3]{KMSW}, and \cite[Proposition 3.3.1]{Mok}.  In \cite[Proof of Proposition 3.1]{BinXu}, an analogous argument is given for $\pi$ a supercuspidal representation of an odd split special orthogonal group~$\G$ and~$\tau$ unitary supercuspidal of~$\GL_k(\F)$, this part of the proof extending to the more general tempered case. For non-quasi-split odd special orthogonal groups the required normalized intertwining relation is given in \cite[Lemma 4.3]{MR4776199}.   Hence for all classical groups of Definition \ref{defclassgroup}, (\ref{Propplanchereldecomp2}) holds for~$\pi$ and~$\tau$ tempered.  (Note for unitary groups, with~$\pi$ irreducible and~$\tau$ square integrable, the equality is also contained in \cite[Theorem 2.5.1]{ChenZou2}.)
 }
%
%

By the Langlands quotient theorem and multiplicativity of both sides, using (\ref{Propplanchereldecomp1}) for the general linear groups part of the Levi, we can reduce to the case where~$\pi$ is tempered.  

Furthermore, by multiplicativity of the Plancherel measure, multiplicativity of gamma factors, and as the local Langlands correspondence for general linear groups is compatible with parabolic induction, we can reduce to the case where~$\tau$ is supercuspidal and~$\pi$ is tempered.  So it is sufficient to show we can further reduce to where~$\tau$ is unitary supercuspidal, and this is possible as twisting~$\tau$ by an unramified character~$|\det|^{r}$ translates each variable~$s$ by~$r$ (noting that in each case~$R |\det |^r=|\det|^{2r}$).
\end{proof}

For~$\phi_1,\phi_2:\W_\F\rightarrow \GL_n(\mathbb{C})$ semisimple parameters, denote by~$\Gamma(\phi_1,\phi_2)$ the divisor of the rational function~$\gamma(s,\phi_1^* \otimes \phi_2, \psi_\E) \gamma(-s,\phi_1 \otimes \phi_2^*, \overline{\psi_{\E}})$.  As the $\epsilon$-factor is a monomial in~$q^{-s}$, this is equal to the divisor of the ratio of~$L$-factors
\[\frac{L(1-s,\phi_1\otimes\phi_2^*)L(1+s,\phi_1^*\otimes\phi_2)}{L(s,\phi_1^*\otimes\phi_2)L(-s,\phi_1\otimes\phi_2^*)}\]
(and does not depend on~$\psi$).

\begin{proposition} 
\label{semisimpleconverse}
Let~$\phi_1,\phi_2:\W_\F\rightarrow \GL_n(\mathbb{C})$ be semisimple parameters.   If, 
\[\Gamma(\tau,\phi_1)=\Gamma(\tau,\phi_2),\]
for all irreducible semisimple parameters~$\tau:\W_\F\rightarrow \GL_m(\mathbb{C})$ and all~$m\leqslant n$, then~$\phi_1\simeq \phi_2$.
\end{proposition}
\begin{proof}
Let~$\chi_s$ be the unramified character of~$\W_\F$ such that~$\chi_s(\Fr)=q^{-s}$ and set~$\nu=\chi_1$.   Thus viewing the local factors as a function of~$\chi_s$, we have
\[L(\chi_s,\tau^*\otimes\phi)=\det(1-\tau^*\otimes\phi\otimes\chi_s(\Fr)_{|\V_{\tau^*\otimes\phi}^{\I_{\F}}})^{-1}.\]
From this description, we see at once that if~$\tau$ and~$\phi$ are irreducible, then $L(\chi_s,\tau^* \otimes \phi)  = 1$ if $\tau$ is not an unramified twist of $\phi$, and otherwise has a pole of order $1$ precisely at those $\chi_s$ such that $\tau\simeq \phi \otimes \chi_s$.

Viewing~$\Gamma(\tau,\phi)$ as a formal sum of unramified characters by the above translation, considering each~$L$-factor in turn, it follows that the divisor~$\Gamma(\tau, \phi)$ has contributions of:
\begin{enumerate}[(i)]
\item poles of order $1$ at those characters $\chi_s$ such that $\tau\simeq\phi \chi_s \nu$;
\item poles of order $1$ at those characters $\chi_s$ such that $\tau\simeq\phi \chi_s \nu^{-1}$;
\item zeroes of order $2$ at those characters $\chi_s$ such that $\tau\simeq\phi \chi_s$.
\end{enumerate}
For an arbitrary semisimple parameter $\phi$, as the divisor is additive with respect to direct sums, it is determined by the irreducible summands in the reducible case.

More precisely, for a fixed irreducible $\tau:\W_\F\rightarrow \GL_n(\mathbb{C})$, let $\H_{\tau}$ be the group of unramified characters $\chi$ such that $\chi \tau\simeq\tau$. and let $\G_{\tau}=\Hom(\W_\F/\I_\F,\mathbb{C}^\times)/\H_{\tau}$.  Then $[\chi] \mapsto \tau \chi^{-1}$ defines a bijection from $\G_{\tau}$ to the set of unramified twists of $\tau$, where $[\chi]$ denotes the class of $\chi$ in $\G_{\tau}$.

With this notation, the divisor $\Gamma(\tau, \phi)$ is equal to the sum, over $[\chi] \in \G_{\tau}$ of the expression:
\[m_{[\chi]} (-[\chi \nu] + 2[\chi] -[\chi \nu^{-1}])\]
where $m_{[\chi]}$ is the multiplicity of $\chi^{-1} \tau$ in $\phi$.  

Put another way, if we regard both the summands of $\phi$ isomorphic to unramified twists of $\tau$ and the divisor $\Gamma(\tau,\phi)$ as expressions in the group ring $\mathbb{Z}[\G_{\tau}]$, then the map that computes the latter from the former is simply multiplication by~$-[\nu] + 2[1] - [\nu^{-1}]=([1]-[\nu])([1]-[\nu^{-1}])$.  Since this is a non-zero divisor in $\mathbb{Z}[\G_{\tau}]$, this map is injective; that is, we can recover the multiplicites of all summands of $\phi$ that are unramified twists of $\tau$ from $\Gamma(\tau,\phi)$, so we have proven the desired converse theorem.  
\end{proof}

This semisimple converse theorem is inspired by the following tempered converse theorem of Gan and Savin:
 \begin{proposition}[{\cite[Lemma A.6]{GanIchino16} and~\cite[Lemma 12.3]{gan_savin}}]\label{temperedconverse}
 Let~$\phi_1,\phi_2$ be tempered~$\SL_2$-parameters for~$\GL_n(\F)$.  If 
\[\Gamma(\tau,\phi_1)=\Gamma(\tau,\phi_2),\]
for all irreducible semisimple parameters~$\tau:\W_\F\rightarrow \GL_m(\mathbb{C})$ and all~$m\leqslant n$, then~$\phi_1\simeq \phi_2$.
 \end{proposition}
 
The proof of this tempered converse theorem is simpler than that of the semisimple converse theorem as no cancellation can occur in the products of~$L$-factors considered; note that the added tempered hypothesis allows Gan and Savin to deduce that the full~$\SL_2$-parameters are isomorphic (rather than just their semisimplifications).

\subsection{Compatibility with parabolic induction}

Compatibility with parabolic induction, Conjecture \ref{ParaindLanglands}, is built into the construction of local Langlands for general linear groups using the classification of Bernstein and Zelevinsky, and for symplectic and split special orthogonal groups is shown in \cite{Moussaoui}.  Using Plancherel measures, gamma factors, and the semisimple converse theorem we establish this conjecture for classical groups (in the special case of symplectic groups and split special orthogonal groups recovering Moussaoui's result by a different method):

\begin{proposition}
\label{parabolicinduction:theorem}
Let~$\G$ be a classical~$p$-adic group and~$\P$ a parabolic subgroup of~$\G$ with Levi decomposition~$\P=\M \N$.  Let~$\rho$ be an irreducible representation of~$\M$ and~$\pi$ be an irreducible subquotient of~$i^{\G}_{\M,\P}(\rho)$.  Then~the semisimple parameters~$\iota_{\M,\G}\circ(\mathcal{LL}_{\M}(\rho))_{ss}$ and~$(\mathcal{LL}_{\G}(\pi))_{ss}$ are~conjugate in~$\widehat{\G}(\mathbb{C})$. 
\end{proposition}
\begin{proof}

We can decompose~$\M=\M_0\times \M_1$ and~$\rho=\rho'\otimes \pi'$ with~$\rho'=\bigotimes_{i=1}^r \rho_i$ an irreducible representation of~$\M_0=\prod_{i=1}^r\GL_{n_i}(\E)$ a product of general linear groups and~$\pi'$ an irreducible representation of the classical group~$\M_1$.

Fix an integer $k$ and a supercuspidal irreducible representation $\tau$ of~$\GL_k(\E)$.   By multiplicativity of the Plancherel measure, Proposition \ref{multofplancherel}, we can write an equality of rational functions in $q^{-s}$:
\begin{equation}\label{planchereleq1}\mu(\tau_s\otimes \pi) = \mu(\tau_s\otimes\pi')\prod_{i=1}^r\mu(\tau_s\otimes \rho_i)\mu(\tau_s\otimes {(\rho_i^c)}^{\vee}),\end{equation}
where all Plancherel measures considered are normalized with respect to the additive character~$\psi$ or with respect to~$\psi_\E$.

Let~$\phi'=\mathcal{R}\circ \mathcal{LL}_\G(\pi')$, $\phi_i= \mathcal{LL}_{\GL_{n_i}(\E)}(\rho_i)$, and $\phi_{\tau} = \mathcal{LL}_{\GL_k(\E)}(\tau)$.  Applying Proposition \ref{Propplanchereldecomp} to the right hand side of Equation \ref{planchereleq1}, and using multiplicativity of gamma factors and that~$\mathcal{LL}_{\GL_{n_i}(\E)}((\rho_i^c)^\vee)\simeq (\phi_i^ c)^*$, we obtain
\begin{align}\label{multplanchereleq}
\mu(\tau_s\otimes \pi) &= 
\gamma(2s,R'\circ \phi_\tau,\psi)\gamma(-2s,R'\circ \phi_{\tau}^*,\overline{\psi})\\
\notag&\gamma(s,\phi_{\tau}\otimes(\phi'^*\oplus{\textstyle\bigoplus}_{i=1}^r(\phi_i^*\oplus(\phi_i^c))),\psi_{\E})\gamma(-s,\phi_{\tau}^*\otimes (\phi'\oplus{\textstyle\bigoplus}_{i=1}^r(\phi_i\oplus(\phi_i^c)^*)),\overline{\psi_{\E}}),
\end{align}
where~$\R'$ is as in Proposition~\ref{Propplanchereldecomp}. 

On the other hand, applying~Proposition \ref{Propplanchereldecomp} to~$\mu(\tau_s\otimes \pi) $ directly, we also have
\begin{equation}\label{Planchereldirect}\mu(\tau_s\otimes\pi)=\gamma(s,\phi_{\tau}\otimes\phi^*,\psi_{\E})\gamma(-s,\phi_{\tau}^*\otimes \phi,\overline{\psi_{\E}})\gamma(2s,R\circ \phi_\tau,\psi)\gamma(-2s,R\circ \phi_{\tau}^*,\overline{\psi}).\end{equation}
Moreover,~$R=R'$ as~$\G$ has the same \emph{type} as~$\G'$ (odd special orthogonal, even orthogonal, symplectic, odd unitary, or even unitary).  Therefore, from equations Equations \ref{multplanchereleq} and \ref{Planchereldirect}, we obtain
\begin{align*}
\gamma(s,\phi_{\tau}\otimes\phi^*,\psi_{\E})&\gamma(-s,\phi_{\tau}^*\otimes \phi,\overline{\psi_{\E}})\\&=\gamma(s,\phi_{\tau}\otimes(\phi'^*\oplus{\textstyle\bigoplus}_{i=1}^r(\phi_i^*\oplus(\phi_i^c))),\psi_{\E})\gamma(-s,\phi_{\tau}^*\otimes (\phi'\oplus{\textstyle\bigoplus}_{i=1}^r(\phi_i\oplus(\phi_i^c)^*)),\overline{\psi_{\E}}).\end{align*}
Letting~$\tau$ vary, so that~$\phi_{\tau}$ varies over all irreducible semisimple parameters, and applying Proposition \ref{semisimpleconverse}, we see that, 
up to semisimplification the Langlands parameter~$\phi$ is equal to the sum $\phi' \oplus \bigoplus_{i=1}^r (\phi_i \oplus (\phi_i^c)^*)$.  In other words,~$\mathcal{R}(\mathcal{LL}_\G(\pi))$ is (up to semisimplification)~$\GL_m(\mathbb{C})$-conjugate to~$\mathcal{R}(\iota_{\M,\G}\circ\mathcal{LL}_\G(\rho))$.  By \cite[Section 8]{GGP12} (see Section \ref{Sectparametersasselfdual}),~$\mathcal{LL}_\G(\pi)$  is thus (up to semisimplification)~$\widehat{\G}(\mathbb{C})$-conjugate to the correct parameter.
\end{proof}

\begin{remark}
For tempered irreducible representations with tempered supercuspidal support, using Gan and Savin's tempered converse theorem (Theorem \ref{temperedconverse}) in place of the semisimple converse theorem in the above argument one can deduce that the full (tempered)~$\SL_2$-parameters are isomorphic.
\end{remark}

\subsection{Compatibility with automorphisms of the coefficient field}
\begin{proposition}
\label{automorphisms:theorem}
Let~$\G$ be a classical~$p$-adic group,~$\pi$ be an irreducible representation of~$\G$, and~ $\sigma:\mathbb{C}\to \mathbb{C}$ be an automorphism of fields fixing~$\sqrt{q}$.   Then~$\mathcal{LL}_\G(\pi^\sigma)_{ss}$ is conjugate to $(\mathcal{LL}_{\G}(\pi)^{\sigma})_{ss}$ in~$\widehat{\G}(\mathbb{C})$.
\end{proposition}
\begin{proof}
{We prove this using properties of Plancherel measures, gamma factors, and the semisimple converse theorem.}  We extend~$\sigma$ to an automorphism of~$\mathbb{C}(q^{-s})$ via its natural action on the coefficients and via the trivial action on~$q^{-s}$. {In particular,~$\sigma$ acts on the gamma factors and Plancherel measures which are considered as rational functions in~$\mathbb{C}(q^{-s})$.}

{The Plancherel measure, as defined in Appendix \ref{apdx:plancherel}, depends on a choice of Haar measure on the pair of opposite unipotent radicals $\U\times \overline{\U}$ (in the notation of the appendix), which following \cite[B.2]{GanIchino14} we fix by choosing an additive character~$\psi$ of~$\F$.  We claim that changing~$\psi$ to~$\sigma\circ\psi$ defines the same Haar measure, and that this measure takes values in~$\mathbb{Q}(\sqrt{q})$ on compact open subgroups.  Indeed, following the construction of ibid., the measure is defined as a product of self-dual Haar measures on~$\E$ and~$\F$ defined with respect to~$\psi$ and~$\psi_\E$ (with an additional fixed factor of~$|2|_{\F}^{-k}$ in the odd special orthogonal case), and these self-dual Haar measures only depend on the conductor of the character and take values in~$\mathbb{Q}(\sqrt{q})$ on compact open subgroups  (cf.~\cite[35.2]{BH06}).
%
%
%
} 
%

{Let~$\tau$ be an irreducible representation of~$\GL_k(\E)$.  It now} follows from Proposition~\ref{prop:functoriality}, that
\begin{equation}
\label{plancherelsigma1}\mu_{\sigma\circ \psi}((\tau^\sigma)_s\otimes \pi^\sigma)=\sigma(\mu_{\psi}(\tau_s\otimes\pi)).
\end{equation}

For~$\K\in\{\E,\F\}$, the uniqueness of local epsilon factors and their explicit definition as Tate's constants in the~$\GL_1(\K)$ setting implies that, for any representation~$\varphi:\W_\K\rightarrow\GL_n(\mathbb{C})$, we have
 \begin{equation}
 \label{gammasigma}\sigma(\gamma(s,\varphi,\psi_{\K}))=\gamma(s,\varphi^\sigma,\sigma\circ \psi_{\K}).\end{equation}

Write~$\phi=\mathcal{R}\circ \mathcal{LL}_{\G}(\pi)$,~$\phi'=\mathcal{R}\circ \mathcal{LL}_{\G}(\pi^\sigma)$, and~$\phi_{\tau}=\mathcal{LL}_{\GL_k(\F)}(\tau)$.  Now~$\mathcal{LL}_{\GL_k(\E)}$ is equivariant for~$\sigma$, which follows from its characterization in terms of local factors and the Galois action on them -- see \cite[Theorem 3.2]{BHDavenport} for this action on the~$\GL_k(\E)$-side.  Hence writing~$\phi_{\tau}^\sigma=(\phi_{\tau})^\sigma$, we have \[\phi_{\tau}^\sigma=\mathcal{LL}_{\GL_k(\E)}(\tau^\sigma).\]  

Thus, by Proposition \ref{Propplanchereldecomp} and {Equation \ref{plancherelsigma1}, we have
\begin{align*}
\gamma(s,\phi_{\tau}^\sigma\otimes\phi'^*,\sigma\circ \psi_{\E})&\gamma(-s,(\phi_{\tau}^\sigma)^*\otimes \phi',\overline{\sigma\circ\psi_{\E}})\gamma(2s,R\circ \phi_{\tau}^\sigma,\sigma\circ\psi)\gamma(-2s,R\circ (\phi_{\tau}^\sigma)^*,\overline{\sigma\circ \psi})=\\
\notag &\sigma(\gamma(s,\phi_{\tau}\otimes\phi^*,\psi_{\E})\gamma(-s,\phi_{\tau}^*\otimes \phi,\overline{\psi_{\E}})\gamma(2s,R\circ \phi_\tau,\psi)\gamma(-2s,R\circ \phi_{\tau}^*,\overline{\psi})).
\end{align*}
We have~$(\phi^*)^\sigma=(\phi^\sigma)^*$,~$(\phi_{\tau}^*)^\sigma=(\phi_{\tau}^\sigma)^*$, $R\circ \phi_{\tau}^\sigma= (R\circ \phi_{\tau})^\sigma$}, {and for~$\K\in\{\E,\F\}$ we have~$\sigma\circ \overline{\psi_{\K}}=\overline{\sigma\circ\psi_{\K}}$ as~$\overline{\psi_{\K}}=\psi_{\K}^{-1}$. Hence, applying Equation \ref{gammasigma}, we find
\begin{align*}
\gamma(s,\phi_{\tau}^\sigma\otimes\phi'^*,\sigma\circ\psi_{\E})&\gamma(-s,(\phi_{\tau}^\sigma)^*\otimes \phi',\overline{\sigma\circ \psi_{\E}})=
\gamma(s,\phi_{\tau}^ \sigma\otimes(\phi^*)^\sigma,\sigma\circ \psi_{\E})\gamma(-s,(\phi_{\tau}^*)^\sigma\otimes \phi^\sigma,\overline{\sigma\circ \psi_{\E}}).\end{align*}}
Hence by Proposition \ref{semisimpleconverse}, we deduce that, up to semisimplification,~$\phi'$ and~$\phi^\sigma$ are conjugate in~$\GL_m(\mathbb{C})$.  By \cite[Section 8]{GGP12} (see Section \ref{Sectparametersasselfdual}),~$(\mathcal{LL}_\G(\pi^\sigma))_{ss}$  is thus~$\widehat{\G}(\mathbb{C})$-conjugate to~$(\mathcal{LL}_\G(\pi)^\sigma)_{ss}$.
\end{proof}

\subsection{Compatibility with group isomorphisms}
\begin{proposition}
\label{isomorphisms:theorem}
Let~$\G$ be a classical~$p$-adic group,~$\pi$ be an irreducible representation of~$\G$, and~ $\alpha:\G\to \G'$ be an isomorphism of reductive groups. Let ${^\alpha\pi}$ be the irreducible representation of $\G'$ obtained by pre-composing with $\alpha^{-1}$. Then the induced isomorphism ${^L\alpha}:{^L\G'}(\mathbb{C})\to {^L\G}(\mathbb{C})$ (which is well-defined up to $\widehat{\G}$-conjugacy) takes the $\widehat{\G'}(\mathbb{C})$ conjugacy class of $\mathcal{LL}_{\G'}({^\alpha\pi})_{ss}$ to the~$\widehat{\G}(\mathbb{C})$-conjugacy class of $\mathcal{LL}_{\G}(\pi)_{ss}$.\end{proposition}
\begin{proof}
Write~$\phi=\mathcal{R}\circ \mathcal{LL}_{\G}(\pi)$, and $\phi'=\mathcal{R}\circ {^L\alpha}\circ \mathcal{LL}_{\G'}({^{\alpha}\pi})$. If we show that 
\[\phi,\ \phi':\W_F\times \SL_2(\mathbb{C})\to \GL_m(\mathbb{C})\]
are isomorphic in $\GL_m(\mathbb{C})$, then by the properties of $\mathcal{R}$ (see Section \ref{Sectparametersasselfdual}), the $\widehat{\G}(\mathbb{C})$-conjugacy class of $\mathcal{LL}_G(\pi)$ is equivalent to that of ${^L\alpha}\circ \mathcal{LL}_{\G'}({^{\alpha}\pi})$, as desired. Let~$\phi_{\tau}=\mathcal{LL}_{\GL_k(\F)}(\tau)$. We now use an argument similar to that of Proposition~\ref{automorphisms:theorem} to show that $\phi$ is equivalent to $\phi'$ after semisimplification.

By Appendix~\ref{section:compatibility_with_isom}, we have
\begin{equation}
\label{plancherelsigma}\mu_{\psi}^\G(\tau_s\otimes \pi)=\mu_{\psi}^{\G'}(\tau_s\otimes{^\alpha\pi}),
\end{equation}
where we have identified $\mathbb{C}[\M/\M^{\circ}]\simeq \mathbb{C}[\M'/(\M')^\circ]\simeq \mathbb{C}[q^{\pm s}]$, where $\M = \GL_k(\E)\times \G$ and $\M'$ is $\GL_k(\E)\times \G'$.

The gamma factor associated to $\mathcal{R'}\circ \mathcal{LL}_{G'}({^\alpha}\pi)$ is an invariant of the isomorphism class of $\mathcal{R'}\circ \mathcal{LL}_{G'}({^\alpha}\pi)$, so we can and do choose our representation $\mathcal{R'} = \mathcal{R}\circ {^L\alpha}$.
Now, by Proposition \ref{Propplanchereldecomp} and Equation \ref{plancherelsigma}, we have
\begin{align*}
\gamma(s,\phi_{\tau}\otimes{\phi^*}',\psi_{\E}^\sigma)&\gamma(-s,\phi_{\tau}^*\otimes \phi',\overline{\psi_{\E}})\gamma(2s,R\circ \phi_{\tau},\psi)\gamma(-2s,R\circ \phi_{\tau}^*,\overline{\psi})=\\
\notag &\gamma(s,\phi_{\tau}\otimes\phi^*,\psi_{\E})\gamma(-s,\phi_{\tau}^*\otimes \phi,\overline{\psi_{\E}})\gamma(2s,R\circ \phi_\tau,\psi)\gamma(-2s,R\circ \phi_{\tau}^*,\overline{\psi}).
\end{align*}
We have~${\phi^*}'={\phi'}^*$, whence
\[\gamma(s,\phi_{\tau}\otimes\phi'^*,\psi_{\E}^\sigma)\gamma(-s,\phi_{\tau}^*\otimes \phi',\overline{\psi_{\E}})=
\gamma(s,\phi_{\tau}\otimes\phi^*,\psi_{\E})\gamma(-s,\phi_{\tau}^*\otimes \phi,\overline{\psi_{\E}}).\]
Hence by Proposition \ref{semisimpleconverse}, we deduce that, after semisimplification,~$\phi'$ and~$\phi$ are conjugate in~$\GL_m(\mathbb{C})$. \end{proof}

\begin{remark}
{It is already known that~$\mathcal{LL}_{\GL_k(\F)}$ is compatible with isomorphisms (even without semisimplification) by \cite[Proposition 5.2.5]{Haines}. For classical~$p$-adic groups, Proposition~\ref{isomorphisms:theorem} establishes Conjecture \ref{Hainesisomorphisms:conjecture} -- the weakened form of \cite[Conjecture 5.2.4]{Haines} where the parameters are considered only up to semisimplification.}
\end{remark}

\subsection{Compatibility with integrality}
Choose an isomorphism~$\Ql\simeq \mathbb{C}$ to obtain an~$\ell$-adic correspondence for classical groups.  
  The Langlands parameters of supercuspidal representations are discrete by \cite[Section 6]{Arthur}, \cite{Mok}, \cite[Theorem 1.6.1]{KMSW}, {and \cite[Theorem 1.2]{MR4776199}}.  Moreover, as the centre of~$\G$ is compact all central characters are integral, hence all supercuspidal representations by \cite[II 4.12]{Vig96}, and all Frobenius semisimple discrete parameters are integral by Proposition \ref{propintegralparameters}.
Therefore, by compatibility with parabolic induction (Proposition \ref{parabolicinduction:theorem}), Proposition \ref{propintegralparameters} and \cite[Corollary 1.6]{DHKMfiniteness}, we obtain:

\begin{proposition}\label{integralclass}
Let~$\G$ be a classical $p$-adic group.  An irreducible representation of~$\G$ is integral if and only if its associated~$\ell$-adic Langlands parameter is integral.  
\end{proposition}

\subsection{The extended packet conjecture}
We prove Conjecture \ref{genericpacket} for quasi-split classical groups:

\begin{proposition}\label{extendedpacketclass}
Let~$\G$ be a quasi-split classical~$p$-adic group and~$(\U,\psi)$ be a Whittaker datum for~$\G$.  In each extended~$L$-packet of~$\G$ there exists a unique~$\psi$-generic representation.
\end{proposition}

\begin{proof}
By Proposition \ref{extendedpacketunderconjectures}, assuming Conjecture \ref{LLcorresp} for~$\G$, it suffices to check Conjectures \ref{temperedpacketconj} and \ref{Davidsmaxorbitconj}.  In fact, we do not need properties \eqref{LLmain3}, \eqref{LLmain4} of Conjecture \ref{LLcorresp}, as they were not used in the proof of the proposition, and so Conjecture \ref{LLcorresp} follows from the main theorems of the works cited in Section \ref{LLforclassgroups}.

The existence statement of Conjecture \ref{temperedpacketconj} -- the \emph{tempered packet conjecture} -- is established for symplectic groups and split odd special orthogonal groups in \cite[Proposition 8.3.2] {Arthur}, and for quasi-split unitary groups in~\cite[Corollary 9.2.4]{Mok}.   For the uniqueness statement of Conjecture  \ref{temperedpacketconj} see \cite{Atobe} (see also the introduction of ibid.~for a history of previous proofs of the same result). 

By Proposition \ref{Davidsmaxorbitprop}, Conjecture  \ref{Davidsmaxorbitconj} is equivalent to the Gross--Prasad--Rallis Conjecture \ref{GrossPrasadRallis}.  Gan and Ichino in~\cite[Proposition B.1]{GanIchino16} prove Conjecture \ref{GrossPrasadRallis} using other properties of the expected local Langlands correspondence which they remark are known for classical groups \cite[B.2]{GanIchino16}. 
\end{proof}

\subsection{The semisimple local Langlands correspondence for classical groups}
It follows that the semisimplifications of the local Langlands correspondences for classical groups of Arthur et al. are ``semisimple correspondences" in the sense of the last section:
\begin{corollary}\label{semisimpleLLforclass}
The semisimple local Langlands correspondence for a classical $p$-adic group exists and satisfies the basic desiderata of Definition \ref{semisimplecorrespdef}.
\end{corollary}

\begin{proof}
By Proposition \ref{semisimplecorresp}, and Propositions \ref{parabolicinduction:theorem}, \ref{automorphisms:theorem}, \ref{isomorphisms:theorem}, \ref{extendedpacketclass}, and their counterparts for general linear groups, it remains to check the following basic properties:
\begin{enumerate}
\item Finite fibres, surjectivity in the quasi-split case, and the Langlands classification (which is built into the construction starting from the tempered case) follow from the main theorems.
\item compatibility with unramified twisting and the central character condition both follow as the centre of a classical group is compact (so these conditions are empty for~$\G$), and for Levi subgroups these follow from compatibility with the correspondence for general linear groups.
\end{enumerate}
\end{proof}

\section{Local Langlands in families}

\subsection{The universal unramified character}
{For the remainder we suppose~$\F/\mathbb{Q}_p$.  Let~$\M$ be a Levi
subgroup of~$\G$ and let $\Psi_{\M}$ denote the torus of unramified
characters of $\M$. For any $\mathbb{Z}[1/p]$-algebra $\R$, the group
$\Psi_{\M}(\R)=\Hom(\M/\M^{\circ},\R^{\times})$ acts on smooth
$\R[\M]$-modules by twisting. Through the \emph{Kottwitz isomorphism} (cf.~\cite[Section 3.3]{Haines})
\[\Psi_{\M}\simeq (((\Z_{\widehat{\M}})^{\I_\F})_{\Fr})^\circ,\]
the torus $\Psi_{\M}$ also acts on the coarse quotient
$\underline{\Z}^1(\W_\F,\widehat\M)\sslash \widehat{\M}$. Indeed, multiplication of
cocycles yields an action of $(\Z_{\widehat{\M}})^{\I_\F}=
\underline{\Z}^1(\W_\F/\I_\F,(\Z_{\widehat{\M}})^{\I_\F})\subset
\underline{\Z}^1(\W_\F,\Z_{\widehat{\M}})$  on
$\underline{\Z}^1(\W_\F,\widehat\M)$ that descends to an action on
$\underline{\Z}^1(\W_\F,\widehat\M)\sslash \widehat{\M}$, and the latter factors over
the coinvariants $((\Z_{\widehat{\M}})^{\I_\F})_{\Fr}$.}

{Through the Kottwitz isomorphism again, the universal unramified character $\chi_{\univ,\M}: \M/\M^{\circ}\To{}
\mathbb{Z}[1/p][\M/\M^{\circ}]^{\times}$ of Definition \ref{def_inertial} corresponds to  a 
$\mathbb{Z}[1/p][\M/\M^\circ]$-point of $(((\Z_{\widehat{\M}})^{\I_\F})_{\Fr})^\circ$ that
we denote by $z_{\univ,\M}$.
Given an $\R$-point $\varphi$ of $\underline{\Z}^1(\W_\F,\widehat{\M})\sslash\widehat{\M}$, we then
define its \emph{universal unramified twist}  by promoting $\varphi$ to an
$\R[\M/\M^{\circ}]$-point and then taking the product $\varphi\cdot z_{\univ,\M} $, which
is an $\R[\M/\M^{\circ}]$-point of $\underline{\Z}^1(\W_\F,\widehat{\M})\sslash\widehat{\M}$.}


{By~\cite[Section 3.3]{Haines}, for any field $\K$ of characteristic zero, and any $\K$-point $x$ of $\Spec
(\mathbb{Z}[1/p][\M/\M^\circ])$, 
 the $\K$-point  $z_{x}:=(z_{\univ,\M})_x$ corresponds to the character $\chi_{x}:=(\chi_{\univ,\M})_x$
 under local Langlands for (unramified) characters of $\M$. In the notation of
 item (\ref{LLmain3}) of Conjecture  \ref{LLcorresp}, this
 means that for $\chi=\chi_{x}$, the cocycle $z_{\chi}$  belongs to the cohomology class $z_{x}$.}

\subsection{Interpolating a semisimple correspondence in families; the banal case}
We need the following definitions from \cite{DHKM}:
\begin{enumerate}
\item A prime~$\ell$ is called~\emph{$\LG$-banal} if the fibre~$\underline{\Z}^1(\W_\F^0,\widehat{\G})_{\Fl}$ is reduced.
\item The integer $\N_{\widehat{\G}}$ from \cite[Section 4]{DHKM}.
\item Let~$\N_{\LG}$ denote the product of all primes dividing~$\N_{\widehat{\G}}$ and all non-$\LG$-banal primes.
\end{enumerate}

\begin{definition}\label{stablecentredef}
Suppose~$\mathscr{C}=(\mathscr{C}_{\M})$ is a semisimple correspondence for~$\G$ over~$\Qbar$. Let~$\R$ be a subring of~$\Qbar$. The \emph{stable centre}~$\mathfrak{Z}^{\St}_{\R}(\G)$ subring of~$\mathfrak{Z}_{\G,\R}$ consisting of all elements which act uniformly on all extended~$L$-packets for~$\mathscr{C}$ over~$\Qbar$.
%
%
\end{definition}

\begin{theorem}\label{axiomaticLLIF}
Suppose~$\mathscr{C}=(\mathscr{C}_\M)$ is a semisimple correspondence for~$\G$ over~$\Qbar$.  
\begin{enumerate}
\item\label{Axiomatic1} There exists a unique quasi-finite morphism
\[\mathscr{C}\mathrm{IF}_\G:(\mathfrak{R}_{\LG,\mathbb{Q}(\sqrt{q})})^{\widehat{\G}}\rightarrow  \mathfrak{Z}_{\G,\mathbb{Q}(\sqrt{q})}\]
compatible with $\mathscr{C}$.   
\item\label{Axiomatic2}  Let~$\N_\G$ denote the product of the non-$\G$-banal primes, then~the image of~$(\mathfrak{R}_{\LG,\mathbb{Z}[\sqrt{q}^{-1},1/\N_{\G}]})^{\widehat{\G}}$ under~$\mathscr{C}\mathrm{IF}_\G$ is contained in~$\mathfrak{Z}_{\G,\mathbb{Z}[\sqrt{q}^{-1},1/\N_{\G}]}$.
\item\label{Axiomatic3}  The image of~$\mathscr{C}\mathrm{IF}_\G$ on~$(\mathfrak{R}_{\LG,\mathbb{Z}[\sqrt{q}^{-1},1/\N_{\G}]})^{\widehat{\G}}$ is contained in the {subrings~$\mathfrak{Z}^{\St}_{\G,\mathbb{Z}[\sqrt{q}^{-1},1/\N_{\G}]}$, and~$\mathfrak{Z}_{\G,\mathbb{Z}[\sqrt{q}^{-1},1/\N_{\G}]}^{\ad}$.}
\item\label{Axiomatic4}  Suppose further~$\G$ is~$\F$-quasi-split.  
Let~$\M_{\G}=\mathrm{lcm}(\N_\G,\N_{\LG})$.  Then, the composition with the natural map~$ \mathfrak{Z}_{\G,\mathbb{Z}[\sqrt{q}^{-1},1/\M_{\G}]}^{\ad}\rightarrow\mathfrak{E}_{\mathbb{Z}[\sqrt{q}^{-1},1/\M_{\G}]}(\G)$, 
\[(\mathfrak{R}_{\LG,\mathbb{Z}[\sqrt{q}^{-1},1/\M_{\G}]})^{\widehat{\G}}\xrightarrow{\mathscr{C}\mathrm{IF}_\G}  \mathfrak{Z}^{\ad}_{\G,\mathbb{Z}[\sqrt{q}^{-1},1/\M_{\G}]}\rightarrow \mathfrak{E}_{\mathbb{Z}[\sqrt{q}^{-1},1/\M_{\G}]}(\G)\]
is an isomorphism, and these maps induce isomorphisms
\[(\mathfrak{R}_{\LG,\mathbb{Z}[\sqrt{q}^{-1},1/\M_{\G}]})^{\widehat{\G}}\simeq \mathfrak{Z}^{\St}_{\G,\mathbb{Z}[\sqrt{q}^{-1},1/\M_{\G}]}\simeq \mathfrak{E}_{\mathbb{Z}[\sqrt{q}^{-1},1/\M_{\G}]}(\G).\]
\end{enumerate}
\end{theorem}

\begin{proof}
Let~$[\M,\rho]_{\G}\in\mathfrak{B}_{\Qbar}(\G)$ be a~$\Qbar$-inertial class, indexing
a~$\Zbar[1/\N_{\G}]$-block.  Choose a representative~$(\M,\rho)\in[\M,\rho]_{\G}$ with
finite order central character.   Hence~$\rho$ is defined over a number field~$\K$
and~$\cO_{\K}$-integral by Proposition \ref{integralreps}.  We
pick a semisimple parameter $\phi$ in the $\widehat{\M}(\Qbar)$-conjugacy class
$\mathscr{C}_{\M}(\M,\rho)$.  

\emph{We first consider the statements over $\Zl$ for appropriate $\ell$.} 
For any embedding~$\Qbar\hookrightarrow\Ql$, by Proposition \ref{propintegralparameters}
and~$(\mathscr{C}4)$, the $\ell$-adic parameter $\phi\otimes\Ql$ is~integral.  Thus, after maybe conjugating by
some~$m\in\widehat{\M}(\Ql)$, the parameter~$\phi_{\ell}=(\phi\otimes\Ql)^m$ takes values
in~$\LM(\Zl)$ and is $\ell$-adically continuous. By the universal property
of~$\mathfrak{R}_{\LM,\mathbb{Z}_{\ell}}^e$ (Theorem \ref{Modulitheorem}
(\ref{Moduliuniversal})), this integral parameter corresponds to some map of~$\Zl$-algebras
$\mathfrak{R}_{\LM,\Zl}^e\rightarrow \Zl$, for $e$ sufficiently large.  Accordingly, the
universal unramified twist~$\phi_{\ell}\cdot z_{\univ,\M}$ defined above corresponds to a map of~$\Zl$-algebras
\[\Psi_{[\M,\rho],\ell}^{\M}:(\mathfrak{R}_{\LM,\Zl}^e)^{\widehat{\M}}\rightarrow \Zl[\M/\M^\circ],\]
and its pushforward along $\iota_{\M,\G}:\LM\hookrightarrow \LG$ corresponds to the composition
\[\Psi_{[\M,\rho],\ell}^{\G}:\,(\mathfrak{R}_{\LG,\Zl}^e )^{\widehat{\G}}\rightarrow (\mathfrak{R}_{\LM,\Zl}^e)^{\widehat{\M}}\rightarrow \Zl[\M/\M^\circ].\]
For any character $\chi:\, \M/\M^{\circ}\To{}\Ql$, the composition
$(\mathfrak{R}_{\LM,\Zl}^e)^{\widehat{\M}}\rightarrow \Zl[\M/\M^\circ]\To{\chi}\Ql$ corresponds to the~$\widehat{\M}(\Ql)$-conjugacy class $\phi_{\ell}\cdot z_{\chi}$, which, by property $(\mathscr{C}3)$, equals~$\mathscr{C}_{\M}(\M,\rho\cdot \chi)$.
{In particular, if~$\chi$ belongs to the stabilizer~$\H_{(\M,\rho)}$ of~$\rho$ under
unramified twisting, then this equals~$\mathscr{C}_{\M}(\M,\rho)$.  
This implies that}
\[\Psi_{[\M,\rho],\ell}^{\M} \left((\mathfrak{R}_{\LM,\Zl}^e)^{\widehat\M}\right) \subset \Zl[\M/\M^\circ]^{\H_{(\M,\rho)}}.\]
On the other hand, Property $(\mathscr{C}6)$  says that $\iota_{\M,\G}(\phi_{\ell}\cdot z_{\chi})$
belongs to $\mathscr{C}_{\G}(\M,\rho\cdot \chi)$.  Since $(\M,\rho\cdot \chi)$ is $\G$-conjugate to $(\M,\rho\cdot\chi^{w})$  for any
$w\in \W_{(\M,\rho)}$, this implies that 
$$ \Psi_{[\M,\rho],\ell}^{\G} \left((\mathfrak{R}_{\LG,\Zl}^e)^{\widehat\G}\right) \subset
\left(\Zl[\M/\M^\circ]^{\H_{(\M,\rho)}}\right)^{\W_{(\M,\rho)}}.$$

We claim that the restricted and corestricted map
\begin{equation}
\Psi_{[\M,\rho],\ell}^{\G}:\,
(\mathfrak{R}_{\LG,\Zl}^e)^{\widehat\G} \To{}
\left(\Zl[\M/\M^\circ]^{\H_{(\M,\rho)}}\right)^{\W_{(\M,\rho)}}\label{morph_R_to_Z}
\end{equation}
 is finite. Indeed, we already know from \cite[Cor. 2.4]{DHKMfiniteness} that the  map 
$(\mathfrak{R}_{\LG,\Zl}^e)^{\widehat{\G} }\rightarrow
(\mathfrak{R}_{\LM,\Zl}^e)^{\widehat{\M} }$ is finite.
Moreover,  denoting by $\A_{\M}$ the maximal split central torus of
$\M$, the composition 
$\Zl[\A_{\M}]\To{}\mathfrak{R}^{e}_{{{\tensor*[^L]\A{_\M}}},\Zl}\To{}(\mathfrak{R}_{\LM,\Zl}^{e})^{\LM}\To{}\Zl[\M/\M^\circ]$ is the
%
%
natural map
twisted by the central
character $\omega_{\rho}$ of $\rho$, so it is finite.

Observe that we have not used any hypothesis on the prime $\ell$ so far. But now, our
assumption that $\ell$ is banal  implies, by  Theorem \ref{banaldecomptheorem} and
Proposition~\ref{Corembed}, that the RHS of (\ref{morph_R_to_Z}) is a  factor of
$\mathfrak{Z}_{\G,\Zl}$. Therefore, taking the product of all maps
$\Psi_{[\M,\rho],\ell}^{\G}$ composed with the projection 
$(\mathfrak{R}_{\LG,\Zl})^{\widehat{\G}}\rightarrow
(\mathfrak{R}_{\LG,\Zl}^{e})^{\widehat{\G}}$ yields a morphism of $\Zl$-algebras
\begin{equation}
\Psi_{\ell}:\,(\mathfrak{R}_{\LG,\Zl})^{\widehat{\G}} \rightarrow
\mathfrak{Z}_{\G,\Zl}\label{mor_R_to_Z_global}
\end{equation}
that is   compatible with the given semisimple
correspondence $\mathscr{C}_{\G,\Ql}$ by design. Note that this compatibility makes this map
unique since the target is reduced and $\ell$-torsion free. For the same reason,
the image of $\Psi$ is contained in $\mathfrak{Z}_{\G,\Zl}^{\rm st}$ by definition of
this ring, and is contained in $\mathfrak{Z}_{\G,\Zl}^{\rm ad}$ {by Property $(\mathscr{C}8)$.}

Suppose now that~$\G$ is~$\F$-quasi-split, and fix a depth~$n$.
Since $\ell$ is banal, we know from Theorem \ref{BHBanalcase} that the
natural map 
$\mathfrak{Z}_{\G,\Zl,n} \rightarrow
  \mathfrak{E}_{\G,\Zl,n}$ is surjective.
  Composing with (\ref{mor_R_to_Z_global}), we thus get a finite
  morphism
  \[\Phi_{\ell,n}: (\mathfrak{R}_{\LG,\Zl})^{\widehat{\G} } \rightarrow
  \mathfrak{E}_{\G,\Zl,n}.\]
Let $e_{n}$ be the sum of all primitive idempotents $\varepsilon$ in
$(\mathfrak{R}_{\LG,\Ql})^{\widehat{\G} }$ such that
$\Phi_{\ell,n}(\varepsilon)\neq 0$. So $e_{n}$ cuts out finitely many connected components of
the space of conjugacy classes of semisimple $\Ql$-parameters, and the ``supercuspidal
supports and genericity'' condition $(\mathscr{C}5)$ tells us that
$\Phi_{\ell,n} :\, e_{n}(\mathfrak{R}_{\LG,\Ql})^{\widehat{\G} } \rightarrow
  \mathfrak{E}_{\G,\Ql,n} $ induces a bijection
on~$\Ql$-points. Since $\Phi_{\ell,n}$ is finite and both source and target are  products of normal
integral $\Ql$-algebras, we deduce that $\Phi_{\ell,n}$ induces an isomorphism
$e_{n}(\mathfrak{R}_{\LG,\Ql})^{\widehat{\G} } \To\sim
  \mathfrak{E}_{\G,\Ql,n}.$  

Let us now assume further that $\ell$ does not divide~$\M_{\G}$. In this case, it follows from
Proposition 6.2 and Theorem 6.7 of \cite{DHKM} that all idempotents of
$(\mathfrak{R}_{\LG,\Ql})^{\widehat{\G} }$ -- in particular $e_{n}$ --  lie in the integral subring
$(\mathfrak{R}_{\LG,\Zl})^{\widehat{\G} }$  and, moreover, that 
$e_{n}(\mathfrak{R}_{\LG,\Zl})^{\widehat{\G} }$ is a product of normal
integral flat $\Zl$-algebras.  As above, by finiteness, it follows that $\Phi_{\ell,n}$ induces
an isomorphism
$e_{n}(\mathfrak{R}_{\LG,\Zl})^{\widehat{\G} } \To\sim
\mathfrak{E}_{\G,\Zl,n}, $  since $\mathfrak{E}_{\G,\Zl,n}$ is reduced and flat over $\Zl$.
Going to the limit over $n$, we get an isomorphism
$$\Phi_{\ell}:\,(\mathfrak{R}_{\LG,\Zl})^{\widehat{\G} } \To\sim \mathfrak{E}_{\G,\Zl}.$$

\medskip
\emph{We now descend the statements to $\mathbb{Z}_\ell[\sqrt{q}]$.}  
For banal $\ell$, Galois-invariance $(\mathscr{C}7)$ and Corollary  \ref{GaloisBC}
imply that $\Psi_{\ell}((\mathfrak{R}_{\LG,\mathbb{Z}_\ell[\sqrt{q}]})^{\widehat{\G}
})\subset\mathfrak{Z}_{\G,\mathbb{Z}_\ell[\sqrt{q}]}$.  Assuming invariance under
automorphisms of $\G$, we even get that 
$\Psi_{\ell}((\mathfrak{R}_{\LG,\mathbb{Z}_\ell[\sqrt{q}]})^{\widehat{\G}})\subset\mathfrak{Z}^{\rm
ad}_{\G,\mathbb{Z}_\ell[\sqrt{q}]}$ hence, by Theorem \ref{theoremmathfrakE}, that
$\Phi_{\ell}(\mathfrak{R}_{\LG,\mathbb{Z}_\ell[\sqrt{q}]})^{\widehat{\G}})
\subset\mathfrak{E}_{\G,\mathbb{Z}_\ell[\sqrt{q}]}$. By descent, $\Phi_{\ell}$ thus
induces an isomorphism
$$ \Phi_{\ell}:\, (\mathfrak{R}_{\LG,\mathbb{Z}_\ell[\sqrt{q}]})^{\widehat{\G}}
\To\sim \mathfrak{E}_{\G,\mathbb{Z}_\ell[\sqrt{q}]}.$$

\medskip
\emph{We now consider the statements over $\Qbar$ and $\mathbb{Q}(\sqrt{q})$.}
Taking up the notation $[\M,\rho]$ and $\phi$ from the beginning of this proof, we would
like to mimic the construction of $\Psi_{\ell}$ over $\Zl$ and $\Ql$, but we can't use the universal
property of $\mathfrak{R}_{\LG,\Zl}^e$ on the nose. Instead, let $\I_{\F}^{e}$ be the open
subgroup of $\I_{\F}$ introduced in Corollary 4.6 of \cite{DHKM}, and let
$\underline \Z^{1}(\W_{\F}/\I_{\F}^{e},\widehat\G)=\Spec(\mathfrak{S}_{\LG}^{e})$ be the corresponding affine
scheme of $1$-cocycles over $\mathbb{Z}[1/p]$. By Proposition 4.7 of
\emph{loc.cit.}, we have a closed embedding 
$\underline \Z^{1}(\W_{\F}/\I_{\F}^{e},\widehat\G)\hookrightarrow
\underline \Z^{1}(\W^{0}_{\F}/\P_{\F}^{e},\widehat\G)$ that induces an isomorphism on the
categorical quotients over $\Qbar$, that is, an isomorphism of $\Qbar$-algebras
$(\mathfrak{R}_{\LG,\Qbar}^{e})^{\hG}\To\sim (\mathfrak{S}_{\LG,\Qbar}^{e})^{\hG}$. Now we
can repeat the argument that we had over $\Zl$.
By universality, the {class $\iota_{\M,\G}\circ (\phi\cdot z_{\rm univ,\Qbar})$} corresponds to a morphism
{$\Psi_{[\M,\rho],\Qbar}^{\G} : (\mathfrak{S}_{\LG,\Qbar}^{e})^{\hG}\To{}\Qbar[\M/\M^{\circ}]$}
such
that
$\Psi_{[\M,\rho],\Qbar}^{\G}((\mathfrak{S}_{\LG,\Qbar}^{e})^{\hG})\subset
(\Qbar[\M/\M^{\circ}]^{\H_{(\M,\rho)}})^{\W_{(\M,\rho)}}$.
Taking the limit of these maps provides the desired morphism
$$\Psi_{\Qbar} : (\mathfrak{R}_{\LG,\Qbar})^{\hG}=(\mathfrak{S}_{\LG,\Qbar})^{\hG}
\To{} \mathfrak{Z}_{\G,\Qbar}.$$
As above, this map descends to $\mathbb Q(\sqrt q)$ and, composing with the action of the
Bernstein center on the Gelfand-Graev representations, we get the isomorphism
$$ \Phi_{\Qbar}:\, (\mathfrak{R}_{\LG,\Qbar})^{\hG}\To\sim \mathfrak{E}_{\LG,\Qbar}$$
that also descends to $\mathbb Q(\sqrt q)$.  The theorem now follows from Lemma \ref{localglobalcoeffs}.
\end{proof}

\begin{remark}
Rather than working locally in the proof of Theorem \ref{axiomaticLLIF}, and piecing together using Lemma \ref{localglobalcoeffs},  we could have followed the last strategy (defining~$\Psi_{\Qbar} $) working directly over $\Zbar[1/\M_{\G}]$ using Proposition 6.2 of \cite{DHKM} and choosing a~$\Zbar[1/\M_{\G}]$-integral representative of the semisimple parameter of~$\rho$ using an analogue of \cite[Corollary 2.10]{DHKMfiniteness} for~$\underline \Z^{1}(\W_{\F}/\I_{\F}^{e},\widehat\G)$.   We have preferred the strategy in the proof as the first part requires no hypothesis on~$\ell$, and along the way we recover Theorem \ref{axiomaticLLIF} \eqref{Axiomatic1} and \eqref{Axiomatic2} over~$\mathbb{Z}[\sqrt{q}^{-1},1/\N_{\G}]$ without needing to invert any additional primes dividing~$\M_{\G}/\N_{\G}$.
\end{remark}

 \subsection{Local Langlands in families for general linear groups}
Given the theory of moduli of Langlands parameters over~$\mathbb{Z}[1/p]$ of \cite{DHKM}, the mild argument used in the proof of Theorem \ref{axiomaticLLIF} to extend from local to global integral coefficients (namely, Lemma \ref{localglobalcoeffs}) also applies to the~$\GL_n(\F)$ theorem of the second and fourth authors \cite{HM}, and we obtain:

\begin{corollary}\label{GLnLLIF}
There are natural isomorphisms
\[\begin{tikzcd}
(\mathfrak{R}_{\GL_n,\mathbb{Z}[\sqrt{q}^{-1}]})^{\GL_n} \arrow[r] &  \mathfrak{Z}_{\GL_n(\F),\mathbb{Z}[\sqrt{q}^{-1}]}\arrow[r] &\mathfrak{E}_{\GL_n(\F),\mathbb{Z}[\sqrt{q}^{-1}]},
\end{tikzcd}\]
where the second is the canonical map, and the first interpolates the semisimple local Langlands correspondence of \cite{HarrisTaylor, Henniart, Scholze}.
\end{corollary}

\appendix

\section{A few abstract lemmas}
We need the following simple lemmas.  The first is a variant of a standard lemma from commutative algebra:

\begin{lemma}\label{scalarextandpros}
Let~$\R$ be a commutative ring, and~$\H$ be a locally profinite group such that there exists a compact open subgroup of~$\H$ of pro-order invertible in~$\R$. 
Let~$Q$ be a finitely generated projective smooth~$\R[\H]$-module,~$\R'$ a
commutative~$\R$-algebra, and~$M$ a smooth~$\R[\H]$-module.  Then the natural map
\[\Hom_{\R[\H]}(Q,M)\otimes \R'\xrightarrow{\sim} \Hom_{\R'[\H]}(Q\otimes \R',M\otimes \R')\]
$f\otimes r'\mapsto r'(f\otimes 1)$ defines an isomorphism.
\end{lemma}

\begin{proof}
  Using the Hom-tensor adjunction, it suffices to prove that the natural map
\[\Hom_{\R[\H]}(Q,M)\otimes_{\R} \R'\xrightarrow{} \Hom_{\R[\H]}(Q,M\otimes_{\R} \R')\]
is an isomorphism.  Writing the $\R$-module $\R'$ as the cokernel of a map of free
$\R$-modules, it suffices to see that $M\mapsto \Hom_{\R[\H]}(Q,M)$  commutes with
cokernels and all coproducts. Commutation with cokernels follows from projectivity,
and commutation with coproducts follows from finite generation, which together with
projectivity implies compactness. In more details, finite generation implies that $Q$ is a
direct factor of some $(\ind_{\U}^{\H} (\R))^{n}$ for some compact open
subgroup $\U$ of $\H$ of pro-order invertible in $\R$. So, compactness of $Q$ follows from
that of $\ind_{\U}^{\H} (\R)$, which itself follows from the fact that the functor of
$\U$-invariants is exact and commutes with all coproducts.
\end{proof}

\begin{lemma}\label{centrelemma}
Let~$\A$ be an~$\R$-algebra with centre~$\Z(\A)$ and~$\R'$ a flat commutative~$\R$-algebra, such that either
\begin{enumerate}
\item $\A$ is a flat~$\R$-module and~$\R$ is Noetherian; or
\item $\A$ is finitely generated over~$\Z(\A)$.  
\end{enumerate}Then~$\Z(\A)\otimes_\R\R'=\Z(\A\otimes_\R \R')$, the centre of~$\A\otimes_\R \R'$.
\end{lemma}
\begin{proof}
The centre of~$\A$ is the kernel of the~$\R$-linear map~$f: \A \rightarrow \prod_{a \in \A} \A$ that is the direct product over~$\A$ of the commutators $b \mapsto [b,a]$.  Since~$\R'$ is flat over~$\R$ we get
\[0\rightarrow \Z(\A)\otimes_{\R}\R'\rightarrow \A {\otimes}_{\R}\R'\xrightarrow{f\otimes 1} \left( \prod_{a \in \A} \A\right)\otimes_{\R}\R'.\]
By \cite[Theorem 1]{Goodearl}, if~$\R$ is Noetherian and~$\A$ is flat, we find that the natural map~\[\phi:(\prod \A)\otimes_{\R}\R'\rightarrow \prod( \A\otimes_{\R}\R')\] is an injection\footnote{This is also true if~$\R'$ is Mittag-Leffler over~$\R$ by \cite[\href{https://stacks.math.columbia.edu/tag/059M}{Tag 059M}]{stacks-project}.}. Now the result follows as the kernel of $
\phi\circ(f \otimes 1):\A\otimes_{\R} \R'\rightarrow \prod( \A\otimes_{\R}\R')$ is precisely~$\Z(\A \otimes_\R \R')$.   

{If~$\A$ is finitely generated over~$\Z(\A)$, then choosing a generating set~$\mathcal{G}$ we can describe the centre of~$\A$ as the kernel of the map~$f': \A \rightarrow \bigoplus_{a \in \mathcal{G}} \A$, and as direct sums commute with tensor products, we can conclude in a similar fashion.}
\end{proof}

To deduce morphisms with global integral coefficients from knowing they define locally integral morphisms we have the following lemma:
\begin{lemma}\label{localglobalcoeffs}
Let~$\K$ be a number field with ring of integers~$\cO$,~$\N$ be an integer, and~$\mathfrak{R},\mathfrak{Z}$ be torsion-free~$\cO[1/\N]$-algebras, and suppose that we have a morphism
\[\phi:\mathfrak{R}\otimes \K\rightarrow \mathfrak{Z}\otimes\K\] 
such that for all prime ideals~$\mathfrak{l}$ prime to~$\N$, and all associated embeddings~$\K\hookrightarrow \Ql$, the morphism~$\phi_{\ell}:\mathfrak{R}\otimes \Ql\rightarrow \mathfrak{Z}\otimes\Ql$ satisfies~$\phi_{\ell}(\mathfrak{R}\otimes \Zl)\leqslant  \mathfrak{Z}\otimes\Zl$. 
\begin{enumerate} 
\item 
Then~$\phi( \mathfrak{R})\leqslant \mathfrak{Z}$, hence~$\phi$ defines a morphism~$\phi:\mathfrak{R}\rightarrow \mathfrak{Z}$ of~$\cO[1/\N]$-algebras.
\item Suppose further~$\phi:\mathfrak{R}\otimes \K\xrightarrow{\sim} \mathfrak{Z}\otimes\K$ is an isomorphism and that it defines isomorphisms locally~$\phi_{\ell}(\mathfrak{R}\otimes \Zl)\simeq  \mathfrak{Z}\otimes\Zl$.  Then~$\phi$ defines an isomorphism~$\phi:\mathfrak{R}\xrightarrow{\sim} \mathfrak{Z}$ of~$\cO[1/\N]$-algebras.
\end{enumerate}
\end{lemma}

\begin{proof}
Consider~$\mathfrak{Z}\leqslant \phi(\mathfrak{R})+ \mathfrak{Z}$.  As~$\phi(\mathfrak{R})\leqslant \mathfrak{Z}\otimes\K$, it implies that the quotient~$(\phi(\mathfrak{R})+ \mathfrak{Z})/\mathfrak{Z}$ is prime to~$\N$ torsion.  However,~$\Zl$ is flat over~$\cO[1/\N]$ and~$\phi_{\ell}(\mathfrak{R}\otimes \Zl)\subseteq  \mathfrak{Z}\otimes\Zl$, hence the~$\mathfrak{l}$-torsion is zero.  Hence the quotient is trivial, and~$\phi(\mathfrak{R})\leqslant \mathfrak{Z}$.  The second part follows by applying the same argument to~$\phi^{-1}$.
\end{proof}

\section{Intertwining operators}\label{Appendix:intertwining}

In this appendix we develop the theory of intertwining operators in a general and purely algebraic way, using the ideas of \cite{Waldspurger} and \cite{datnu}, which probably trace their origin to unpublished lecture notes of Bernstein. The point of doing this is to eventually deduce, in Appendix \ref{apdx:j_functions}, more general and precise versions of well-known properties of the Harish--Chandra $j$-function and Plancherel measure. 

We fix a maximal split torus $\A_0$ of~$\G$ and a minimal parabolic~$\P_0$ of~$\G$ containing~$\A_0$.  A parabolic subgroup containing~$\A_0$ is called \emph{semistandard}, and if, moreover, it contains~$\P_0$, it is called \emph{standard}.  A semistandard (resp.~standard) parabolic subgroup has a unique Levi subgroup containing~$\A_0$ which we call \emph{semistandard} (resp.~\emph{standard}).  Given a semistandard (resp.~standard) parabolic subgroup, we always choose Levi decompositions with semistandard (resp.~standard) Levi factors.

Let $\R$ be a Noetherian $\ZZ[\sqrt{q}^{-1}]$-algebra, let $\P=\M\U_{\P}$ and $\Q = \M\U_{\Q}$ be Levi decompositions of two semistandard parabolic subgroups of $\G$ with the same Levi subgroup $\M$, and fix $\mu$, $\mu'$ two choices of $\ZZ[\frac{1}{p}]$-valued Haar measures on $\U_\P$ and $\U_\Q$. Let $\A_\M$ be the split component of the center of $\M$. Let $\sigma$ be a smooth $\R[\M]$-module. Let $i_\P^\G$ and $r_\P^\G$ be the normalized parabolic induction and restriction functors. Let $\leqslant$ denote the Bruhat ordering on the set $\W_\M \backslash \W_\G / \W_\M$, which is defined by
\[w\leqslant w' \text{ if and only if } \Q w\P \subset \overline{\Q w'\P}.\] 
We define the following subfunctors of $i_\P^\G $:
\begin{align*}
\widetilde{F}_{\Q\P}^{<w}(\sigma) &= \left\{f\in i_\P^\G(\sigma)\ :\ \text{Supp}(f) \cap \left(\bigcup_{w'<w}\Q w'\P \right)=\emptyset\right\}\\
&= \left\{f\in i_\P^\G(\sigma)\ :\ \text{Supp}(f)\cap \overline{\Q w\P} \subset \Q w\P\right\}\\
&= \left\{f\in i_\P^\G(\sigma)\ :\ \text{Supp}(f)\cap  \Q w\P\text{ is compact mod $\P$}\right\},
\end{align*} and similarly define the subfunctor $\widetilde{F}_{\Q\P}^{\leqslant w}\subset \widetilde{F}_{\Q\P}^{< w}$ by replacing $<$ with $\leqslant$. Note that if $w_1\leqslant w_2$ then $\widetilde{F}_{\Q\P}^{\leqslant w_2}\subset \widetilde{F}_{\Q\P}^{\leqslant w_1}$. Any refinement of $\leq$ to a total ordering gives a filtration of $i_\P^\G$ by subfunctors.

The step of the filtration $\widetilde{F}_{\Q\P}^{<1}(\sigma)$ consists of functions $f$ such that $\text{Supp}(f)\cap \Q\P$ is compact mod $\P$, and for such $f$ we can define the map
\begin{align*}
q: \widetilde{F}_{\Q\P}^{<1}(\sigma) &\to \sigma\\
f&\mapsto \int_{\U_\Q/(\U_\Q\cap \U_\P)} f(u)du = \int_{\U_\Q\cap \U_{\overline{\P}}}f(u)du.
\end{align*}
where~$\overline{\P}$ is the opposite parabolic to~$\P$ with respect to~$\M$, and the equality follows from~$\U_{\Q}=(\U_{\Q}\cap\U_{\overline{\P}})(\U_{\Q}\cap\U_{\P})$. 

Let $F_{\Q\P}^{< w}$ be the image of $\widetilde{F}_{\Q\P}^{< w}$ in the quotient map $i_\P^\G(\sigma) \to r_\Q^\G\circ i_\P^\G(\sigma)$, and similarly for $\leqslant$. Then our map $q$ factors through $F_{\Q\P}^{< w}$ and defines an isomorphism
\[F_{\Q\P}^{<1}/F_{\Q\P}^{\leqslant 1}(\sigma) \xrightarrow{\simeq} \sigma.\] 
Given $w$ in $\W_\M \backslash \W_\G / \W_\M$ denote by $\dot{w}$ a choice of lift of a representative of $w$ to an element of $G$. For every $w$ and $\dot{w}$ one can also construct isomorphisms
\[F_{\Q\P}^{<w}/F_{\Q\P}^{\leqslant w}\xrightarrow{\simeq} i_{\M\cap w(\P)}^\M \circ w \circ r_{w^{-1}(\Q)\cap \M}^{\M}(\sigma).\]

Consider the following ideals of the ring $\R[\A_\M]$:
\begin{align*}
I_{\sigma}&= \ker(\R[\A_\M]\xrightarrow{\sigma} \End_\R(\sigma))\\
I_{\sigma}^{\Q\P} &= \ker\left(\R[\A_\M]\to \End_\R((r_\Q^\G\circ i_\P^\G /F_{\Q\P}^{<1})(\sigma))\right)\\
I_{\sigma}^{\Q,w}& = \ker\left(\R[\A_\M]\to \End_\R( \dot{w}(r_{w^{-1}(\Q)\cap \M}^{\M}(\sigma)))\right)\text{\ , for $w\in \W_\M\backslash \W_\G/\W_\M$}\\
J_{\sigma} &= \prod_{w<1}I_{\sigma}^{\Q,w}.
\end{align*}

\begin{lemma}
We have inclusions $$J_{\sigma} \subset I_{\sigma}^{\Q\P} \subset \bigcap_{w<1}I_{\sigma}^{\Q,w}.$$
\end{lemma}
\begin{proof}
The quotient $(r_\Q^\G\circ i_\P^\G /F_{\Q\P}^{<1})(\sigma)$ has a filtration with subquotients isomorphic to $$i_{\M\cap w\P}^\M\circ \dot{w}\circ r_{\M\cap w^{-1}\Q}^\M(\sigma)$$ for $w<1$, and the action of $\R[\A_\M]$ on $i_{\M\cap w\P}^\M\circ \dot{w}\circ r_{\M\cap w^{-1}\Q}^\M(\sigma)$ is via its action on $\dot{w}(r_{\M\cap w^{-1}\Q}^\M(\sigma))$.
\end{proof}

\begin{definition}
An element $b\in \R$ is $(\sigma,\P,\Q)$-singular if it is in the intersection $\left(I_{\sigma} + J_{\sigma}\right) \cap \R$ and it is not a zero divisor.
\end{definition}

\begin{lemma}\label{lem:singular}
Suppose $b$ is $(\sigma,\P,\Q)$-singular.
\begin{enumerate}
\item If $\sigma'$ is a subquotient of $\sigma$, then $b$ is $(\sigma',\P,\Q)$-singular.
\item If $f:\R\to \R'$ is a homomorphism of rings such that $f(b)$ is not a zero divisor, then $f(b)$ is $(\sigma\otimes_{\R,f}\R',\P,\Q)$-singular.
\end{enumerate}
\end{lemma}
\begin{proof}
Let $s$ be an element of $\R[\A_\M]$. If $s$ kills $\sigma$, it must also kill $\sigma'$, so $I_{\sigma}\subset I_{\sigma'}$. Since $s$ acts on an $\R[w(\M)\cap \M]$ module via the morphism $\A_\M\to \A_{w(\M)\cap \M}$, if $s$ kills the $\R[w(\M)\cap \M]$-module $\dot{w}(r_{\M\cap w^{-1}\Q}^\M(\sigma))$ for some $w<1$, then it must also kill $\dot{w}(r_{\M\cap w^{-1}\Q}^\M(\sigma'))$ because $s$ commutes with any morphism of $\R[w(\M)\cap \M]$-modules. This proves that $J_{\sigma}\subset J_{\sigma'}$. Thus $(J_{\sigma}+I_{\sigma})\cap \R\subset (J_{\sigma'}+I_{\sigma'})\cap \R$.

Using a similar argument, the second claim follows from the fact that $I_{\sigma}\otimes_\R\R'\subset I_{\sigma\otimes_\R\R'}$ and $J_{\sigma}\otimes_\R\R'\subset J_{\sigma\otimes_\R\R'}$ after identifying $\R[\A_\M]\otimes_\R\R'\simeq \R'[\A_\M]$. For the latter inclusion, we invoke the fact that $\dot{w}\circ r_{\M\cap w^{-1}\Q}^\M$ commutes with extension of scalars.
\end{proof}

Let us relate this to the terminology that an $\R[\M]$-module $\sigma$ ``is $(\P,\Q)$-regular,'' which appears in \cite[Lemma 2.10]{datnu}, and which is defined to mean $I_{\sigma} + \bigcap_{w<1}I_{\sigma}^{\Q,w} = \R[\A_\M]$ or, equivalently, that $I_{\sigma} + I_{\sigma}^{\Q\P} = \R[\A_\M]$. 

\begin{lemma}
\begin{enumerate}
\item An $\R[\M]$-module $\sigma$ is $(\P,\Q)$-regular if and only if $1$ is $(\sigma,\P,\Q)$-singular.
\item If $b$ is $(\sigma,\P,\Q)$-singular, then $\sigma\otimes_\R\R[1/b]$ is $(\P,\Q)$-regular. 
\item Suppose $b\in \R$ is not a zero divisor and $\sigma$ has bounded $b$-power torsion. If $\sigma\otimes_\R\R[1/b]$ is $(\P,\Q)$-regular, then $b^r$ is $(\sigma,\P,\Q)$-singular for some positive integer $r$.
\end{enumerate}
\end{lemma}
\begin{proof}
If $1$ is $(\sigma,\P,\Q)$-singular we have $$\R[\A_\M] = I_{\sigma}+J_{\sigma} \subset I_{\sigma} + I_{\sigma}^{\Q\P} \subset I_{\sigma} + \bigcap_{w<1}I_{\sigma}^{\Q,w},$$ so all these ideals are $\R[A_M]$, and $\sigma$ is $(\P,\Q)$-regular. Conversely, suppose $\sigma$ is $(\P,\Q)$-regular, i.e., $I_{\sigma} + \bigcap_{w<1}I_{\sigma}^{\Q,w} = \R[\A_\M]$. If we write $1 = i + l$ with $i\in I_{\sigma_{\R}} $ and $l\in \bigcap_{w<1}I_{\sigma_{\R}}^{\Q,w}$, then $$1 = 1^{|\W_\G|-1} = (i+l)^{|\W_\G|-1} \in I_{\sigma_{\R}} +  \prod_{w<1}I_{\sigma_{\R}}^{\Q,w}= I_{\sigma_{\R}}+J_{\sigma_{\R}},$$ so $1$ is $(\sigma,\P,\Q)$-singular.

For (2), suppose $b$ is a non-zero divisor in $(I_{\sigma}+J_{\sigma})\cap \R$. Then if $\R' = \R[1/b]$ and $\sigma_{\R'} = \sigma\otimes_\R\R'$, we have $I_{\sigma}\otimes \R' \subset I_{\sigma_{\R'}}$ and $J_{\sigma}\otimes \R'\subset J_{\sigma_{\R'}}$ by arguing as in the proof of Lemma~\ref{lem:singular}. Hence $I_{\sigma_{\R'}}+J_{\sigma_{\R'}}$ contains a unit and is therefore equal to $\R'[\A_\M]$.

For (3), suppose conversely that $I_{\sigma_{\R'}} + J_{\sigma_{\R'}} = \R'[\A_\M]$. To show that some power of $b$ is $(\sigma,\P,\Q)$-singular, let us first restrict to the case that $\sigma$ is $b$-torsion free. Then each of the modules $\dot{w}(r_{\M\cap w^{-1}\Q}^\M(\sigma))$ is also $b$-torsion free, and it follows that $I_{\sigma_{\R'}} = I_{\sigma}\otimes \R'$ and $J_{\sigma_{\R'}}=J_{\sigma}\otimes \R'$. Hence there is an $r$ such that $b^r$ is in $I_{\sigma}+ J_{\sigma}$. 

In general, the torsion submodule $\sigma^{\mathrm{tor}}$ is $b^s$-torsion for some fixed integer $s$. If $\sigma^{\mathrm{tf}}$ denotes the $b$-torsion free quotient of $\sigma$, we have $\sigma_{\R'} = \sigma^{\mathrm{tf}}\otimes_\R\R'$, so the previous paragraph implies there exists an $r$ such that $b^r$ is in $I_{\sigma^{\mathrm{tf}}} + J_{\sigma^{\mathrm{tf}}}$. It follows that $b^{s+r}$ is in $I_{\sigma}+J_{\sigma}$.
\end{proof}

Suppose $b$ is an element of $\R$ that is $(\sigma,\P,\Q)$-singular and choose any decomposition
$$b = j_{\sigma} + i_{\sigma}$$ into elements $j_{\sigma}\in J_{\sigma}$ and $i_{\sigma}\in I_{\sigma}$. Since $j_{\sigma}$ is in $I_{\sigma}^{\Q\P}$, it defines a morphism
$$j_{\sigma}:r_\Q^\G\circ i_\P^\G  (\sigma) \to F_{\Q\P}^{<1}(\sigma).$$ We compose it with the map $q$ above to get an $M$-equivariant morphism $r_\Q^\G\circ i_\P^\G  (\sigma)\to \sigma$. Passing through the isomorphism of Frobenius reciprocity, i.e., $\Hom_{\R[\M]}(r_\Q^\G\circ i_\P^\G  (\sigma), \sigma)\simeq \Hom_{\R[\G]}(i_\P^\G (\sigma),i_{\Q}^{\G}(\sigma))$, we obtain a morphism $i_\P^\G (\sigma)\to i_{\Q}^{\G}(\sigma)$. Finally, we extend scalars to $\R[1/b]$ and divide this morphism by $b$ to obtain $J_{\Q|\P}(\sigma)$, the intertwining operator. More compactly, we can write this

\begin{align*}
J_{\Q|\P}(\sigma): i_\P^\G (\sigma)[1/b] &\to i_{\Q}^{\G}(\sigma)[1/b]\\
f&\mapsto J_{\Q|\P}(\sigma)(f),
\end{align*}
where
$$J_{\Q|\P}(\sigma)(f)(g):= \frac{1}{b}\int_{\U_\Q\cap \U_{\overline{\P}}}j_{\sigma}(\overline{gf})(u)\ du,$$
where $\overline{gf}$ denotes the image of $gf$ under the map $i_\P^\G (\sigma)\to r_\Q^\G\circ i_\P^\G (\sigma)$. While $j_{\sigma}(\overline{gf})$ is an element of $F_{\Q\P}^{<1}$ and thus not, strictly speaking, a function on $G$, we can lift it to $\widetilde{F}_{\Q\P}^{<1}\subset i_\P^\G (\sigma)$ and take the integral $\int j_{\sigma}(\overline{gf})(u)du$, which factors through parabolic restriction.

If $f$ is in $\widetilde{F}_{\Q\P}^{<1}(\sigma)$, then $j_{\sigma}\overline{f} = b \overline{f} - i_{\sigma}\overline{f} = b\overline{f}$, so the expression simplifies to
$$J_{\Q|\P}(\sigma)(f)(1) = \int_{\U_\Q\cap \U_{\overline{\P}}}f(u)du.$$ 
\begin{lemma}[\cite{datnu} Lemma 7.12]\label{lem:characterize}
The intertwining operator $J_{\Q|\P}(\sigma)$ is the unique element of $\Hom_{\R[1/b][\G]}(i_\P^\G (\sigma)[1/b],i_{\Q}^{\G}(\sigma)[1/b])$ having the form $J_{\Q|\P}(\sigma)(f)(1) = \int_{\U_\Q\cap \U_{\overline{\P}}}f(u)du$ on $f$ in $\widetilde{F}_{\Q\P}^{<1}(\sigma)$.
\end{lemma}

In particular, $J_{\Q|\P}(\sigma)$ does not depend on our choice of decomposition of $b$ into $j_{\sigma}+i_{\sigma}$. 

\begin{remark}\label{rmk:smaller_b}
In fact, $J_{\Q|\P}(\sigma)$ only depends on $b$ in that we require the extension of scalars to $\R[1/b]$ to define it. Given any $b'$ dividing $b$ in $\R$ such that $b'$ is also $(\sigma,\P,\Q)$-singular, the same operator $J_{\Q|\P}(\sigma)$ descends to $\R[1/b']$.
\end{remark}

\subsection{Properties of intertwining operators}

\begin{proposition}\label{prop:functoriality}
Suppose $b$ is $(\sigma,\P,\Q)$-singular, and let $f:\R\to \R'$ be a homomorphism of $\mathbb{Z}[\sqrt{q}^{-1}]$-algebras such that $f(b)$ is not a zero divisor and let $\sigma_{\R'}$ denote $\sigma\otimes_\R\R'$. Suppose there is a morphism of $\R[\M]$-modules $q:\sigma\to \sigma'$ and $b'$ is $(\sigma',\P,\Q)$-singular. Then the following diagrams commute, respectively,

$$
\begin{tikzcd}
i_\P^\G (\sigma)[1/b] \arrow[r,"J_{\Q|\P}(\sigma)"] \arrow[d,"\text{id}\otimes 1"]& i_{\Q}^{\G}(\sigma)[1/b] \arrow[d,"\text{id}\otimes 1"]\\
i_\P^\G (\sigma_{\R'})[1/f(b)] \arrow[r,"J_{\Q|\P}(\sigma_{\R'})"] & i_{\Q}^{\G}(\sigma_{\R'})[1/f(b)]
\end{tikzcd}
\ \ \
\begin{tikzcd}
i_\P^\G (\sigma)[1/(bb')] \arrow[r,"J_{\Q|\P}(\sigma)"] \arrow[d,rightarrow]& i_{\Q}^{\G}(\sigma)[1/(bb')] \arrow[d,rightarrow]\\
i_\P^\G (\sigma')[1/(bb')] \arrow[r,"J_{\Q|\P}(\sigma')"] & i_{\Q}^{\G}(\sigma')[1/(bb')]
\end{tikzcd}
$$
\end{proposition}
\begin{proof}
By Lemma~\ref{lem:characterize}, each of the horizontal arrows on the top row has, and is uniquely characterized by, the property that for any $h\in \widetilde{F}_{\Q\P}^{<1}(\sigma)$, $J_{\Q|\P}(\sigma)(h)(1)$ is given by $\int_{\U_\Q\cap \U_{\overline{\P}}}h(u)du$. The same goes for the horizontal arrows on the bottom row with the inputs replaced with $(\sigma_{\R'},\R'[1/f(b)])$ and $(\sigma', \R[1/(bb')])$, respectively. But the integral $\int_{\U_\Q\cap \U_{\overline{\P}}}h(u)du$ is a finite sum and commutes with extension of scalars along $f:\R\to \R'$, which proves the commutativity of the first diagram. Similarly, $$q\left(\int_{\U_\Q\cap \U_{\overline{\P}}}h(u)du\right) = \int_{\U_\Q\cap \U_{\overline{\P}}}(q\circ h)(u)du,$$ which completes the proof. 
\end{proof}

In practice, given a homomorphism $f:\R\to \R'$ and a $(\sigma,\P,\Q)$-singular element $b$ of $\R$, it may be hard to tell whether $f(b)$ is not a zero divisor. The following lemma guarantees at least one such $b$ under nice circumstances, given the existence of a $(\sigma\otimes_\R\R',\P,\Q)$-singular element of $\R'$.
\begin{lemma}\label{hard_lemma}
Suppose $f:\R\to \R'$ is a homomorphism of noetherian integral domain $\mathbb{Z}[\sqrt{q}^{-1}]$-algebras with kernel $\CP$. Suppose $\sigma$ is an admissible and finitely generated $\R[\M]$-module that is $\R$-torsion free, and suppose there exists $b'\in \R'$ that is $(\sigma\otimes_\R\R',\P,\Q)$-singular. 
\begin{enumerate}
\item There is a nonzero $s$ in $\R'$ such that $sb'$ is contained in the image of $f$ (and any such $sb'$ is necessarily $(\sigma\otimes_\R\R',\P,\Q)$-singular).
\item The identity element (and hence any nonzero element) in the localization $\R_{\CP}$ is $(\sigma\otimes_{\R}\R_{\CP},\P,\Q)$-singular.
\item The set of $(\sigma,\P,\Q)$-singular elements of $\R$ is not contained in $\CP$.
\end{enumerate}
\end{lemma}
\begin{proof}
For any $\R$-algebra $\widetilde{\R}$ we define $\sigma_{\widetilde{\R}} = \sigma\otimes_\R\widetilde{\R}$ and
\begin{align*}
S_{\widetilde{\R}} &:= \im\left(\widetilde{\R}[\A_\M]\to \End_{\widetilde{\R}}(\sigma_{\widetilde{\R}}) \right)\\
    T^w_{\widetilde{\R}} &:= \im\left(\widetilde{\R}[\A_\M]\to \End_{\widetilde{\R}}(w(r^\M_{w^{-1}(\Q)\cap \M}(\sigma_{\widetilde{\R}}))\right),
\end{align*}
for $w\in \W_\M\backslash \W_\G/\W_\M$.

Note that $S_\R$ is a finitely generated $\R$-module because of our assumption that $\sigma$ is admissible and $\R[\M]$-finite. Since parabolic restriction preserves admissibility \cite[Cor 1.5]{DHKMfiniteness}, $T_\R^w$ is also finitely generated as an $\R$-module.

It follows from the definitions that the natural map $S_\R\otimes_\R\widetilde{\R}\to S_{\widetilde{\R}}$ is surjective, and likewise for $T_\R^w$. When $\widetilde{\R}$ is a flat $\R$-algebra, this map is an isomorphism \cite[proof of Lemma 7.2]{datnu}) and restricts to an isomorphism of ideals $I_{\sigma}\otimes_\R\widetilde{\R}\simeq I_{\sigma_{\widetilde{\R}}}$, and likewise for $J_{\sigma}$.

Now, following the second step of the proof of Lemma 7.2 in \cite{datnu}, the kernels of the canonical morphisms
$$S_\R\otimes _{\R,f} \Frac(\R')\to S_{ \Frac(\R')},\ \ \ \ \ T^w_\R\otimes_{\R,f} \Frac(\R')\to T^w_{ \Frac(\R')}$$ are nilpotent ideals. 

The homomorphism $f$ extends to a homomorphism $\R_{\CP} \to \R'$, and the image of $f$ is isomorphic to $\R_{\CP}/\CP \R_{\CP}$. We will identify $\R_{\CP}/\CP \R_{\CP}$ with a subring of $\R'$. 

Consider the ideal $I_{\sigma_{\R_{\CP}}}\subset \R_{\CP}[\A_\M]$. Its reduction mod $\CP \R_{\CP}$ must contain a power of $I_{\sigma_{\Frac{R'}}}$, and likewise for the reduction mod $\CP \R_{\CP}$ of $J_{\sigma_{\R_{\CP}}}$. By assumption, $b'$ is in $I_{\sigma_{\R'}}+J_{\sigma_{\R'}} \subset I_{\sigma_{\Frac(R')}}+J_{\sigma_{\Frac(\R')}}$. It follows that there is a nonzero element $s$ in $\R'$ such that $$sb' \in \left(I_{\sigma_{\R_{\CP}}}\otimes (\R_{\CP}/\CP \R_{\CP}) \right)+ \left(J_{\sigma_{\R_{\CP}}}\otimes (\R_{\CP}/\CP \R_{\CP})\right)\subset (\R_{\CP}/\CP \R_{\CP})[\A_\M].$$ This means $sb'$ is in both $\R'$ and $(\R_{\CP}/\CP \R_{\CP})[\A_\M]$, and therefore lies in $(\R_{\CP}/\CP \R_{\CP})$, which is the image of $f$. Since the $(\sigma\otimes_\R\R',\P,\Q)$-singular elements form an ideal, this proves (1).

Now choose a lift $\tilde{b}$ of $sb'$ to $\R_{\CP}$. Note that $\tilde{b}$ is a unit because it does not lie in $\CP \R_{\CP}$ because $sb'$ is nonzero. Since $\tilde{b}$ lifts an element of $\left(I_{\sigma_{\R_{\CP}}}+J_{\sigma_{\R_{\CP}}}\right)\otimes (\R_{\CP}/\CP \R_{\CP})$, we have
$$\tilde{b}\in I_{\sigma_{\R_{\CP}}} + J_{\sigma_{\R_{\CP}}} + \CP \R_{\CP}[\A_\M] = \R_{\CP}[\A_\M],$$ where the last equality comes from the fact that $\tilde{b}$ is a unit in $\R_{\CP}[\A_\M]$. Now using the finiteness established at the beginning of the proof along with Nakayama's lemma, we conclude that $$I_{\sigma_{\R_{\CP}}}+J_{\sigma_{\R_{\CP}}} = \R_{\CP}[\A_\M].$$ This proves (2).

Since $1\in \R_{\CP}$ lies in $I_{\sigma_{\R_{\CP}}}+J_{\sigma_{\R_{\CP}}}$, it follows that there is an element $b\in \R$, which is not in $\CP$, such that $b$ is in $I_{\sigma}+J_{\sigma}$. This proves (3).
\end{proof}

The next property we will need is the compatibility of the notion of $(\sigma,\P,\Q)$-singularity with respect to changing $\Q$. Given semistandard parabolics $\P$ and $\Q$, we define $d(\P,\Q) = |\Sigma_{red}(\P)\cap \Sigma_{red}(\overline{\Q})|$, where $\Sigma_{red}(\P)$ denotes the set of reduced roots of $\A_\M$ in $\P$. 
\begin{lemma}
\label{lem:rank_one_subgroups}
Let $\M$ be a standard Levi, and let $\O,\P,\Q$ be three parabolics with Levi component $\M$ such that 
$$d(\O,\Q) = d(\O,\P) + d(\P,\Q)\text{,\ \ \ and\ \ \  $d(\O,\P)=1$}.$$ If $b\in \R$ is $(\sigma,\O,\Q)$-singular, then it is also $(\sigma,\O,\P)$-singular and $(\sigma,\P,\Q)$-singular and $$J_{\Q|\O}(\sigma[1/b]) =J_{\Q|\P}(\sigma[1/b])\circ J_{\P|\O}(\sigma[1/b]).$$
\end{lemma}
\begin{proof}
The singularity claims follow from the calculations in \cite[Prop 7.8(i)]{datnu}, which give $J_{\sigma}^{\P\Q} = J_{\sigma}^{\O\Q}$ and $J_{\sigma}^{\O\P} = J_{\sigma}^{\O\Q}$. The statement on intertwining operators exactly follows the proof of \cite[Prop 7.8(i)]{datnu}.
\end{proof}

\begin{lemma}[Compatibility with induction]
\label{lem:compatibility_with_induction}
Let $\P$, $\Q$ be two parabolics with Levi component $\M$.
\begin{enumerate}
\item  Suppose $\P$ and $\Q$ are contained in a parabolic subgroup $\O$ with Levi $\N$. If $b$ is $(\sigma,\P,\Q)$-singular, then it is $(\sigma, \P\cap \N, \Q\cap \N)$-singular and
$$J^\G_{\Q|\P}(\sigma) = i_\O^\G(J^\N_{\Q\cap \N|\P\cap \N}).$$
\item Let $\N$ be a Levi subgroup of $\G$ containing $\M$ such that $\P\cap \N = \Q\cap \N$ and such that $\P\N$ and $\Q\N$ are parabolic subgroups with Levi component $\N$. If $b$ is $(i_{\P\cap \N}^\N(\sigma),\P\N,\Q\N)$-singular, then it is also $(\sigma, \P, \Q)$-singular and
$$J_{\Q|\P}(\sigma) = J_{\Q\N|\P\N}(i_{\P\cap \N}^\N(\sigma)).$$
\end{enumerate}
\end{lemma}
\begin{proof}
The singularity statements follow from the calculations in the proof of \cite[Prop 7.8(ii) and (iii)]{datnu}, and the proofs of statements on intertwining operators closely resemble the arguments in \cite[Prop 7.8(ii) and (iii)]{datnu}.
\end{proof}

\subsection{Intertwining operators for the universal unramified twist}

Now we specialize to the situation where $\R = k[\M/\M^\circ]$ for $k$ an algebraically closed field and we choose a square root of $q$ in $k$. Let $\sigma_0$ be a finite length $k[\M]$-module and let $\sigma$ be the $\R[\M]$-module $\sigma_0\chi_{\univ, \M, k}$. In this setting, we can find elements of $\R$ that are $(\sigma,\P,\Q)$-singular by following the method of \cite[Thm IV.1.1]{Waldspurger}.

For $w\in \W_\M\backslash \W_\G/\W_\M$, let $\mathcal{E}_w$ be the set of characters $\nu:\A_\M\to \R^{\times}$ that have the form $$\nu(a) = \nu_0(a)\overline{w^{-1}aw},$$ where $\overline{a}$ denotes the image of $a$ under $\A_\M\to \M\to \M/\M^{\circ}$, and where $\nu_0$ occurs as the pullback under $\A_\M\to \A_{w(\M)\cap \M}$ of central characters of irreducible subquotients of the module $\dot{w}(r_{\M\cap w^{-1}\Q}^\M(\sigma_0))$ (this module has finite length because parabolic restriction preserves finite length). For any $\nu\in \mathcal{E}_w$, let $d(\nu,w)$ designate an integer such that $$\prod_{\nu\in\mathcal{E}_w}(a - \nu(a))^{d(\nu,w)}$$ annihilates $\dot{w}(r_{\M\cap w^{-1}\Q}^\M(\sigma))$; we remark that such an integer exists because it exists after extending scalars to $\text{Frac}(\R) = k(\M/\M^{\circ})$.

\begin{lemma}\label{lem:distinctcentchar}
Let $\mu$ be an element of $\mathcal{E}_1$ and let $\nu$ be an element of $\mathcal{E}_w$ for $w\neq 1$. Then $\mu\neq \nu$.
\end{lemma}
\begin{proof}
This is proven in Equation (6) in \cite{Waldspurger} in the course of proving Thm IV.1.1 of \cite{Waldspurger}. The same argument works for our more general $k$.
\end{proof}

Given $\mu\in \mathcal{E}_1$, we will need to consider the set $\mathcal{E}^{\mu}$ of all such $\nu$'s distinct from $\mu$, i.e. the set
$$\mathcal{E}^{\mu}: = \left(\bigsqcup_{w\neq 1}\mathcal{E}_w\right) \sqcup \left(\mathcal{E}_1 - \{\mu\}\right).$$ Note that while the $\mathcal{E}_w$, $w\neq 1$ may not be disjoint from one another, we define $\mathcal{E}^{\mu}$ to be their disjoint union.

\begin{lemma}\label{lem:constructb}
For every pair of distinct characters $\mu,\nu:\A_\M\to k[\M/\M^{\circ}]^{\times}$, fix elements $a_{\mu,\nu}\in \A_\M$ such that $\mu(a_{\mu,\nu})\neq\nu(a_{\mu,\nu})$. The element 
$$b(\sigma_0,\P,\Q) = \prod_{\mu\in \mathcal{E}_1}\prod_{\nu\in \mathcal{E}^{\mu}}\left(\mu(a_{\mu,\nu}) - \nu(a_{\mu,\nu})\right)^{d(\mu,w)+d(\nu,w)-1}$$ is $(\sigma,\P,\Q)$-singular.
\end{lemma}
\begin{proof}
Each factor $\left(\mu(a_{\mu,\nu}) - \nu(a_{\mu,\nu})\right)^{d(\mu,w)+d(\nu,w)-1}$ is the resultant of the two polynomials $(X-\mu(a_{\mu,\nu}))^{d(\mu,1)}$ and $(X-\nu(a_{\mu,\nu}))^{d(\nu,w)}$ in the polynomial ring $\R[X]$. Thus there are polynomials $F(X)$ and $G(X)$ such that $$(X-\mu(a_{\mu,\nu}))^{d(\mu,1)}F(X) + (X-\nu(a_{\mu,\nu}))^{d(\nu,w)}G(X)  = \left(\mu(a_{\mu,\nu}) - \nu(a_{\mu,\nu})\right)^{d(\mu,w)+d(\nu,w)-1}.$$ Now set $X=a_{\mu,\nu}$ to get an equality in $\R[\A_\M]$. Since we have
\begin{align*}
\prod_{\mu\in\mathcal{E}_1}(a_{\mu,\nu}-\mu(a_{\mu,\nu}))^{d(\mu,1)}F(a_{\mu,\nu}) &\in I_{\sigma} \\
\prod_{\nu\in\mathcal{E}_{w}}(a_{\mu,\nu}-\nu(a_{\mu,\nu})^{d(\nu,w)}G(a_{\mu,\nu}) &\in I_{\sigma}^{\Q,w}.
\end{align*}
it follows that the product $b(\sigma_0,\P,\Q)$ is an element of $I_{\sigma} + J_{\sigma}$. Since each factor $\left(\mu(a_{\mu,\nu}) - \nu(a_{\mu,\nu})\right)^{d(\mu,w)+d(\nu,w)-1}$ is contained in $\R\subset \R[\A_\M]$, we have shown that $b(\sigma_0,\P,\Q)$ is in $(I_{\sigma}+J_{\sigma})\cap \R$, and thus is $(\sigma,\P,\Q)$-singular.
\end{proof}

\section{Harish--Chandra $j$-functions and Plancherel measure}\label{apdx:j_functions}

In this appendix we will use the theory of intertwining operators to deduce more general versions of properties of $j$-functions and Plancherel measures appearing in the literature. We now take $\Q = \overline{\P}$ the opposite parabolic, and we still take $\R$ to be $k[\M/\M^{\circ}]$. Given $b_1\in \R$ that is $(\sigma,\P,\overline{\P})$-singular and $b_2$ that is $(\sigma, \overline{\P},\P)$-singular, the product $b=b_1b_2$ is both $(\sigma,\P,\overline{\P})$- and $(\sigma, \overline{\P},\P)$-singular. We set $$b = b_P^G(\sigma_0)=b(\sigma_0,\P,\overline{\P})b(\sigma_0,\overline{\P},\P),$$ with the notation from Lemma~\ref{lem:constructb} and we define $j_\P^\G(\sigma_0)\in \End_{\R[1/b]}(i_\P^\G (\sigma)[1/b])$ to be the composition of intertwining operators
$$i_\P^\G (\sigma)[1/b] \xrightarrow{J_{\overline{\P}|\P}(\sigma)} i_{\overline{\P}}^\G(\sigma)[1/b]\xrightarrow{J_{\P|\overline{\P}}(\sigma)}i_\P^\G (\sigma)[1/b].$$

\begin{lemma}
Let $\sigma_0$ be a finite length $k[\M]$-module such that $\cK\to \End_G(i_\P^\G (\sigma_{\cK}))$ is an isomorphism (e.g. $\sigma_0$ irreducible). Then the endomorphism $j_\P^\G(\sigma_0)$ in $\End_{\R[1/b][\G]}(i_\P^\G (\sigma)[1/b])$ is a nonzero scalar in $\R[1/b]$.
\end{lemma}
\begin{proof}
After extending scalars from $\R[1/b]$ to an algebraic closure $\cK=\overline{\text{Frac}(\R)}$ of $\text{Frac}(\R)$ we obtain, by compatibililty of $J_{\overline{\P}|\P}$ and $J_{\P|\overline{\P}}$ with extension of scalars, intertwining operators
\begin{align*}
J_{\overline{\P}|\P}(\sigma_\cK)&\in \Hom_{\cK[\G]}(i_\P^\G (\sigma)_{\cK},i_{\overline{\P}}^\G(\sigma)_{\cK})\\
J_{\P|\overline{\P}}(\sigma_\cK)&\in \Hom_{\cK[\G]}(i_{\overline{\P}}^\G(\sigma)_{\cK},i_\P^\G (\sigma)_{\cK})\ .
\end{align*}
Since both $i_\P^\G (\sigma_{\cK})$ and $i_{\overline{\P}}^\G(\sigma_{\cK})$ are irreducible by the generic irreducibility theorem, and both intertwining operators $J_{\overline{\P}|\P}(\sigma_\cK)$ and $J_{\P|\overline{\P}}(\sigma_{\cK})$ are nonzero, the intertwiners are invertible. Thus their composition is multiplication by a nonzero scalar in $\cK$. On the other hand, the natural map
$$\End_{\R[1/b][\G]}\left(i_\P^\G (\sigma)\left[1/b\right]\right) \to \End_{\cK[\G]}(i_\P^\G (\sigma)_\cK)\simeq \cK$$ is injective by torsion-freeness, so $j_\P^\G(\sigma_0)$ is also nonzero.

It remains to prove that $j^\G_\P(\sigma_0)$ lies not just in $\cK$ but in the ring $\R[1/b]$. But we have inclusions
$$\R\hookrightarrow \R\left[1/b\right] \hookrightarrow \End_{\R[1/b][\G]}\left(i_\P^\G (\sigma)\left[1/b\right]\right) \hookrightarrow \End_{\cK[\G]}(i_\P^\G (\sigma)_\cK) \simeq \cK.$$ Since $\sigma_0$ is admissible, $\End_{\R[\G]}(i_\P^\G (\sigma))$ is finitely generated as an $\R$-module and hence $\End_{\R[1/b][\G]}(i_\P^\G (\sigma)\left[1/b\right])$ is finitely generated as an $\R[1/b]$-module and $j_\P^\G(\sigma_0)$ is an integral element, i.e. satisfies a monic polynomial with coefficients in $\R[1/b]$. But $\R=k[\M/\M^{\circ}]$ is a Noetherian normal integral domain, and so is $\R[1/b]$, and thus $\R[1/b]$ is integrally closed in its fraction field $\cK$. We conclude that $j_\P^\G(\sigma_0)$ is in $\R[1/b]$.
\end{proof}

When $\sigma_0$ is irreducible we can extend scalars from $\R[1/b]$ to an algebraic closure $\cK=\overline{\text{Frac}(\R)}$ of $\text{Frac}(\R)$ and argue as in \cite[IV.3(1)-(3)]{Waldspurger} to establish the following properties for the element $j_\P^\G(\sigma_0)$ of $\cK$:
\begin{itemize}
\item $j_\P^\G(\sigma_0)$ does not depend on $\P$,
\item For $w\in \W^\G$, let ${^w\sigma_0}$ be the representation of $w\M$ obtained by composing the action of $\sigma_0$ with the action of $w$. The image of $j_{w\P}^\G({^w\sigma_0})$ in the composite $$k[w\M/(w\M)^{\circ}][1/{^wb}] \to k[\M/\M^{\circ}][1/b] \to \cK$$ is equal to $j_\P^\G(\sigma_0)$.
\item $j_\P^\G(\sigma_0^{\vee}) = j_\P^\G(\sigma_0)$. 
\end{itemize}
Thus while the ring $\R[1/b]$ depends on $b$, which depends on $\P$ and the choice of $\sigma_0$ within its orbit under $\W^\G$, the element $j_\P^\G(\sigma_0)$ of $\Frac(\R)$ is independent of $b$ and $\P$ (c.f. Remark~\ref{rmk:smaller_b}).
\begin{corollary}\label{cor:subquotients}
If $\sigma_0$ is a finite length $k[\M]$-module with a composition series $$0=\sigma_r\subset \sigma_{r-1}\subset \cdots \subset \sigma_0$$ with irreducible subquotients $\sigma_i/\sigma_{i+1} =\tau_i$, then for each $i$,
\begin{enumerate}
\item $b$ is $(\tau_i\chi_{\univ, M, k},\P,\overline{\P})$- and $(\tau_i\chi_{\univ, M, k},\overline{\P},\P)$-singular, 
\item $j_\P^\G(\sigma_0)$ stabilizes the submodule $i_\P^\G (\sigma_i\chi_{\univ, M, k})\left[1/b\right]$ of $i_\P^\G (\sigma)\left[1/b\right]$
\item $j_\P^\G(\sigma_0)|_{i_\P^\G (\sigma_i\chi_{\univ, M, k})[1/b]} = j_\P^\G(\sigma_i)$,
\item the endomorphism induced by $j_\P^\G(\sigma_0)$ on $i_\P^\G (\tau_i\chi_{\univ, M, k})\left[1/b\right]$ is the scalar $j_\P^\G(\tau_i)\in \R[1/b]$.
\item $j_\P^\G(\sigma_0)$ is nonzero and independent of the choice of $\P$ and $b$.
\end{enumerate}
\end{corollary}
\begin{proof}
(1) follows immediately from Lemma~\ref{lem:singular} and (2) through (4) follow from the functoriality properties in Proposition~\ref{prop:functoriality}. Item (5) follows from the fact that the same is true for each $j_\P^\G(\tau_i)$.
\end{proof}
Therefore, the invariance properties of $j_\P^\G(\sigma_0)$ for irreducible $\sigma_0$ remain true for finite length $\sigma_0$. We will henceforth drop the subscript $\P$ and denote $j_\P^\G(\sigma_0)$ simply by $j^\G(\sigma_0)$.

\subsection{Compatibility with isomorphisms}\label{section:compatibility_with_isom}

Suppose $\phi:\G\to \G'$ is an isomorphism of connected reductive $\F$-groups. Given a Noetherian $\mathbb{Z}[\sqrt{q}^{-1}]$-algebra $\R$ and a smooth $\R[\G]$-module $\pi$, let $\pi'$ denote the $\R[\G']$-module with $\G'$-action given by $\pi \circ \phi^{-1}$. If $\M\subset \G$ is a standard Levi subgroup, $\M':=\phi(\M)$ is a standard Levi of $\G'$.

Precomposition of functions with $\phi^{-1}$ gives an isomorphism of induced modules $$\lambda: i_\P^\G (\sigma) \overset{\sim}{\to} i_{\P'}^{\G'}(\sigma').$$ The isomorphism $\A_\M\simeq \A_{\M'}$ allows us to extend $\phi$ to an isomorphism $\R[\A_\M]\to \R[\A_{\M'}]$, which identifies $I_{\sigma}$ and $J_{\sigma}$ with $I_{\sigma'}$ and $J_{\sigma'}$ in $R[\A_{\M'}]$, so if $b$ is $(\sigma,\P,\Q)$-singular it is also $(\sigma',\P', \Q')$-singular. It follows from Lemma~\ref{lem:characterize} that the following diagram commutes:
$$
\begin{tikzcd}
i_\P^\G (\sigma)[1/b] \arrow[r,"J_{\Q|\P}(\sigma)"] \arrow[d,"\lambda"]& i_{\Q}^{\G}(\sigma)[1/b] \arrow[d,"\lambda"]\\
i_{\P'}^{\G'}(\sigma')[1/b] \arrow[r,"J_{\Q'|\P'}(\sigma')"] & i_{Q'}^{G'}(\sigma')[1/b]
\end{tikzcd}\ \ .
$$

If $k$ is a field containing $\sqrt{q}^{-1}$, and $R=k[\M/\M^{\circ}]$ the isomorphism $\M\simeq \M'$ induces an isomorphism $\tilde\phi$ of $\R$ with $\R'= k[\M'/(\M')^{\circ}]$. Note that $\tilde\phi$ and $\lambda$ induce an isomorphism $$\End_{\R[1/b][\G]}(i_\P^\G (\sigma)[1/b])\otimes_{\tilde\phi} \R' \simeq \End_{\R'[1/b'][\G']}(i_{\P'}^{\G'}(\sigma')[1/b']).$$ 
\begin{lemma}\label{lem:compatibility_isom}
Let $\sigma_0$ be a finite length $k[\M]$-module and choose $b\in k[\M/\M^{\circ}]$ that is $(\sigma,\P,\overline{\P})$-singular for some parabolic $\P$. Then $b':=\tilde\phi(b)\in \R'$ is $(\sigma',\P',\overline{\P'})$-singular and
$j^\G(\sigma_0)$ corresponds to $j^{G'}(\sigma_0')$ under the identification of $\End_{\R[1/b][\G]}(i_\P^\G (\sigma)[1/b])\otimes_{\tilde\phi} \R'$ with $\End_{\R'[1/b'][\G']}(i_{\P'}^{\G'}(\sigma')[1/b'])$. 
\end{lemma}
Note that if $\sigma_0$ is irreducible, the statement is simply $\tilde\phi(j^\G(\sigma_0)) = j^{\G'}(\sigma_0')$, where we have also used $\tilde\phi$ to denote the extension of $\tilde\phi$ to $\R[1/b]\simeq \R'[1/b']$.
\begin{proof}
The map $\tilde\phi$ induces isomorphism $\R[\A_\M]\to \R'[\A_{\M'}]$, which identifies $I_{\sigma}$ and $J_{\sigma}$ with $I_{\sigma'}$ and $J_{\sigma'}$ in $\R'[\A_{\M'}]$, and this proves that $\tilde\phi(b)$ is singular. To simplify notation, we can then identify $\R = \R'$ and think of $\chi_{\univ,\M,k}$ and $\chi_{\univ,\M',k}$ as both valued in the ring $\R$, via the isomorphism $\phi:\M\simeq \M'$. The lemma then follows from the diagram in the previous paragraph.
\end{proof}

If we had chosen a different maximal split torus $A_0'$ we would have an isomorphism $G\to G$ taking parabolics that are semistandard for $A_0$ to those semistandard for $A_0'$. Thus we have shown that the requirement that $P$ be semistandard can be relaxed up to the equivalence described in Lemma~\ref{lem:compatibility_isom}.

A similar functoriality argument shows that the $j$ function is compatible with isomorphisms of $\sigma_0$. More precisely, any isomorphism $\sigma \overset{\sim}{\to} \sigma'$ of $\R[\M]$-modules induces an isomorphism $\lambda:i_\P^\G (\sigma)\overset{\sim}{\to}i_\P^\G (\sigma')$, and Proposition~\ref{prop:functoriality} implies that, given a $(\sigma,\P,\Q)$-singular element $b$, the following diagram commutes:
$$
\begin{tikzcd}
i_\P^\G (\sigma)[1/b] \arrow[r,"J_{\Q|\P}(\sigma)"] \arrow[d,"\lambda"]& i_{\Q}^{\G}(\sigma)[1/b] \arrow[d,"\lambda"]\\
i_{\P}^{\G}(\sigma')[1/b] \arrow[r,"J_{\Q|\P}(\sigma')"] & i_{\Q}^{\G}(\sigma')[1/b]
\end{tikzcd}\ \ .
$$
It follows that, given an isomorphism $\lambda_0:\sigma_0 \overset{\sim}{\to} \sigma_0'$ of finite length $k[\M]$-modules, $j^\G(\sigma_0)$ is mapped to $j^\G(\sigma_0')$ under the induced isomorphism
$$\End_{\R[1/b][\G]}(i_\P^\G (\sigma)[1/b])\overset{\sim}{\to} \End_{\R[1/b][\G]}(i_\P^\G (\sigma')[1/b]).$$ Moreover, the endomorphisms $j^\G(\sigma_0)$ and $j^\G(\sigma_0')$ induce equivalent scalar endomorphisms in $\R[1/b]$ after being restricted to irreducible subquotients of $\sigma_0$ and $\sigma_0'$ that correspond under $\lambda_0$.

In particular, $j^\G(\sigma_0)$ is an invariant of the isomorphism class of $\sigma_0$ when $\sigma_0$ is irreducible, or more generally when $i_\P^\G (\sigma_{\mathcal{K}})$ has only scalar endomorphisms.

\subsection{Factorization}

Let $\Sigma_{red}(\P) = \Sigma_{red}(\A_\M,\P)$ be the subset of reduced roots of $\A_\M$ that are positive relative to $\P$. We can fix an ordering $\Sigma_{red}(\P)= \{\alpha_1,\dots,\alpha_r\}$ such that there are sequences of semistandard parabolic subgroups $(\P_0, \dots, \P_r)$ and $(\Q_0,\dots, \Q_r)$ satisfying:
\begin{itemize}
\item $\P_0=\P$ and $\P_r = \overline{\P}$, and each $\P_i$ has Levi component $\M$,
\item $\Sigma_{red}(\P_i)\cap \Sigma_{red}(\overline{\P_{i-1}}) = \alpha_i$,
\item $\Q_i$ has Levi $\M_{\alpha_i}$, where $\M_{\alpha_i}$ is the Levi subgroup containing $\M$ and the root subgroup attached to $\alpha_i$.
\end{itemize}
In this case, we also have $\overline{\P_i\cap \M_{\alpha_i}} = \P_{i-1}\cap \M_{\alpha_i}$, and $d(\P,\overline{\P}) = \sum_id(\P_i,\P_{i+1})$, and $d(\P_i,\P_{i+1})=1$.

Let $\sigma_0$ be a finite length $k[\M]$-module and let $b=b_\P^\G(\sigma) \in k[\M/\M^{\circ}]$ be as above. Then $b$ is both $(\sigma, \P_i\cap \M_{\alpha_i}, \P_{i-1}\cap \M_{\alpha_i})$ and $(\sigma,\P_{i-1}\cap \M_{\alpha_i},\P_i\cap \M_{\alpha_i})$-singular by Lemmas~\ref{lem:rank_one_subgroups} and~\ref{lem:compatibility_with_induction}(1), so we have elements $$j^{\M_{\alpha_i}}(\sigma_0)\in \End_{\M_{\alpha_i}}(i_{\P_i\cap \M_{\alpha_i}}^{\M_{\alpha_i}}(\sigma)[1/b]).$$ 

\begin{proposition}\label{prop:factorization}
We have the following factorization property in the ring $\End_{\R[1/b][\G]}(i_\P^\G (\sigma)[1/b])$:
$$j^\G(\sigma_0) = \prod_{i=1}^ri_{\Q_{\alpha_i}}^\G(j^{\M_{\alpha_i}}(\sigma_0)).$$
\end{proposition}
\begin{proof}
This is a combination of the factorization property for intertwining operators over rank one subgroups in Lemma~\ref{lem:rank_one_subgroups} with the inductivity property of intertwining operators in Lemma~\ref{lem:compatibility_with_induction}(1).
\end{proof}
\begin{remark}\label{rmk:factorization}
\begin{enumerate}
\item When $\sigma_0$ is irreducible (or more generally when $i_\P^\G (\sigma_{\mathcal{K}})$ has only scalar endomorphisms), all the terms in Proposition~\ref{prop:factorization} are simply elements of $\R[1/b]$, so the result can be stated more succinctly as
$$j^\G(\sigma_0) = \prod_{\alpha\in \Sigma_{red}(\P)}j^{\M_{\alpha}}(\sigma_0).$$ 
\item Since we have already established the left side of Proposition~\ref{prop:factorization} is independent of $\P$, Proposition~\ref{prop:factorization} implies the product on the right side is independent of the ordering of the factors.
\item Each term $j^{\M_{\alpha_i}}$ on the right side of Proposition~\ref{prop:factorization} is actually defined over the subring
$$k[(\M\cap \M_{\alpha_i}^\circ)/\M^{\circ})][{1}/{b_i}],$$ where $b_i$ is some $(\sigma, \P_i\cap \M_{\alpha_i},\overline{\P_i\cap \M_{\alpha_i}})$-singular element. If we let $b' = \prod_ib_i$, we have
$$k[(\M\cap \M_{\alpha_i}^\circ)/\M^{\circ})][1/b']\subset \R[1/b'],$$ and it follows from the proposition that $j^\G(\sigma_0)$ is defined over $\R[1/b']$. Note that $(\M\cap \M_{\alpha_i}^\circ)/\M^{\circ}$ is a free abelian group of rank one.
\end{enumerate}
\end{remark}

\subsection{Multiplicativity}

In this section, we prove the multiplicativity property of $j$-functions in our more general framework. Let $\M\subset \N\subset \G$ be Levi subgroups, and let $\P$ be the standard parabolic with Levi $\M$, let $\R=k[\M/\M^{\circ}]$, let $\R' = k[\N/\N^{\circ}]$ and let $f:\R\to \R'$ denote the homomorphism induced by the inclusion $\M\subset \N$.  If $\sigma_0$ is a finite length $k[\M]$-module, $i_{\P\cap \N}^\N(\sigma_0)$ is a finite length $k[\N]$-module. Observe that 
$$i_{\P\cap \N}^\N(\sigma_0)\chi_{\univ, \N, k} \simeq i_{\P\cap \N}^\N(\sigma_0\chi_{\univ,\M,k})\otimes_{\R,f}\R' \simeq i_{\P\cap \N}(\sigma)\otimes_\R\R'.$$ Since $i_{\P\cap \N}^\N(\sigma)$ is an admissible and finitely-generated $\R[\M]$-module, and we know there exists $b'\in k[\N/\N^{\circ}]$ that is  $( i_{\P\cap \N}(\sigma)\otimes_\R\R', \P\N, \overline{\P\N})$-singular (by Lemma~\ref{lem:constructb}), we can apply Lemma~\ref{hard_lemma} to produce an element $b_1\in \R$ that is $(i_{\P\cap \N}^\N(\sigma), \P\N,\overline{\P\N})$-singular and such that $f(b_1)\neq 0$. Similarly, there exists $b_2\in \R$ that is $(i_{\P\cap \N}^\N(\sigma), \overline{\P\N},\P\N)$-singular with $f(b_2)\neq 0$. Then set $b = b_1b_2$, which is both kinds of singular, and observe $f(b)\neq 0$ since $\R'$ is a domain.

Let $\Q$ denote the parabolic with Levi $\M$ containing $\P\cap \N$ and the unipotent radical of $\overline{\P\N}$. Now Lemma~\ref{lem:compatibility_with_induction}(2) tells us $b$ is also $(\sigma, \P,\Q)$ and $(\sigma, \Q,\P)$-singular.

\begin{lemma}\label{lem:specialization}
Let $f$ also denote the canonical map of localizations $\R[1/b]\to \R'[1/f(b)]$. Under the isomorphism
$$\End_{\R'[1/f(b)][\G]}(i_{\P\N}^\G(i_{\P\cap \N}^\N(\sigma_0)\chi_{\univ, \N,k})[1/f(b)])\simeq \End_{\R[1/b][\G]}(i_{\P\N}^\G i_{\P\cap \N}^\N(\sigma)[1/b])\otimes_{\R[1/b],f}\R'[1/f(b)],$$ the following two elements are identified:
$$j^\G(i_{\P\cap \N}^\N(\sigma_0)) =j^\G(\sigma_0)\otimes_{\R[1/b],f}1$$
\end{lemma}
\begin{proof}
This follows from the first commutative diagram of Proposition~\ref{prop:functoriality}.
\end{proof}

\begin{corollary}\label{cor:multiplicativity}
Let $\Sigma_{red}(\A_\M,\P)$ be the set of roots of $\A_\M$ in $\text{Lie}(\P)$. Then we have
$$j^\G(i_{\P\cap \N}^\N(\sigma_0)) = \prod_{\beta}i_{\Q_{\beta}}^G(j^{\M_{\beta}}(\sigma_0))\otimes_{\R[1/b],f}1,$$ where $\beta$ runs over $\Sigma_{red}(\A_\M,\P) \backslash \Sigma_{red}(\A_\M,\P\cap \N)$.
\end{corollary}
\begin{proof}
We can replace $\G$ with $\P\Q$ and work within $\P\Q$. First apply Proposition~\ref{prop:factorization} to $j^\G(\sigma_0)$ and notice that $$\Sigma_{red}(\A_\M,\P)\backslash \Sigma_{red}(\A_\M,\P\cap \N) = \Sigma_{red}(\A_\M,\P)\cap \Sigma_{red}(\A_\M,\overline{\Q}).$$ Then specialize each term in the product and apply Lemma~\ref{lem:specialization}.
\end{proof}

When $j^\G(\sigma_0)$ is scalar (or when $i_\P^\G (\sigma_{\mathcal{K}})$ has scalar endomorphisms), we have a more compact statement:
\begin{corollary}\label{cor:multiplicativity_irreducible}
Let $\sigma_0$ be an irreducible $k[\M]$-module and let $\pi$ be any irreducible subquotient of $i_{\P\cap \N}^\N(\sigma_0)$. Then
$$j^\G(\pi) = \prod_{\beta}f(j^{\M_{\beta}}(\sigma_0)),$$ where $\beta$ runs over  $\Sigma_{red}(\A_\M,\P) \backslash \Sigma_{red}(\A_\M,\P\cap \N)$.
\end{corollary}

\begin{remark}\label{GanIchinodetailsremark}
This property in the special case of ``irreducible subrepresentation'' instead of ``subquotient'' appears in \cite[Prop B.3, B.4]{GanIchino14} for classical groups. While trying to extend it to subquotients, we found a need for inputs such as Lemma~\ref{hard_lemma} and Corollary~\ref{cor:subquotients} to fill in some of the details of their claim 
$$\mu(\tau\otimes \pi) = \mu(\tau\otimes I_{\P'}^\G(\tau'\otimes \pi')) = \left.\left(\frac{\mu(\tau\otimes\nu\otimes\pi')}{\mu(\nu\otimes \pi')}\right)\right|_{\nu=\tau'}$$ (\cite[p.636]{GanIchino14}). 

In choosing the approach we have taken here, i.e., following \cite{datnu}, we gain the result for irreducible subquotients in the process (Corollary~\ref{cor:multiplicativity_irreducible}).  In the next subsection and in the proof of Proposition~\ref{multofplancherel} we translate the notation in Corollary~\ref{cor:multiplicativity_irreducible} to classical groups.
\end{remark}

\subsection{Plancherel measures for classical groups}\label{apdx:plancherel}

Let  $\G$ be a classical group, fix a nontrivial additive character $\psi$ of $\F$, and fix a parabolic $\P = \M\U$. Let $\sigma_0$ be a finite length $k[\M]$-module such that $\cK \to \End_{\cK[\G]}(i_\P^\G (\sigma_0\chi_{\univ,\M,k}\otimes \cK))$ is an isomorphism, where again $\cK = \Frac(k[\M/\M^{\circ}])$, so that $j^\G(\sigma_0)$ defines a nonzero scalar in $\cK$. For example, this is the case when $\sigma_0$ is irreducible. Recall that our construction of intertwining operators $J_{\P|\Q}(\sigma)$ required fixing a choice of Haar measures $\mu$ and $\mu'$ on $\U_\P$ and $\U_\Q$ in the beginning. The $j$-function $j^\G(\sigma_0)\in \cK$ depends on this choice up to a scalar multiple in $\cK$. There is a particular choice of Haar measures $\mu$ and $\mu'$ on $\U$ and $\overline{\U}$ relative to $\psi$ that is well-suited for applications to local Langlands. We will not describe this choice in detail but simply refer to \cite[B.2]{GanIchino14}, and make the same choices here.

The Plancherel measure of $\sigma_0$ is defined to be
$$\mu^\G_{\psi}(\sigma_0):=\mu^\G(\sigma_0):= j^\G(\sigma_0)^{-1}\in \cK.$$ By Section~\ref{section:compatibility_with_isom}, it depends only on the isomorphism class of $\sigma_0$, and is compatible with isomorphisms of groups $\G\simeq \G'$ after identifying $\cK = \text{Frac}(k[\M/\M^{\circ}])$ with $\text{Frac}(k[\M'/(\M')^{\circ}])$, where $\M'$ is the image of $\M$ in $\G'$. By Lemma~\ref{lem:characterize} and Proposition~\ref{prop:functoriality}, we conclude that if $\alpha:k\to k'$ is any homomorphism fixing $\sqrt{q}$, then $\mu^\G_{\alpha\circ \psi}({^\alpha\sigma_0}) = \alpha(\mu^\G_{\psi}(\sigma_0))$, where we have also used $\alpha$ to denote the induced homomorphism $\cK\to \text{Frac}(k'[\M/\M^{\circ}])$.

The choice of $(P,M)$ is equivalent to a splitting of $\E$-vector spaces $\V=\W_1\oplus\cdots\oplus \W_r\oplus \V'\oplus \W_1'\oplus\cdots \oplus \W_r'$, with $\W_i,\W_i'$ totally
isotropic and pairwise orthogonal and $\V'$ orthogonal to $\W_i\oplus \W_i'$. This gives an identification $\M=\GL(\W_1)\times\cdots \times \GL(\W_r)\times\G'$, where $\G'$ is a classical group of the same type as $\G$. 

Henceforth we suppose $\sigma_0$ is irreducible and decompose it as $$\sigma_0 = \tau_1\otimes\cdots \otimes \tau_r\otimes \pi.$$

Recall the following notational convention from Subsection~\ref{subsection:plancherel_body}: the map $(g_{\W_1},\dots ,g_{\W_r},g_{\V'})\mapsto (\text{val}_{\E}(\det(g_{\W_1})),\dots,\text{val}_{\E}(\det(g_{\W_r}))) $ induces an isomorphism
$\M/\M^{\circ}\To\sim \mathbb{Z}\times \cdots \times \mathbb{Z}$, and thus an isomorphism
$\mathbb{C}[\M/\M^\circ] \simeq \mathbb{C}[(q^{-s_1})^{\pm 1},\dots, (q^{-s_r})^{\pm 1}]$ where we regard $q^{-s_i}$ as indeterminates. 

The universal unramified character $|\det|^s:\, \GL(\W)\To{} \CC[(q^{-s})^{\pm1}]$ 
is given by $g_\W\mapsto (q^{-s})^{\text{val}_{\E}(\det(g_\W))}$ and we denote $\tau_s=  \tau|\det|^{s}$. Upon identifying
$\mathbb{C}(\M/\M^\circ)$ with $\mathbb{C}((q^{-s_1})^{\pm 1},\dots, (q^{-s_r})^{\pm 1})$, we write $(\tau_1)_{s_1}\otimes \cdots \otimes (\tau_r)_{s_r}\otimes \pi$ for the universal unramified twist of $\tau_1\otimes\cdots \otimes \tau_r\otimes \pi$. 

We will denote by 
$\mu^{\G}_{\psi}((\tau_1)_{s_1}\otimes \cdots \otimes (\tau_r)_{s_r}\otimes \pi)$  the rational function in~$\mathbb{C}((q^{-s_1})^{\pm 1},\dots, (q^{-s_r})^{\pm 1})$
corresponding to $\mu^{\G}_{\psi}(\tau_1\otimes\cdots \otimes \tau_r\otimes \pi)\in\mathbb{C}(\M/\M^{\circ})$. Often the subscript $\psi$ or the superscript $\G$ are dropped. In the following, for a positive integer~$m$, we sometimes write~$\GL_m$ for~$\GL_m(\E)$. 

We begin by illustrating the factorization property of Proposition~\ref{prop:factorization} in the special case of classical groups.
\begin{corollary}\label{cor:j_factorization_classical}
We have the factorization $$\mu^G(\sigma_0)=\left(\prod_{1\leqslant i< j\leqslant r}\mu^{\GL_{k_i+k_j}}((\tau_i)_{s_i}\otimes((\tau_j)_{s_j})\mu^{\GL_{k_i+k_j}}((\tau_i)_{s_i}\otimes((\tau_j)_{s_j}^c)^{\vee})\right)\cdot \prod_{1\leqslant i \leqslant r}\mu((\tau_i)_{s_i}\otimes \pi).$$
\end{corollary}
\begin{proof}
Without loss of generality (Subsection~\ref{section:compatibility_with_isom}) we can write $P$ as the upper parabolic 
\begin{align*}
\left(\begin{smallmatrix}
\GL_{k_1}(\E) & * &*&*&*&*&*\\
&\ddots&*&*&*&*&*\\
&&\GL_{k_r}(\E)&*&*&*&*\\
&&& \G'&*&*&*\\
&&&&\GL_{k_r}(\E)&*&*\\
&&&&&\ddots&*\\
&&&&&&\GL_{k_1}(\E)
\end{smallmatrix}\right)\ .
\end{align*}
Recall that in the classical group $\G$, the elements in the $\GL_{k_i}(\E)$ factors in the bottom right of $M$ are the images of their counterparts in the top left under $^t((\cdot)^c)^{-1}$).

In Remark~\ref{rmk:factorization}(1), $\alpha$ runs over the roots of $\A_\M$ positive relative to the parabolic $P$. The groups $\M_{\alpha}$ in Remark~\ref{rmk:factorization}(1) come in three shapes:

\begin{align*}
\left(\begin{smallmatrix}
\GL_{k_1}(\E) &  &&&*&&\\
&\ddots&&&&&\\
&&\GL_{k_r}(\E)&&&&*\\
&&& \G'&&&\\
*&&&&\GL_{k_r}(\E)&&\\
&&&&&\ddots&\\
&&*&&&&\GL_{k_1}(\E)
\end{smallmatrix}\right)
&\text{,\ or }&
\left(\begin{smallmatrix}
\GL_{k_1}(\E) &*  &&&&&\\
*&\ddots&&&&&\\
&&\GL_{k_r}(\E)&&&&\\
&&& \G'&&&\\
&&&&\GL_{k_r}(\E)&*&\\
&&&&*&\ddots&\\
&&&&&&\GL_{k_1}(\E)
\end{smallmatrix}\right) &\text{,\ or }&
\end{align*}
\begin{align*}
\left(\begin{smallmatrix}
\GL_{k_1}(\E) &&&&&&\\
&\ddots&&*&&&\\
&&\GL_{k_l}(\E)&&&&\\
&*&& \G'&&*&\\
&&&&\GL_{k_l}(\E)&&\\
&&&*&&\ddots&\\
&&&&&&\GL_{k_1}(\E)
\end{smallmatrix}\right).
\end{align*}
The first kind of $\M_{\alpha}$ has $\GL_{k_i+k_j}(\E)$ as a direct factor for some $i$ and $j$, and the terms $j^{\M_{\alpha}}(\sigma_0)$ in Remark~\ref{rmk:factorization}(1) are intertwining operators on the induced representations of the form
$$i_{\P_{\alpha}}^{\M_{\alpha}}(\sigma_0\chi_{\text{univ},\M}) = \left(i_{\P''}^{\GL_{k_i+k_j}(\E)}\left((\tau_i)_{s_i}\otimes ((\tau_j)_{s_j}^c)^{\vee}\right)\right)\otimes (\tau_1)_{s_1}\otimes\cdots\widehat{i}\cdots\widehat{j}\cdots \otimes (\tau_r)_{s_r}\otimes \pi,$$ where $\P_{\alpha}$ is the upper parabolic in $\M_{\alpha}$ with Levi $\M$, and $\P''$ is the upper parabolic of $\GL_{k_i+k_i}(\E)$ with Levi $\GL_{k_i}(\E)\times \GL_{k_j}(\E)$, and $\widehat{(\cdot)}$ indicates factors that are removed from the product. The second kind of $M_{\alpha}$ gives rise to an intertwining operator on the induced module $$i_{\P_{\alpha}}^{\M_{\alpha}}(\sigma_0\chi_{\text{univ},\M}) = \left(i_{\P''}^{\GL_{k_i+k_j}(\E)}\left((\tau_i)_{s_i}\otimes (\tau_j)_{s_j}\right)\right)\otimes(\tau_1)_{s_1}\otimes\cdots\widehat{i}\cdots\widehat{j}\cdots \otimes (\tau_r)_{s_r}\otimes \pi.$$ The intertwining operators $j^{M_{\alpha}}(\sigma_0)$ in these two cases give the factors 
\begin{align*}
\mu^{\GL_{k_i+k_j}(\E)}\left((\tau_i)_{s_i}\otimes ((\tau_j)_{s_j}^c)^{\vee}\right)&\text{\ ,\ and\ } \mu^{\GL_{k_i+k_j}(\E)}\left((\tau_i)_{s_i}\otimes (\tau_j)_{s_j}\right),
\end{align*} respectively, in the statement of the corollary.

For the third kind of $\M_{\alpha}$, the factors $j^{\M_{\alpha}}(\sigma_0)$ in Remark~\ref{rmk:factorization}(1) are intertwining operators on the induced representations 
$$i_{\P_{\alpha}}^{\M_{\alpha}}(\sigma_0) = \left(i_{\P''}^{\G''}\left((\tau_i)_{s_i}\otimes \pi\right)\right)\otimes (\tau_1)_{s_1}\otimes\cdots\widehat{i}\cdots \otimes (\tau_r)_{s_r},$$ where $\G''$ is the classical group appearing as a direct factor of $\M_{\alpha}$ and $\P''$ is the upper parabolic of $\GL_{k_i}(\E)\times \G'$ in $\G''$. Again, the factors outside the induction do not contribute, and the inverse of the factor $\mu^{\M_{\alpha}}(\sigma_0)$ is simply $\mu^{\G''}((\tau_i)_{s_i}\otimes \pi)$.
\end{proof}

Finally, we turn to proving the multiplicativity property of Proposition~\ref{multofplancherel}. Recall our convention from Subsection~\ref{subsection:plancherel_body} that for $\tau$, $\tau'$ irreducible representations of $\GL_k(\E)$ and $\GL_{k'}(E)$, respectively, $\mu(\tau_s\otimes \tau'_t)$ denotes the image of the Plancherel measure $\mu_{\psi_E}^{\GL_{k+k'}}(\tau\otimes \tau')$ in $\CC(q^{-s},q^{-t})$. By the factorization property, it lands in $\CC(q^{-s}q^t)$, and the specializations
\begin{align*}
\mu(\tau_s\otimes \tau'):=\mu(\tau_s,\tau_t)|_{q^{-t}=1}&\text{\ ,\ and\ } \mu(\tau_s,\tau'_{-s}):=\mu(\tau_s\otimes\tau'_t)|_{q^{-t} = q^s}
\end{align*} 
are defined.

\begin{proof}[Proof of Proposition~\ref{multofplancherel}]
To prove statement (\ref{multclassfactor}) of Proposition~\ref{multofplancherel} we apply Corollary~\ref{cor:multiplicativity} with $\M=\GL_{n_1}(\E)\times \cdots \times \GL_{n_r}(\E)\times M_1$ and $\N = \GL_k \times \G$, and let $f:k[\M/\M^{\circ}]\to k[\N/\N^{\circ}]$ be the homomorphism induced from $\M\subset \N$. As above, we also use $f$ to denote the extension of $f$ to a map on localizations $$f:k[\M/\M^{\circ}][1/b] \to k[\N/\N^{\circ}][1/f(b)],$$ where $b$ is a singular element for $i_{\P}^\G(\rho)$ such that $f(b)\neq 0$ (note that Lemma~\ref{hard_lemma} is used here). More precisely, by computations similar to those in the proof of Corollary~\ref{cor:j_factorization_classical}, we deduce from Corollary~\ref{cor:multiplicativity} that $$j^G(\tau\otimes \pi) = f\left(\left(\prod_{i=1}^rj^{\GL_{k+n_i}}(\tau\otimes \rho_i)j^{\GL_{k+n_i}}(\tau\otimes ({\rho_i}^c)^{\vee})\right)j(\tau\otimes\rho')\right).$$

After making the identifications $$k[(\GL_k\times \GL_{n_i})/(\GL_k\times \GL_{n_i})^{\circ}] \simeq k[(q^{-s})^{\pm 1}, (q^{-t_i})^{\pm1}]$$
$$k[(\GL_k\times \G)/(\GL_k\times \G)^{\circ}] \simeq k[(\GL_k\times \G')/(\GL_k\times \G')^{\circ}] \simeq k[\GL_k/\GL_k^{\circ}]\simeq k[(q^{-s})^{\pm1}],$$ we find that the map $f$ specializes $q^{-t_i}\mapsto 1$. Now, inverting the factors to $\mu$ instead of $j$ and adding the variables $q^{-s}$ and $q^{-t_i}$ to the notation, our expression becomes
$$\mu(\tau_s\otimes \pi) = \left(\prod_{i=1}^s\mu^{\GL_{k+n_i}}(\tau_s\otimes \rho_i')\mu^{\GL_{k+n_i}}(\tau\otimes ({\tau_i'}^c)^{\vee})\right)\mu^{\GL_k\times \G'}(\tau_s\otimes\pi'),$$ as claimed.

For statement (\ref{multGLnfactor}) one applies Corollary~\ref{cor:multiplicativity_irreducible} to the pair
$$(\M,\sigma_0) = (\GL_{k_1}(\E)\times\cdots\times \GL_{k_l}(\E)\times \G'\ ,\ \tau_1'\otimes\cdots\otimes \tau_l'\otimes \pi)$$
and, if $\P$ is the upper parabolic of $\G$ with Levi $\M$, we have
$$(\N, i_{\P\cap \N}^\N(\sigma_0)) = (\GL_m(\E)\times \G'\ ,\  i_{\P'}^{\GL_m(\E)}(\tau_1'\otimes\cdots\otimes \tau_l')\otimes \pi)$$
where $m=k_1+\cdots+k_l$. In the notation of Corollary~\ref{cor:multiplicativity_irreducible} and its surrounding discussion, $\beta$ runs over the roots of $\A_\M$ in both $\text{Lie}(\P)$ and $\text{Lie}(\overline{\Q})$, where $\P$ and $\overline{\Q}$ are parabolic subgroups of $\G$ of the following form, respectively:
 
\begin{align*}
\left(\begin{smallmatrix}
\GL_{k_1}(\E) & * &*&*&*&*&*\\
&\ddots&*&*&*&*&*\\
&&\GL_{k_l}(\E)&*&*&*&*\\
&&& \G'&*&*&*\\
&&&&\GL_{k_l}(\E)&*&*\\
&&&&&\ddots&*\\
&&&&&&\GL_{k_1}(\E)
\end{smallmatrix}\right)\ ,\ 
&&
\left(\begin{smallmatrix}
\GL_{k_1}(\E) & &&*&*&*&*\\
*&\ddots&&*&*&*&*\\
*&*&\GL_{k_l}(\E)&*&*&*&*\\
&&& \G'&*&*&*\\
&&&&\GL_{k_l}(\E)&&\\
&&&&*&\ddots&\\
&&&&*&*&\GL_{k_1}(\E)
\end{smallmatrix}\right).
\end{align*}
(recall that, in the classical group $\G$, the elements in the $\GL_{k_i}(\E)$ factors in the bottom right are the images of their counterparts in the top left under $^t((\cdot)^c)^{-1}$). The Levi subgroups $\M_{\beta}$ in Corollary~\ref{cor:multiplicativity_irreducible} come in two shapes:

\begin{align*}
\left(\begin{smallmatrix}
\GL_{k_1}(\E) &  &&&*&&\\
&\ddots&&&&&\\
&&\GL_{k_l}(\E)&&&&*\\
&&& \G'&&&\\
*&&&&\GL_{k_l}(\E)&&\\
&&&&&\ddots&\\
&&*&&&&\GL_{k_1}(\E)
\end{smallmatrix}\right)
&\text{ , or }&
\left(\begin{smallmatrix}
\GL_{k_1}(\E) &&&&&&\\
&\ddots&&*&&&\\
&&\GL_{k_l}(\E)&&&&\\
&*&& \G'&&*&\\
&&&&\GL_{k_l}(\E)&&\\
&&&*&&\ddots&\\
&&&&&&\GL_{k_1}(\E)
\end{smallmatrix}\right).
\end{align*}
The first kind of $\M_{\beta}$ has $\GL_{k_i+k_j}(\E)$ as a direct factor for some $i$ and $j$, and the terms $j^{\M_{\beta}}(\sigma_0)$ in Corollary~\ref{cor:multiplicativity_irreducible} are intertwining operators on the induced representations 
$$i_{\P_{\beta}}^{\M_{\beta}}(\sigma_0\chi_{\text{univ},\M}) = \left(i_{\P''}^{\GL_{k_i+k_j}(\E)}\left((\tau_i')_{s_i}\otimes (((\tau_j')_{s_j})^c)^{\vee}\right)\right)\otimes (\tau_1')_{s_1}\otimes\cdots\widehat{i}\cdots\widehat{j}\cdots \otimes (\tau_l')_{s_l}\otimes \pi,$$ where $\P_{\beta}$ is the upper parabolic in $\M_{\beta}$ with Levi $\M$, and $\P''$ is the upper parabolic of $\GL_{k_i+k_i}(\E)$ with Levi $\GL_{k_i}(\E)\times \GL_{k_j}(\E)$, and $\widehat{(\cdot)}$ indicates factors that are removed from the product. Since the intertwining operators act as the identity on the factors outside the induction, we have $j^{\M_{\beta}}(\sigma_0) = j^{\GL_{k_i+k_j}(\E)}\left(\tau_i'\otimes ((\tau_j')^c)^{\vee}\right)$. Finally, if we identify $\CC[\M/\M^\circ]\cong \CC[q^{\pm s_1},\dots,q^{\pm s_l}]$ and $\CC[\N/\N^\circ]\cong \CC[q^{\pm s}]$, the map $f:\CC[\M/\M^\circ]\to \CC[\N/\N^\circ]$ induced by the inclusion $\M\subset \N$ sends both $q^{-s_i}$ and $q^{-s_j}$ to $q^{-s}$, and since $j=\mu^{-1}$ we obtain the first collection of factors of the product in statement (\ref{multGLnfactor}).

For the second kind of $\M_{\beta}$, the factors $j^{\M_{\beta}}(\sigma_0)$ are intertwining operators on the induced representations 
$$i_{\P_{\beta}}^{\M_{\beta}}(\sigma_0) = \left(i_{\P''}^{\G''}\left((\tau_i')_{s_i}\otimes \pi\right)\right)\otimes (\tau_1')_{s_1}\otimes\cdots\widehat{i}\cdots \otimes (\tau_l')_{s_l},$$ where $\G''$ is the classical group appearing as a direct factor of $\M_{\beta}$ and $\P''$ is the upper parabolic of $\GL_{k_i}(\E)\times \G'$ in $\G''$. Again, the factors outside the induction do not contribute, the factor $j^{\M_{\beta}}(\sigma_0)$ is simply $j^{\G''}(\tau_i'\otimes \pi)$, and the map $f$ takes this factor to $\mu((\tau_j')_s\otimes \pi)$, as desired.

For statement (\ref{multGLnmeasure}) one applies Corollary~\ref{cor:multiplicativity_irreducible} to the pair 
\[(\M,\sigma_0) = (\GL_{k_1}(\E)\times\cdots\times \GL_{k_r}(\E)\times\GL_{k_1'}(\E)\times\cdots\times \GL_{k_l'}(\E), \tau_1\otimes\cdots\otimes\tau_r\otimes\tau_1'\otimes\cdots\otimes\tau_l'),\]
and 
\[(\N,i_{\P\cap \N}^\N(\sigma_0)) = (\GL_m(\E)\times\GL_n(\E)\ ,\ i_{\P'}^{\GL_m}(\tau_1\otimes\cdots\otimes \tau_r)\otimes i_{\P''}^{\GL_n(\E)}(\tau_1'\otimes\cdots \otimes \tau_l')),\]
where $\P$ is the upper parabolic of $\GL_{m+n}(\E)$ containing $\N$ and $\P\cap \N = \P'\times \P''$. Identify $\CC[\M/\M^{\circ}]\simeq \CC[q^{\pm s_i},q^{\pm t_j}]_{\substack{1\leqslant i\leqslant r\\ 1\leqslant j\leqslant l}}$ and $\CC[\N/\N^{\circ}] \simeq \mathbb{C}[q^{\pm s},q^{\pm t}]$. The map $f$ induced by $\M\subset \N$ sends $q^{-s_i}$ to $q^{-s}$ and $q^{-t_j}$ to $q^{-t}$. Since $j^{\M_{\beta}}(\sigma_0)$ is insensitive to the Weyl action, each $\M_{\beta}$ in the statement of Corollary~\ref{cor:multiplicativity_irreducible} can be taken to be the Levi subgroup obtained by replacing the~$\GL_{k_i}(\E)\times\GL_{k_j'}(\E)$ factor in $\M$ with~$\GL_{k_i+k_j'}(\E)$, so the desired statement follows from the relation~$\mu = j^{-1}$.
\end{proof}

\bibliographystyle{alpha}
\bibliography{conjecture.bib}

\newcommand{\etalchar}[1]{$^{#1}$}
\begin{thebibliography}{BZCHN24}

\bibitem[AMS18]{MR3845761}
Anne-Marie Aubert, Ahmed Moussaoui, and Maarten Solleveld.
\newblock Generalizations of the {S}pringer correspondence and cuspidal
  {L}anglands parameters.
\newblock {\em Manuscripta Math.}, 157(1-2):121--192, 2018.

\bibitem[AMS22]{aubert2022}
Anne-Marie Aubert, Ahmed Moussaoui, and Maarten Solleveld.
\newblock Affine hecke algebras for classical $p$-adic groups.
\newblock arXiv:2211.08196, 2022.

\bibitem[Art13]{Arthur}
James Arthur.
\newblock {\em The endoscopic classification of representations}, volume~61 of
  {\em American Mathematical Society Colloquium Publications}.
\newblock American Mathematical Society, Providence, RI, 2013.
\newblock Orthogonal and symplectic groups.

\bibitem[Ato17]{Atobe}
Hiraku Atobe.
\newblock On the uniqueness of generic representations in an {$L$}-packet.
\newblock {\em Int. Math. Res. Not. IMRN}, (23):7051--7068, 2017.

\bibitem[Ber84]{BD}
J.~N. Bernstein.
\newblock Le ``centre'' de {B}ernstein.
\newblock In {\em Representations of reductive groups over a local field},
  Travaux en Cours, pages 1--32. Hermann, Paris, 1984.
\newblock Edited by P. Deligne.

\bibitem[BH00]{BHDavenport}
Colin~J. Bushnell and Guy Henniart.
\newblock Davenport-{H}asse relations and an explicit {L}anglands
  correspondence. {II}. {T}wisting conjectures.
\newblock volume~12, pages 309--347. 2000.
\newblock Colloque International de Th\'{e}orie des Nombres (Talence, 1999).

\bibitem[BH03]{BHGen}
Colin~J. Bushnell and Guy Henniart.
\newblock Generalized {W}hittaker models and the {B}ernstein center.
\newblock {\em Amer. J. Math.}, 125(3):513--547, 2003.

\bibitem[BH06]{BH06}
Colin~J. Bushnell and Guy Henniart.
\newblock {\em The local {L}anglands conjecture for {$\rm GL(2)$}}, volume 335
  of {\em Grundlehren der Mathematischen Wissenschaften [Fundamental Principles
  of Mathematical Sciences]}.
\newblock Springer-Verlag, Berlin, 2006.

\bibitem[BMHN22]{BMHN}
Alexander Bertoloni~Meli, Linus Hamann, and Kieu~Hieu Nguyen.
\newblock Compatibility of the {F}argues--{S}cholze correspondence for unitary
  groups, ar{X}iv:2207.13193.
\newblock 2022.

\bibitem[BMIY24]{MIY}
Alexander Bertoloni~Meli, Naoki Imai, and Alex Youcis.
\newblock The {J}acobson-{M}orozov morphism for {L}anglands parameters in the
  relative setting.
\newblock {\em Int. Math. Res. Not. IMRN}, (6):5100--5165, 2024.

\bibitem[Bor79]{Borel}
A.~Borel.
\newblock Automorphic {$L$}-functions.
\newblock In {\em Automorphic forms, representations and {$L$}-functions
  ({P}roc. {S}ympos. {P}ure {M}ath., {O}regon {S}tate {U}niv., {C}orvallis,
  {O}re., 1977), {P}art 2}, Proc. Sympos. Pure Math., XXXIII, pages 27--61.
  Amer. Math. Soc., Providence, R.I., 1979.

\bibitem[BZCHN24]{BZCHN}
David Ben-Zvi, Harrison Chen, David Helm, and David Nadler.
\newblock Coherent {S}pringer theory and the categorical {D}eligne-{L}anglands
  correspondence.
\newblock {\em Invent. Math.}, 235(2):255--344, 2024.

\bibitem[CDFZ24]{CDFZ24}
Clifton Cunningham, Sarah Dijols, Andrew Fiori, and Qing Zhang.
\newblock Generic representations, open parameters and {A}{B}{V}-packets for
  p-adic groups.
\newblock {\em arXiv:2404.07463}, 2024.

\bibitem[CFM{\etalchar{+}}22]{CFMMX}
Clifton L.~R. Cunningham, Andrew Fiori, Ahmed Moussaoui, James Mracek, and Bin
  Xu.
\newblock Arthur packets for {$p$}-adic groups by way of microlocal vanishing
  cycles of perverse sheaves, with examples.
\newblock {\em Mem. Amer. Math. Soc.}, 276(1353):ix+216, 2022.

\bibitem[CS19]{ChanSavin}
Kei~Yuen Chan and Gordan Savin.
\newblock Bernstein-{Z}elevinsky derivatives: a {H}ecke algebra approach.
\newblock {\em Int. Math. Res. Not. IMRN}, (3):731--760, 2019.

\bibitem[CZ21a]{ChenZou1}
Rui Chen and Jialiang Zou.
\newblock Local {L}anglands correspondence for even orthogonal groups via theta
  lifts.
\newblock {\em Selecta Math. (N.S.)}, 27(5):Paper No. 88, 71, 2021.

\bibitem[CZ21b]{ChenZou2}
Rui Chen and Jialiang Zou.
\newblock Local {L}anglands correspondence for unitary groups via theta lifts.
\newblock {\em Represent. Theory}, 25:861--896, 2021.

\bibitem[Dat05]{datnu}
Jean-Fran\c{c}ois Dat.
\newblock {$v$}-tempered representations of {$p$}-adic groups. {I}. {$l$}-adic
  case.
\newblock {\em Duke Math. J.}, 126(3):397--469, 2005.

\bibitem[Dat09]{datfinitude}
Jean-Francois Dat.
\newblock Finitude pour les repr\'{e}sentations lisses de groupes
  {$p$}-adiques.
\newblock {\em J. Inst. Math. Jussieu}, 8(2):261--333, 2009.

\bibitem[Dat18]{DatDudas}
Jean-Fran\c{c}ois Dat.
\newblock Simple subquotients of big parabolically induced representations of
  {$p$}-adic groups.
\newblock {\em J. Algebra}, 510:499--507, 2018.

\bibitem[Del73]{Deligne}
P.~Deligne.
\newblock Les constantes des \'{e}quations fonctionnelles des fonctions {$L$}.
\newblock In {\em Modular functions of one variable, {II} ({P}roc. {I}nternat.
  {S}ummer {S}chool, {U}niv. {A}ntwerp, {A}ntwerp, 1972)}, pages 501--597.
  Lecture Notes in Math., Vol. 349, 1973.

\bibitem[DHKM20]{DHKM}
Jean-Fran\c{c}ois Dat, David Helm, Robert Kurinczuk, and Gil Moss.
\newblock Moduli of {L}anglands parameters, ar{X}iv:2009.06708.
\newblock 2020.

\bibitem[DHKM24]{DHKMfiniteness}
Jean-Fran\c{c}ois Dat, David Helm, Robert Kurinczuk, and Gil Moss.
\newblock Finiteness for {H}ecke algebras of $p$-adic groups.
\newblock {\em J. Amer. Math. Soc.}, 37(3):929--949, 2024.

\bibitem[EH14]{EH14}
Matthew Emerton and David Helm.
\newblock The local {L}anglands correspondence for {${\rm GL}_n$} in families.
\newblock {\em Ann. Sci. \'{E}c. Norm. Sup\'{e}r. (4)}, 47(4), 2014.

\bibitem[Fin21]{Fintzen}
Jessica Fintzen.
\newblock Types for tame {$p$}-adic groups.
\newblock {\em Ann. of Math. (2)}, 193(1):303--346, 2021.

\bibitem[FS20]{FarguesScholze}
Laurent Fargues and Peter Scholze.
\newblock Geometrization of the local {L}anglands correspondence,
  ar{X}iv:2102.13459.
\newblock 2020.

\bibitem[GGP12]{GGP12}
Wee~Teck Gan, Benedict~H. Gross, and Dipendra Prasad.
\newblock Symplectic local root numbers, central critical {$L$} values, and
  restriction problems in the representation theory of classical groups.
\newblock Number 346, pages 1--109. 2012.
\newblock Sur les conjectures de Gross et Prasad. I.

\bibitem[GI14]{GanIchino14}
Wee~Teck Gan and Atsushi Ichino.
\newblock Formal degrees and local theta correspondence.
\newblock {\em Invent. Math.}, 195(3):509--672, 2014.

\bibitem[GI16]{GanIchino16}
Wee~Teck Gan and Atsushi Ichino.
\newblock The {G}ross-{P}rasad conjecture and local theta correspondence.
\newblock {\em Invent. Math.}, 206(3):705--799, 2016.

\bibitem[Goo72]{Goodearl}
K.~R. Goodearl.
\newblock Distributing tensor product over direct product.
\newblock {\em Pacific J. Math.}, 43:107--110, 1972.

\bibitem[GS12]{gan_savin}
Wee~Teck Gan and Gordan Savin.
\newblock Representations of metaplectic groups {I}: epsilon dichotomy and
  local {L}anglands correspondence.
\newblock {\em Compos. Math.}, 148(6):1655--1694, 2012.

\bibitem[Hai14]{Haines}
Thomas~J. Haines.
\newblock The stable {B}ernstein center and test functions for {S}himura
  varieties.
\newblock In {\em Automorphic forms and {G}alois representations. {V}ol. 2},
  volume 415 of {\em London Math. Soc. Lecture Note Ser.}, pages 118--186.
  Cambridge Univ. Press, Cambridge, 2014.

\bibitem[Ham21]{Hamann}
Linus Hamann.
\newblock Compatibility of the {F}argues--{S}cholze and {G}an-{T}akeda {L}ocal
  {L}anglands, ar{X}iv:2109.01210.
\newblock 2021.

\bibitem[Han22]{Hansen}
David Hansen.
\newblock Notes on homological properties of the {W}hittaker representation.
\newblock 2022.

\bibitem[Hel16]{HelmDuke}
David Helm.
\newblock Whittaker models and the integral {B}ernstein center for {${\rm
  GL}_n$}.
\newblock {\em Duke Math. J.}, 165(9):1597--1628, 2016.

\bibitem[Hel20]{curtis}
David Helm.
\newblock Curtis homomorphisms and the integral {B}ernstein center for
  {GL{$_n$}}.
\newblock {\em Algebra Number Theory}, 14(10):2607--2645, 2020.

\bibitem[Hel23]{Hellmann}
Eugen Hellmann.
\newblock On the derived category of the {I}wahori-{H}ecke algebra.
\newblock {\em Compos. Math.}, 159(5):1042--1110, 2023.

\bibitem[Hen00]{Henniart}
Guy Henniart.
\newblock Une preuve simple des conjectures de {L}anglands pour {${\mathrm{
  GL}}(n)$} sur un corps {$p$}-adique.
\newblock {\em Invent. Math.}, 139(2):439--455, 2000.

\bibitem[HKW22]{HKWKott}
David Hansen, Tasho Kaletha, and Jared Weinstein.
\newblock On the {K}ottwitz conjecture for local shtuka spaces.
\newblock {\em Forum Math. Pi}, 10:Paper No. e13, 79, 2022.

\bibitem[HM17]{HelmMossGamma}
David Helm and Gil Moss.
\newblock Deligne--{L}anglands gamma factors in families, ar{X}iv:1510.08743.
\newblock 2017.

\bibitem[HM18]{HM}
David Helm and Gilbert Moss.
\newblock Converse theorems and the local {L}anglands correspondence in
  families.
\newblock {\em Invent. Math.}, 214(2):999--1022, 2018.

\bibitem[HO13]{HeiermannOpdam}
Volker Heiermann and Eric Opdam.
\newblock On the tempered {$L$}-functions conjecture.
\newblock {\em Amer. J. Math.}, 135(3):777--799, 2013.

\bibitem[HT01]{HarrisTaylor}
Michael Harris and Richard Taylor.
\newblock {\em The geometry and cohomology of some simple {S}himura varieties},
  volume 151 of {\em Annals of Mathematics Studies}.
\newblock Princeton University Press, Princeton, NJ, 2001.
\newblock With an appendix by Vladimir G. Berkovich.

\bibitem[Ish24]{MR4776199}
Hiroshi Ishimoto.
\newblock The {E}ndoscopic {C}lassification of {R}epresentations of
  {N}on-{Q}uasi-{S}plit {O}dd {S}pecial {O}rthogonal {G}roups.
\newblock {\em Int. Math. Res. Not. IMRN}, (14):10939--11012, 2024.

\bibitem[Kak21]{Kakuhama}
Hirotaka Kakuhama.
\newblock Formal degrees and local theta correspondence: quaternionic case, v4
  ar{X}iv:2012.04219.
\newblock 2021.

\bibitem[KMSW21]{KMSW}
Tasho Kaletha, Alberto Minguez, Sug~Woo Shin, and Paul-James White.
\newblock Endoscopic {C}lassification of {R}epresentations: {I}nner {F}orms of
  {U}nitary {G}roups, ar{X}iv:1409.3731.
\newblock 2021.

\bibitem[Kur19]{KOberwolfach}
Robert Kurinczuk.
\newblock Local {L}anglands in families in depth zero (joint with
  {J}ean-{F}ran\c{c}ois {D}at, {D}avid {H}elm, {G}il {M}oss).
\newblock In {\em {F}intzen {J}essica, Gan {W}ee-{T}eck, {T}akeda {S}huichiro:
  {N}ew {D}evelopments in {R}epresentation {T}heory of $p$-adic {G}roups},
  Oberwolfach Rep.~16, pages 2763--2766. 2019.

\bibitem[Lan76]{MR0579181}
Robert~P. Langlands.
\newblock {\em On the functional equations satisfied by {E}isenstein series}.
\newblock Lecture Notes in Mathematics, Vol. 544. Springer-Verlag, Berlin-New
  York, 1976.

\bibitem[Lan89]{Langlands}
R.~P. Langlands.
\newblock On the classification of irreducible representations of real
  algebraic groups.
\newblock In {\em Representation theory and harmonic analysis on semisimple
  {L}ie groups}, volume~31 of {\em Math. Surveys Monogr.}, pages 101--170.
  Amer. Math. Soc., Providence, RI, 1989.

\bibitem[Mey15]{Meyer}
Ralf Meyer.
\newblock Cuspidal representations of reductive {$p$}-adic groups are
  relatively injective and projective.
\newblock {\em Represent. Theory}, 19:290--298, 2015.

\bibitem[Mok15]{Mok}
Chung~Pang Mok.
\newblock Endoscopic classification of representations of quasi-split unitary
  groups.
\newblock {\em Mem. Amer. Math. Soc.}, 235(1108):vi+248, 2015.

\bibitem[Mou17]{Moussaoui}
Ahmed Moussaoui.
\newblock Centre de {B}ernstein dual pour les groupes classiques.
\newblock {\em Represent. Theory}, 21:172--246, 2017.

\bibitem[MR18]{MoeglinRenard}
Colette Moeglin and David Renard.
\newblock Sur les paquets d'{A}rthur des groupes classiques et unitaires non
  quasi-d\'{e}ploy\'{e}s.
\newblock In {\em Relative aspects in representation theory, {L}anglands
  functoriality and automorphic forms}, volume 2221 of {\em Lecture Notes in
  Math.}, pages 341--361. Springer, Cham, 2018.

\bibitem[PR04]{MR2067618}
Gena Puninski and Philipp Rothmaler.
\newblock When every finitely generated flat module is projective.
\newblock {\em J. Algebra}, 277(2):542--558, 2004.

\bibitem[Sch13]{Scholze}
Peter Scholze.
\newblock The local {L}anglands correspondence for {$\mathrm{GL}_n$} over
  {$p$}-adic fields.
\newblock {\em Invent. Math.}, 192(3):663--715, 2013.

\bibitem[Sha90]{Shahidi}
Freydoon Shahidi.
\newblock A proof of {L}anglands' conjecture on {P}lancherel measures;
  complementary series for {$p$}-adic groups.
\newblock {\em Ann. of Math. (2)}, 132(2):273--330, 1990.

\bibitem[{Sta}18]{stacks-project}
The {Stacks Project Authors}.
\newblock \textit{Stacks Project}.
\newblock \url{https://stacks.math.columbia.edu}, 2018.

\bibitem[Ste08]{St08}
Shaun Stevens.
\newblock The supercuspidal representations of {$p$}-adic classical groups.
\newblock {\em Invent. Math.}, 172(2):289--352, 2008.

\bibitem[Vig96]{Vig96}
Marie-France Vign\'{e}ras.
\newblock {\em Repr\'{e}sentations {$l$}-modulaires d'un groupe r\'{e}ductif
  {$p$}-adique avec {$l\ne p$}}, volume 137 of {\em Progress in Mathematics}.
\newblock Birkh\"{a}user Boston, Inc., Boston, MA, 1996.

\bibitem[Vig97]{VigCohomologyofsheavesonthebuilding}
Marie-France Vign\'{e}ras.
\newblock Cohomology of sheaves on the building and {$R$}-representations.
\newblock {\em Invent. Math.}, 127(2):349--373, 1997.

\bibitem[Wal03]{Waldspurger}
J.-L. Waldspurger.
\newblock La formule de {P}lancherel pour les groupes {$p$}-adiques (d'apr\`es
  {H}arish-{C}handra).
\newblock {\em J. Inst. Math. Jussieu}, 2(2):235--333, 2003.

\bibitem[Xu17]{BinXu}
Bin Xu.
\newblock On the cuspidal support of discrete series for {$p$}-adic quasisplit
  {$Sp(N)$} and {$SO(N)$}.
\newblock {\em Manuscripta Math.}, 154(3-4):441--502, 2017.

\bibitem[Zel80]{Zelevinsky}
A.~V. Zelevinsky.
\newblock Induced representations of reductive {$p$}-adic groups. {II}. {O}n
  irreducible representations of {${\rm GL}(n)$}.
\newblock {\em Ann. Sci. \'{E}cole Norm. Sup. (4)}, 13(2):165--210, 1980.

\bibitem[Zhu20]{Zhu}
Xinwen Zhu.
\newblock Coherent sheaves on the stack of {L}anglands parameters,
  ar{X}iv:2008.02998.
\newblock 2020.

\end{thebibliography}

\end{document}